%% file: Spencer_Breiner_-_Scheme_representation_for_first-order_logic.tex
\documentclass[letterpaper,10pt]{report}
\usepackage[table]{xcolor}
\usepackage{booktabs}
\usepackage{amsthm, amsfonts, amssymb, graphicx, amsmath,multicol, multirow, verbatim, url}
\usepackage[all,2cell, cmtip]{xy} \UseAllTwocells
\usepackage{array, arydshln}
\usepackage{bbm}
\usepackage{rotating}
\title{Scheme representation for first-order logic}
\date{\today}
\author{Spencer Breiner\\ Carnegie Mellon University\\\\ \begin{tabular}{cc} Advisor:& Steve Awodey \\ Committee:& Jeremy Avigad\\& Hans Halvorson\\ &Pieter Hofstra\\\end{tabular}}

\newtheorem{thm}[subsection]{Theorem}
\newtheorem{defn}[subsection]{Definition}
\newtheorem{cor}[subsection]{Corollary}
\newtheorem{lemma}[subsection]{Lemma}
\newtheorem{prop}[subsection]{Proposition}

\newtheorem*{thm*}{Theorem}

\renewcommand{\rm}{\mathrm}
\renewcommand{\bf}{\mathbf}

\newcommand{\ol}{\overline}

\let\AA\relax
\DeclareMathOperator{\AA}{\mathcal{A}}
\DeclareMathOperator{\PP}{\mathcal{P}}
\DeclareMathOperator{\LL}{\mathcal{L}}

\DeclareMathOperator{\NN}{\mathbb{N}}
\DeclareMathOperator{\ZZ}{\mathcal{Z}}
\DeclareMathOperator{\CC}{\mathcal{C}}
\DeclareMathOperator{\bCC}{\mathbb{C}}
\DeclareMathOperator{\DD}{\mathcal{D}}

\DeclareMathOperator{\JJ}{\mathcal{J}}
\DeclareMathOperator{\EE}{\mathcal{E}}
\DeclareMathOperator{\bEE}{\mathbb{E}}

\DeclareMathOperator{\FF}{\mathcal{F}}
\DeclareMathOperator{\bFF}{\mathbb{F}}
\DeclareMathOperator{\GG}{\mathcal{G}}
\DeclareMathOperator{\bGG}{\mathbb{G}}
\DeclareMathOperator{\bHH}{\mathbb{H}}
\DeclareMathOperator{\HH}{\mathcal{H}}
\DeclareMathOperator{\KK}{\mathcal{K}}
\DeclareMathOperator{\bTT}{\mathbb{T}}
\DeclareMathOperator{\TT}{\mathcal{T}}
\DeclareMathOperator{\rTT}{\mathring{\TT}}

\DeclareMathOperator{\OO}{\mathcal{O}}
\DeclareMathOperator{\UU}{\mathcal{U}}
\DeclareMathOperator{\VV}{\mathcal{V}}
\DeclareMathOperator{\MM}{\mathcal{M}}
\DeclareMathOperator{\WW}{\mathcal{W}}
\DeclareMathOperator{\XX}{\mathcal{X}}
\DeclareMathOperator{\YY}{\mathcal{Y}}
\DeclareMathOperator{\DS}{\rm{DS}}

\renewcommand{\tilde}{\widetilde}

\let\2\relax
\DeclareMathOperator{\2}{\mathbbm{2}}

\let\SS\relax
\DeclareMathOperator{\SS}{\mathcal{S}}
\DeclareMathOperator{\BB}{\mathbf{B}}

\DeclareMathOperator{\cc}{\mathbf{c}}

\DeclareMathOperator{\Bl}{\rm{Bl}}
\DeclareMathOperator{\Aut}{\rm{Aut}}
\DeclareMathOperator{\Sets}{\mathbf{Sets}}

\DeclareMathOperator{\Sh}{\mathbf{Sh}}

\DeclareMathOperator{\EqSh}{\mathbf{Sh}_{\rm{eq}}}
\DeclareMathOperator{\Stk}{\mathbf{Stk}}

\DeclareMathOperator{\Cat}{\mathbf{Cat}}

\DeclareMathOperator{\Grp}{\mathbf{Grp}}
\DeclareMathOperator{\Grpd}{\mathbf{Grpd}}
\DeclareMathOperator{\Top}{\mathbf{Top}}

\DeclareMathOperator{\Coh}{\mathbf{Coh}}

\DeclareMathOperator{\Ptop}{\mathbf{Ptop}}
\DeclareMathOperator{\BPtp}{\mathbf{BPtp}}

\DeclareMathOperator{\Desc}{\rm{Desc}}

\DeclareMathOperator{\El}{\rm{El}}
\DeclareMathOperator{\Spec}{\rm{Spec}}
\DeclareMathOperator{\Ob}{\rm{Ob}}
\DeclareMathOperator{\Mod}{\mbox{-}\mathbf{Mod}}
\DeclareMathOperator{\LSch}{\mathbf{L-Sch}}

\DeclareMathOperator{\eq}{\rm{eq}}

\let\exp\relax\DeclareMathOperator{\exp}{\mathbf{exp}}

\newcommand{\<}{\langle}
\renewcommand{\>}{\rangle}

\newcommand{\extl}{[\![}
\newcommand{\extr}{]\!]}
\newcommand{\ext}[1]{\extl #1\extr}

\newcommand{\bigextl}{\left[\!\!\left[}
\newcommand{\bigextr}{\right]\!\!\right]}
\newcommand{\bigext}[1]{\bigextl #1\bigextr}

\DeclareMathOperator{\Hom}{\mathbf{Hom}}
\DeclareMathOperator{\Diag}{\rm{Diag}}
\DeclareMathOperator{\uHom}{\underline{\mathrm{Hom}}}
\DeclareMathOperator{\dom}{\mathrm{dom}}
\DeclareMathOperator{\cod}{\mathrm{cod}}
\DeclareMathOperator{\id}{\mathrm{id}}
\DeclareMathOperator{\Ar}{\mathrm{Ar}}
\DeclareMathOperator{\Eq}{\mathrm{Eq}}
\let\eq\relax\DeclareMathOperator{\eq}{\mathrm{eq}}

\DeclareMathOperator{\op}{\mathrm{op}}

\let \lim\relax\DeclareMathOperator{\lim}{\varprojlim}
\DeclareMathOperator{\colim}{\varinjlim}
\DeclareMathOperator{\Sub}{\rm{Sub}}

\DeclareMathOperator{\stalk}{\rm{stalk}}

\DeclareMathOperator{\im}{\rm{im}}

\DeclareMathOperator{\Iff}{\Longleftrightarrow}
\DeclareMathOperator{\reduct}{\upharpoonright}
\DeclareMathOperator{\tri}{\vartriangleright}
\DeclareMathOperator{\rsim}{\rotatebox[origin=c]{90}{$\sim$}}

\newcommand{\overtimes}[1]{\underset{#1}{\times}}
\newcommand{\overplus}[1]{\underset{#1}{\oplus}}
\newcommand{\overprove}[1]{\underset{#1}{\vdash}}

\newcommand{\name}[1]{\ulcorner\! #1 \!\urcorner}

\newcommand{\vertsim}{\rm{\rotatebox[origin=c]{90}{$\sim$}}}
\newcommand{\vertsub}{\rm{\rotatebox[origin=c]{90}{$\subseteq$}}}

\newcommand{\vertleq}{\rm{\rotatebox[origin=c]{90}{$\leq$}}}
\newcommand{\vertiso}{\rm{\rotatebox[origin=c]{270}{$\cong$}}}
\newcommand{\auto}{\lto\hspace{-.45cm}\raisebox{.225cm}{\rotatebox{270}{\scalebox{.7}{$\curvearrowleft$}}}\hspace{.35cm}}

\newcommand{\mono}{\rightarrowtail}
\newcommand{\epi}{\twoheadrightarrow}
\newcommand{\iso}{\stackrel{\sim}{\lto}}

\newcommand{\lto}{\longrightarrow}

\renewcommand{\part}{\rightharpoonup}

\newcommand{\To}{\Rightarrow}
\newcommand{\from}{\leftarrow}
\newcommand{\lfrom}{\longleftarrow}

\newcommand{\labelarrow}[1]{\stackrel{#1}{\lto}}

\newcommand{\arproj}{\ar@{|..>}}

\newcommand{\pbcorner}[1][rd]{\save*!/#1-1.2pc/#1:(-1,1)@^{|-}\restore}
\newcommand{\pbigcorner}[1][rd]{\save*!/#1-1.7pc/#1:(-1,1)@^{|-}\restore}
\newcommand{\pbrcorner}[1][ld]{\save*!/#1+-1.2pc/#1:(-1,1)@^{|-}\restore}

\newcommand{\pbccorner}[1][r]{\save*!/#1-1pc/#1:(-1,1)@^{|-}\restore}
\newcommand{\pbdcorner}[1][ru]{\save*!/#1-1.2pc/#1:(-1,1)@^{|-}\restore}
\newcommand{\pushoutcorner}[1][ul]{\save*!/#1-1.4pc/#1:(-1,1)@^{|-}\restore}

\begin{document}

\maketitle

\setcounter{tocdepth}{1}
\tableofcontents

\include{intro}

\include{dictionary}
\include{Ch1}

\include{Ch2}

\include{Ch3}

\include{Ch4}

\bibliographystyle{plain}
\bibliography{bibfile}

\end{document}

%% file: intro.tex
\chapter*{Introduction}

Although contemporary model theory has been called ``algebraic geometry minus fields'' \cite{hodges}, the formal methods of the two fields are radically different. This dissertation aims to shrink that gap by presenting a theory of ``logical schemes,'' geometric entities which relate to first-order logical theories in much the same way that algebraic schemes relate to commutative rings.
\vspace{.25cm}

Recall that the affine scheme associated with a commutative ring $R$ consists of two components: a topological space $\Spec(R)$ (the spectrum) and a sheaf of rings $\OO_R$ (the structure sheaf). Moreover, the scheme satisfies two important properties: its stalks are local rings and its global sections are isomorphic to $R$. In this work we replace $R$ by a first-order logical theory $\bTT$ (construed as a structured category) and associate it to a pair $(\Spec(\bTT),\OO_{\bTT})$, a topological spectrum and a sheaf of theories which stand in a similar relation to $\bTT$. These ``affine schemes'' allow us to import some familiar definitions and theorems from algebraic geometry.

\subsection*{Stone duality for first-order logic: the spectrum $\Spec(\bTT)$}

In the first chapter we construct $\MM=\Spec(\bTT)$, the \emph{spectrum} of a coherent first-order logical theory. $\MM$ is a topological groupoid constructed from the semantics (models and isomorphisms) of $\bTT$, and its construction can be regarded as a generalization of Stone duality for propositional logic. The construction is based upon an idea of Joyal \& Tierney \cite{JT} which was later developed by Joyal \& Moerdijk \cite{JM1} \cite{JM2}, Butz \& Moerdijk \cite{butz_thesis} and Awodey \& Forssell \cite{forssell_thesis, FOLD}.

Points of the object space $\MM_0$ are models supplemented by certain variable assignments or labellings. Satisfaction induces a topology on these points, much as in the classical Stone space construction. As in the algebraic case, the spectrum is \emph{not} a Hausdorff space. Instead the topology incorporates model theoretic information; notably, the closure of a point $M\in\MM$ (i.e., a labelled model) can be interpreted as the set of $\bTT$-model homomorphisms mapping into $M$.

Every formula $\varphi(x_1,\ldots,x_n)$ determines a ``definable sheaf'' $\ext{\varphi}$ over the spectrum. Over each model $M$, the fiber of $\ext{\varphi}$ is the definable set
$$\textrm{stalk}_M(\ext{\varphi})=\varphi^M=\{\overline{a}\in|M|^n\ |\ M\models\varphi(\overline{a})\}.$$
The space $\ext{\varphi}$ is topologized by the terms $t$ which satisfy $\varphi(t)$; each of these defines a section of the sheaf, sending $M\mapsto t^M\in\varphi^M$.

Although $\ext{\varphi}$ is nicely behaved its subsheaves, in general, are not. The problem is that these typically depend on details of the labellings which have no syntactic relevance. To ``cancel out'' this effect we appeal to $\bTT$-model isomorphisms. Specifically, we can topologize the isomorphisms between models, turning $\Spec(\bTT)$ into a topological groupoid. Each definable sheaf $\ext{\varphi}$ \emph{is} equivariant over this groupoid, which is just a fancy way of saying that an isomorphism $M\cong M'$ induces an isomorphism $\varphi^M \cong {\varphi^M}'$ for each definable set.

The pathological subsheaves, however, are not equivariant; any subsheaf $S\leq\ext{\varphi}$ which \emph{is} equivariant must be a union of definable pieces $\ext{\psi_i}$, where $\psi_i(\overline{x})\vdash\varphi(\overline{x})$. Moreover, $S$ itself is definable just in case it is compact with respect to such covers. This reflects a deeper fact: $\Spec(\bTT)$ gives a presentation of the classifying topos for $\bTT$. That is, there is a correspondence between $\bTT$-models inside a topos $\mathcal{S}$ (e.g., $\Sets$) and geometric morphisms from $\SS$ into the topos equivariant sheaves over $\Spec(\bTT)$.

Since any classical first-order theory can be regarded as a coherent theory with complements, this gives a spectral space construction for classical first-order logic. We close the chapter with an elementary presentation of this construction, and prove a few facts which are specific to this case.

\subsection*{Theories as algebras: the logic of pretoposes}

In the second chapter we recall and prove a number of well-known facts about pretoposes, most of which can be found in either the Elephant \cite{elephant} or Makkai \& Reyes \cite{MakkaiReyes}. We recall the definitions of coherent categories and pretoposes and the pretopos completion which relates them. Logically, a pretopos $\EE=\EE_{\bTT}$ can be regarded as a coherent first-order theory $\bTT$ which is extended by disjoint sums and definable quotients; formally, this corresponds to a conservative extension $\bTT\subseteq\bTT^{\eq}$ which is important in contemporary model theory.

We go on to discuss a number of constructions on pretoposes, including the quotient-conservative factorization and the machinery of slice categories and localizations. In particular, we define the elementary diagram, a logical theory associated with any $\bTT$-model, and its interpretation as a colimit of pretoposes. These are Henkin theories: they satisfy existence and disjunction properties which can be regarded as a sort of ``locality'' for theories.

We also demonstrate that the pretopos completion interacts well with complements, so that classical first-order theories may be regarded as Boolean pretoposes. Similarly, quotients and colimits of theories are well-behaved with respect to complementation, so that the entire machinery of pretoposes can be specialized to the classical case.

\subsection*{Sheaves of theories and logical schemes}

The third chapter is the heart of the disseration, where we introduce a sheaf representation for logical theories and explore the logical schemes which arise from this representation. The most familiar example of such a representation is Grothendieck's theorem that every commutative ring $R$ is isomorphic to the ring of global sections of a certain sheaf over the Zariski topology on $R$. Together with a locality condition, this is essentially the construction of an affine algebraic scheme. Later it was shown by Lambek \& Moerdijk \cite{LM}, Lambek \cite{lambek} and Awodey \cite{awodey_thesis, sheaf_rep} that toposes could also be represented as global sections of sheaves on certain (generalized) spaces. The structure sheaf $\OO_{\bTT}$ is especially close in spirit to the last example.

Formally, the structure sheaf is constructed from the codomain fibration $\EE^I\to\EE$, which can be regarded as a (pseudo-)functor $\EE^{\op}\to\Cat$ sending $A\mapsto\EE/A$; the functorial operation of a map $f:B\to A$ is given by pullback $f^*:\EE/A\to\EE/B$. The quotients ad coproducts in $\EE$ ensure that this map is a ``sheaf up to isomorphism'' over $\EE$ (i.e., a stack). In order to define $\OO_{\bTT}$ we first turn the stack into a (strict) sheaf and then pass across the equivalence $\Sh(\EE_{\bTT})\simeq\EqSh(\MM_{\bTT})$ discussed in chapter 1.

Given a model $M\in\Spec(\bTT)$, the stalk of $\OO_{\bTT}$ over $M$ is the elementary diagram discussed in chapter 2. An object in the diagram of $M$ is a defined by a triple $\<M,\varphi(\overline{x},\overline{y}),\overline{b})$ where $M$ is a model and $\overline{b}$ and $\overline{y}$ have the same arity. We can think of this triple as a definable set
$$\varphi(\overline{x},\overline{b})^M=\left\{\overline{a}\in|M|\ |\ M\models\varphi(\overline{a},\overline{b})\right\}.$$
The groupoid of isomorphisms acts on this sheaf in the obvious way: given $\alpha:M\to N$ we send $\varphi(\overline{x},\overline{b})^M\mapsto \varphi(\overline{x},\alpha(\overline{b}))^N$.

In particular, every formula $\varphi(\overline{x})$ determines a global section $\ulcorner\!\varphi\!\urcorner:M\mapsto \varphi^M$. These are stable with respect to the equivariant action and, together with formal sums and quotients, these are \emph{all} of the equivariant sections. This yields a representation theorem just as in the algebraic case:
$$\Gamma_{\eq}(\OO_{\bTT})\simeq\EE_{\bTT}.$$

The pair $(\MM_{\bTT},\OO_{\bTT})$ is an \emph{affine logical scheme}. A map of theories is an interpretation $I:\bTT\to\bTT'$ (e.g., adding an axiom or extending the language). This induces a forgetful functor $I_\flat:\Spec(\bTT')\to\Spec(\bTT)$, sending each $\bTT'$-model $N$ to the $\bTT$-model which is the interpretation of $I$ in $N$. If $I$ is a linguistic extension then $I_\flat N$ is the usual reduct of $N$ to $\mathcal{L}(\bTT)$.

Moreover, $I$ induces another morphism at the level of structure sheaves: $I^\sharp: I_\flat^*\OO_{\bTT}\to \OO_{\bTT'}.$
On fibers, this sends each $I^*N$-definable set $\varphi^{I^*N}$ to the $N$-definable set $(I\varphi)^N$ (the same set!). The pair $\<I_\flat,I^\sharp\>$ is a map of schemes.

Equivalently, we can represent $I^\sharp$ as a map $\OO_{\bTT}\to I_{\flat *}\OO_{\bTT'}$ and, since $\Gamma_{\bTT'}\circ I_{\flat *}\cong\Gamma_{\bTT}$, the global sections of $I^\sharp$ suffice to recover $I$:
$$\Gamma_{\bTT}(I^\sharp)\cong I\ :\ \bTT=\Gamma_{\bTT}(\OO_{\bTT})\longrightarrow \Gamma_{\bTT}(I_{\flat *}\OO_{\bTT'})\cong \Gamma_{\bTT'}(\OO_{\bTT'})\cong\bTT'.$$
Similarly, we can define a natural transformation of schemes, and these too can be recovered from global sections. This addresses a significant difficulty in Awodey \& Forssell's first-order logical duality \cite{FOLD}: identifying which homomorphisms between spectra originate syntactically. This problem is non-existent for schemes: without a syntactic map at the level of structure sheaves, there is no scheme morphism.

With this framework in place, algebraic geometry provides a methodology for studying this type of object. The necessary definitions to proceed from affine schemes to the general case follow the same rubric as algebraic geometry. There are analogs of locally ringed spaces, gluings (properly generalized to groupoidal spectra) and coverings by affine pieces. Importantly, the equivariant global sections functor presents (the opposite of) the 2-category of theories as a reflective subcategory of schemes. This allows us to construct limits of affine schemes using colimits of theories. This mirrors the algebraic situation, where the polynomial ring $\ZZ[x]$ represents the affine line and its coproduct $\ZZ[x,y]\cong \ZZ[x]+\ZZ[y]$ represents the plane. Via affine covers, one can use this to compute any finite 2-limits for any logical schemes.

\subsection*{Applications}

The scheme $(\MM_{\bTT},\OO_{\bTT})$ associated with a theory incorporates both the semantic and syntactic components of $\bTT$. As such, it is a nexus to study the connections between different branches of logic and other areas of mathematics. In the final chapter of the dissertation we discuss a few of these connections.

\begin{itemize}
\item \textbf{The structure sheaf as a site.\footnote{The author would like to thank Andr\'e Joyal for a helpful conversation in which he suggested the theorem presented in this section.}} $\OO_{\bTT}$ is itself a pretopos internal to the topos of equivariant sheaves over $\MM_{\bTT}$. We can regard this as a site (with the coherent topology, internalized) and consider its topos of sheaves $\Sh_{\MM_{\bTT}}(\OO_{\bTT})$. We prove that this topos classifies $\bTT$-model homomorphisms. We also show that it can be regarded as the (topos) exponential of $\Sh(\MM_{\bTT})$ by the Sierpinski topos $\Sets^I$.

\item \textbf{The structure sheaf as a universe.} Following on the results of chapters 1 and 3 we can show that every equivariant sheaf morphism $\ext{\varphi}\to\OO_{\bTT}$ corresponds to an object $E\in\EE/\varphi$. In this section we introduce an auxilliary sheaf $\El(\OO_{\bTT})\to\OO_{\bTT}$ allowing us to recover $E$ as a pullback:
$$\xymatrix{
\ext{E} \pbcorner \ar[d] \ar[rr] && \El(\OO_{\bTT}) \ar[d]\\
\ext{\varphi} \ar[rr] && \OO_{\bTT}.\\
}$$
This allows us to think of $\OO_{\bTT}$ as a universe of \emph{definably} or \emph{representably} small sets. Formally, we show that $\OO_{\bTT}$ is a coherent universe, a pretopos relativization of Streicher's notion of a universe in a topos \cite{streicher}.

\item \textbf{Isotropy.} In this section we demonstrate a tight connection between our logical schemes and a recently defined ``isotropy group'' \cite{isotropy} which is present in any topos. This allows us to interpret the isotropy group as a logical construction. We also compute the stalk of the isotropy group at a model $M$ and show that its elements can be regarded as parameter-definable automorphisms of $M$.

\item \textbf{Conceptual completeness.} In this section we reframe Makkai \& Reyes' conceptual completeness theorem \cite{MakkaiReyes} as a theorem about schemes. The original theorem says that if an interpretation $I:\bTT\to\bTT'$ induces an equivalence $I^*:\textbf{Mod}(\bTT')\iso\textbf{Mod}(\bTT)$ under reducts, then $I$ itself was already an equivalence (at the level of syntactic pretoposes). The theorem follows immediately from our scheme construction: if $I^*:\Spec(\bTT')\to\Spec(\bTT)$ is an equivalence of schemes, then its global sections define an equivalence $\bTT\simeq\bTT'$. From here we go on to unwind the Makkai \& Reyes proof, providing some insights into the ``Galois theory'' of logical schemes.

\end{itemize}

%% file: dictionary.tex
%\input{../../header/header}

%\begin{document}

\noindent{\LARGE\textbf{A Note on Terminology}}
\vspace{.3cm}

Except in the first section we work in the doctrine of pretoposes, and use the following logical and categorical terminology interchangeably. For example, we take a liberal definition of formula to include sum and quotient ``connectives''.
\vspace{.3cm}
\label{dictionary}
\noindent{
\begin{tabular}{|c|c|}
\hline
\textbf{Categorical Definition} &\textbf{Logical shorthand} \\\hline
Classifying pretopos $\EE=\EE_{\bTT}$ & Theory $\bTT$\\
Set of generating objects $\DD\subseteq\Ob(\EE)$ & Set of basic sorts $\DD$\\\hline
An object $D=D^x\in\EE$& Context, sort $x$\\
Subobject $\varphi\leq D^x$ & (Open) Formula $\varphi(x)$ \\\hline
Morphisms $\sigma,\tau:D^x\to D^y$& Provably functional relation $\sigma(x,y)$,\\
& (Open) Term $\tau(x)=y$ \\\hline
Pretopos functor $I:\EE\to\FF$ & Interpretation of $\EE$ in $\FF$\\
\hspace{1cm} $M:\EE\to\Sets$ & An $\EE$-model, $M\models\EE$\\\hline
Evaluation $M(D^x)$ & Underlying set $|M|^x$\\
\hspace{1cm} $M(\varphi\leq D^x)$& Definable set $\varphi^M$\\\hline
\end{tabular}
}

%\end{document}

%% file: Ch1.tex
\chapter{Logical spectra}

In this chapter we will construct a topological groupoid $\MM=\Spec(\bTT)$ called the \emph{spectrum} of a coherent first-order logical theory $\bTT$. We begin by defining the space of objects $\MM_0$ constructed from the models of $\bTT$. Next we consider the topos of sheaves over $\MM_0$ and, in particular, single out a class of $\bTT$-definable sheaves over this space. Then we will extend $\MM$ to a groupoid whose morphisms are $\bTT$-model isomorphisms; using this we characterize the definable sheaves as those which admit an equivariant action by the isomorphisms in $\MM$. As a corollary we obtain a topological/groupoidal characterization of definability in first-order logic. We close by considering the special case of classical first-order logic and prove a few facts specific to this case.

\section{The spectrum $\MM_0$}\label{sec_CovSp}

In this section we associate a topological spectral space $\MM_0(\bTT)$ with any coherent first-order theory $\bTT$. We build the spectrum from the semantics of $\bTT$, generalizing the Stone space construction from propositional logic. This is based upon an idea of Joyal \& Tierney \cite{JT} which was later developed by Joyal \& Moerdijk \cite{JM1} \cite{JM2}, Butz \& Moerdijk \cite{butz_thesis} \cite{BM_article} and Awodey \& Forssell \cite{forssell_thesis}. 

\begin{defn}
A (multi-sorted) \emph{coherent theory} is a triple $\<\AA,\LL,\bTT\>$ where
\begin{itemize}
\item $\AA=\{A\}$ is a set of basic sorts,
\item $\LL=\{R(\overline{x}), F(\overline{x})=y\}$ is a set of function\footnote{We write $F(\overline{x})=y$ for the function symbol because in multi-sorted logic we must specify the sort $y:B$ of the codomain of $F$ as well as its arity $\overline{x}:A_1\times\ldots\times A_n$.} and relation symbols with signatures in $\AA$ and
\item $\bTT=\left\{\varphi(\overline{x})\overprove{\overline{x}}\psi(\overline{x})\right\}$ is a set of coherent axioms, written in the language $\LL$ and involving variables $\overline{x}=\<x_1,\ldots,x_n\>$ (see below).
\end{itemize}
\end{defn}

We write $x:A$ to indicate that $x$ is a variable ranging over some sort $A\in\AA$. A \emph{context of variables} is a sequence of sorted variables, written ${\<x_1,\ldots,x_n\>:A_1\times\ldots\times A_n}$ or $\overline{x}:\prod_i A_i$. We will assume that $\AA$ is closed under tupleing, although this is not necessary for our arguments; in practice, this allows us to avoid subscript overload by focusing most of the time on the single-variable case $x:A$.

As usual in categorical logic, we will assume that every $\LL$-formula (in particular each basic function or relation) is written in a fixed context of variables $x:A$, indicated notationally by writing $\varphi(x)$. We may always weaken a formula $\varphi(x)$ to include dummy variables $y:B$ by conjunction with a tautology; the weakening $\varphi(x,y)$ is shorthand for the formula $\varphi(x)\wedge(y=y)$. It is important to note that we regard $\varphi(x)$ and $\varphi(x,y)$ as different formulas.

A first-order formula is \emph{coherent} if belongs to the fragment generated by $\{\bot,\top,=,\wedge,\vee,\exists\}$. Axioms in a coherent theory are presented as \emph{sequents}, ordered pairs of coherent formulas written in a common context $x:A$ and usually written $\varphi(x)\vdash_x\psi(x)$. 

Relative to this data, a $\bTT$-model $M$ is defined as usual. Each basic sort $A\in\AA$ defines an underlying set $A^M$; when $A=A_1\times\ldots\times A_n$ is a compound context $A^M=A_1^M\times\ldots\times A_n^M$. If $x:A$ is a context of variables, we may also write $|M|^x$ for the ``underlying set'' $A^M$. When $\LL$ is single-sorted and $x=\<x_1,\ldots,x_n\>$, this agrees with the usual notaion $|M|^x=|M|^{n}$.

For any coherent formula $\varphi(x)$ we may construct a \emph{definable set} $\varphi^M\subseteq |M|^x$ in the usual inductive fashion. $R^M\subseteq |M|^x$ is already defined and $(F(x)=y)^M$ is the graph of the function $F^M:|M|^x\to|M|^y$. If $\varphi^M\subseteq|M|^x$ is already defined, the interpretation of a weakened formula is given by $\varphi(x,y):=\varphi^M\times|M|^y$.

The interpretation of an existential formula is given by the image of the projection:
$$\big(\exists y.\varphi(x,y)\big)^M=\pi_x(\varphi^M)=\big\{a\in|M|^x\ |\ \<a,b\>\in\varphi^M\rm{\ for\ some\ }b\in|M|^y\big\}\\$$
To interpret conjunctions and disjunctions we first weaken the component formulas to a common context $x:A$ and then compute
$$\begin{array}{c}
\big(\varphi\wedge\psi\big)^M=\varphi^M\cap\psi^M\subseteq|M|^x,\\
\big(\varphi\vee\psi\big)^M=\varphi^M\cup\psi^M\subseteq|M|^x.\\
\end{array}$$

We begin by defining the spectrum $\MM_0=\MM_0(\LL)$ associated with a language $\LL$. Every theory $\bTT$ in the language $\LL$ will define a subspace $\MM_0(\bTT)\subseteq\MM_0(\LL)$. We let $\kappa=\rm{LS}(\bTT)=|\LL|+\aleph_0$ denote the Lowenheim-Skolem number of $\bTT$ (see e.g. \cite{hodges}, ch. 3.1). Since any model has an elementary substructure of size $\kappa$, $\bTT$ is already complete for $\kappa$-small models, a fact that we will use repeatedly.

\begin{defn}[Points of $\MM_0$, cf. Butz \& Moerdijk \cite{butz_thesis} \cite{BM_article}]\label{M0points}
A point $\mu\in |\MM_0|$ is a tuple $\<M_{\mu},K^A_{\mu}, v^A_{\mu}\>_{A\in\AA}$, where $M=M_{\mu}$ is an $\LL$-structure and, for each $A\in\AA$,
\begin{itemize}
\item $K^A_{\mu}$ is a subset of $\kappa$ and
\item $v_{\mu}^A:K_{\mu}^A\epi A^M$ is an infinite-to-one function (called a \emph{labelling}).
\end{itemize}
The latter condition means that, for each $a\in A^M$, the set $\{k\in\kappa\ |\ v_{\mu}^A(k)=a\}$ is infinite.
\end{defn}

We call the elements $k\in\kappa$ \emph{parameters} and the points $\mu\in|\MM_0|$ \emph{labelled models}. The motivation behind these parameters is similar to that behind the use of variable assignments in traditional first-order logic. \emph{A priori} we can only say whether a (labelled) model satisfies a sentence; the parameters allow us to evaluate the truth of any (parameterized) formula $\varphi(\overline{k})$ relative to $\mu$.

Fix a labelled model $\mu\in|\MM_0|$, a context of variables $x:A=\<x_i:A_i\>$ and a formula $\varphi(x)$. Given this data, we introduce the following notation and terminology:
\begin{itemize}
\item If $x:A$ is a context of variables we may write $K_{\mu}^x$ (resp. $v_{\mu}^x$) rather than $K_{\mu}^A$ (resp. $v_{\mu}^A$). We call $K_{\mu}^x$ the \emph{domain} of $\mu$ at $x$.
\item We write that $k\in \mu^x$ to indicate that $k\in K^x_{\mu}$, and say that \emph{$k$ is defined for $x$ at $\mu$}.
\item We write $\varphi^\mu$ in place of $\varphi^{M_{\mu}}$ to denote the definable set associated to $\varphi(x)$ in the model $M_{\mu}$.
\item When $k$ is defined for $x$ at $\mu$ we abuse notation by writing $\mu^x(k)$ for the element $v_{\mu}^x(k)\in|M_{\mu}|^x$.
\item Given a formula $\varphi(x)$, we write $\mu\models\varphi(k)$ when $k\in\mu^x$ and ${M_{\mu}\models \varphi[\mu^x(k)/x]}.$
\end{itemize}

\begin{defn}[Pre-basis for $\MM_0$]
\mbox{}
\begin{itemize}
\item For each context $x:A$ and each parameter $k\in\kappa$ there is a basic open set
$$V_{k,x}:=\big\{\mu\in|\MM_0|\ \big|\ k\in\mu^x\big\}.$$
\item Given a basic relation $R(x)$ and a parameter $k\in\kappa$ there is a basic open set
$$V_{R(k)}=\big\{\mu\in\MM_0\ \big|\ \mu\models R(k)\big\}\subseteq V_{k,x}.$$
\item Given a basic function $F(x)=y$ and parameters $k,l\in\kappa$ there is a basic open set
$$V_{F(k)=l}=\big\{\mu\in\MM_0\ \big|\ \mu\models F(k)=l\big\}\subseteq V_{\<k,l\>,\<x,y\>}.$$
\end{itemize}
\end{defn}

From these we can build up a richer collection of open sets corresponding to other parameterized formulas.
\begin{defn}[Basic opens of $\MM_0$]\label{M0opens}
\mbox{}
\begin{itemize}
\item Each prebasic open set above is a basic open set.
\item Given an inductively defined basic open set $V_{\varphi(k)}$ and a parameter $l^y$ there is a weakening
$$V_{\varphi^y(k,l)}=V_{\varphi(k)}\cap V_{l^y}.$$
\item Only when $V_{\varphi(k)}$ and $V_{\psi(k)}$ share the same parameters $k^x$
$$V_{\varphi\wedge\psi(k)}= V_{\varphi(k)}\cap V_{\psi(k)} \hspace{1cm} V_{\varphi\vee\psi(k)}= V_{\varphi(k)}\cup V_{\psi(k)}.$$
\item If $V_{\varphi(k,l)}$ is defined for all parameters $l^y$, then
$$V_{\exists y.\varphi(k)}=\displaystyle\bigcup_{l\in\kappa^y} V_{\varphi(k,l)}$$
\end{itemize}
\end{defn}
If the defining formula of a basic open set is complicated we may sometimes write it in brackets $V[\ldots]$, rather than as a subscript. We have the following easy lemma.

\begin{lemma}
Given a coherent formula $\varphi(x)$ and a parameter $k\in\mu^x$,
$$\mu\in V_{\varphi(k)}\Iff \mu\models\varphi(k).$$
\end{lemma}

\begin{proof}

The claim holds by definition for basic functions and relations.

Suppose $\varphi(x,y)\equiv \varphi(x)\wedge(y=y)$ is a weakening of a formula $\varphi(x)$. Then $\mu\models \varphi(k,l)$ just in case $\mu\models\varphi(k)$ and $l$ is defined for $y$ at $\mu$. But this is exactly the intersection $V_{\varphi(k)}\cap V_{l,y}$. For joins and meets in a common context, the claim follows immediately from the definition of satisfaction.

As for the existential, consider a formula $\varphi(x,y)$. If $\mu\models \exists y.\varphi(k)$ there must be some $b\in|M_{\mu}|^y$ such that $\mu\models \varphi(\mu^x(k),b)$. Because $\mu$ is a surjective labelling there is a parameter $l$ such that $\mu^y(l)=b$, and therefore $\mu\in V_{\varphi(k,l)}\subseteq\displaystyle\bigcup_{l\in\kappa} V_{\varphi(k,l)}$. 
\end{proof}

Properly speaking, $\MM_0$ as described above is not a $T_0$ space and, in fact, there is a proper class of models in the underlying set $|\MM_0|$. However, these models are all $\kappa$-small so they represent only a set of isomorphism-classes. For each isomorphism class there is only a set of labellings and this allows us to modify $\MM_0$ in an essentially trivial way to obtain an ordinary space.

\begin{prop}
Two labelled models $\mu$ and $\nu$ are topologically indistinguishable if and only if $K_\mu^A=K_\nu^A$ for every basic sort $A$ and these labellings induce a $\bTT$-model isomorphism $M_\mu\cong M_\nu$.
\end{prop}
\begin{proof}
Suppose that $\mu$ and $\nu$ are topologically indistinguishable models (i.e., they belong to exactly the same open sets). Because they co-occur in the same context open sets $V_{k^x}$ we see that $k$ is defined for $x$ at $\mu$ just in case it is defined at $\nu$. Considering the basic open set $V_{k=l}$, we see that $\mu(k)=\mu(l)$ if and only if $\nu(k)=\nu(l)$. This means that the labellings induce a bijection between the underlying sets of $\mu$ and $\nu$:
$$\xymatrix{
K_{\mu} \ar@{=}[r] \ar@{->>}[d] & K_{\nu} \ar@{->>}[d]\\
|M_{\mu}| \ar@{-->}[r]^{\sim} & |M_{\nu}|}.$$
Because $\mu\models R(k)$ just in case $\nu\models R(k)$ for each basic relation (or similarly for functions), this bijection is actually an $\LL$-structure homomorphism. Thus $\mu$ and $\nu$ are topologically indistinguishable just in case there is an isomorphism $M_{\mu}\cong M_{\nu}$ which carries the labelling of $\mu$ to the labelling of $\nu$.
\end{proof}

Among these indistinguishable points there is a canonical representative. The relation $\mu(k)=\mu(l)$ defines an equivalence relation $k\sim_{\mu} l$ on $K_{\mu}$ (or equivalently a partial equivalence relation on the full set $\kappa$). We may substitute the quotient $K_{\mu}/\sim_{\mu}$ (also called a \emph{subquotient} $\kappa/\sim_{\mu}$) in place of the underlying set $|M_{\mu}|$, replacing each $a\in M_{\mu}$ by the equivalence class $\{k\in\kappa\ |\ \mu(k)=a\}$. The result is a labelled model indistinguishable from $\mu$ and canonically determined by the neighborhoods of $\mu$.

Thus every labelled model is indistinguishable from one whose underlying set is a subquotient of $\kappa$, and one can easily show that each of these is distinguishable. Thus in definition \ref{M0points} we could have specified these as the points of our space, making $|\MM_0|$ a set and $\MM_0$ a $T_0$ space. We prefer to use the unreduced space because it is sometimes convenient to refer to underlying sets (as in lemma \ref{reassignment} below); at the expense of conceptual clarity we can always rephrase these results in terms of isomorphisms between subquotients.

Caution: even with this modification, the topological space $\MM_0$ is \emph{not} Hausdorff. In some respects this is unsurprising given that the same phenomena occur for algebraic schemes.  However, the reduced version of $\MM_0$ is at least sober, as in the algebraic case; the reader can see \cite{butz_thesis} or \cite{forssell_thesis} for a proof. To verify the failure of separation in $\MM_0$, we observe that points are not closed.

\begin{prop}\label{spectral_closure}
Consider two labelled models $\mu,\nu\in|\MM_0|$. Then $\mu$ belongs to the closure of $\nu$ just in case
\begin{itemize}
\item For each $A\in \AA$, $K_{\mu}^A\subseteq K_{\nu}^A$,
\item these inclusions induce functions on the underlying sets of $\mu$ and $\nu$:
$$\xymatrix{
K_{\mu}^A \ar@{->>}[d] \ar@{}[r]|{\subseteq} & K_{\nu}^A \ar@{->>}[d]\\
A^\mu \ar@{-->}[r] & A^\nu}
$$
\item The induced functions $A^\mu\to A^\nu$ define an $\LL$-structure homomorphism $M_{\mu}\to M_{\mu}$.
\end{itemize}
\end{prop}

\begin{proof}
First suppose that $\mu\in\overline{\nu}$. This means that, for any open set $V\subseteq \MM_0$, $\mu\in V$ implies $\nu\in V$.

First consider open sets of the form $V_{k,x}$. Applying the above observation, this means that $k$ is defined for $x$ at $\mu$ (so that $\mu\in V_{k,x}$) implies that $k$ is defined for $x$ at $\nu$ as well. Therefore $K_{\mu}^x\subseteq K_{\nu}^x$.

Similarly, whenever $\mu\models k=k'$ then we must have $\nu\models k=k'$ as well. It follows that the inclusion $K_{\mu}^x\subseteq K_{\nu}^x$ induces a function on underlying sets:
$$\xymatrix{
&\kappa&\\
K^x_{\mu} \ar@{->>}[d] \ar@{>->}[ur] \ar@{>->}[rr] && K^x_{\nu} \ar@{->>}[d] \ar@{>->}[ul]\\
|M_{\mu}|^x \ar@{-->}[rr]^{|\alpha|^x} && |M_{\nu}|^x.\\
}$$

Finally, if $\mu\models \varphi(k)$ then $\nu\models \varphi(k)$ as well, so the induced functions $|M_{\mu}|^x\to|M_{\nu}|^x$ preserve any coherent properties of elements in $M_{\mu}$. In particular, this holds for basic functions and relations, so the functions $|\alpha|^x$ define an $\LL$-structure homorphism $M_{\mu}\to M_{\nu}$.

Conversely, if there are inclusions $K_{\mu}^A\subset K_{\nu}^A$ which induce a homorphism $M_{\mu}\to M_{\nu}$, then one easily shows that $\mu$ belongs to any open neighborhood of $\nu$.\footnote{In fact, one can show that the only closed point of $\MM_0$ is the empty structure. This is because the equivalence classes making up the underlying sets $|M_{\mu}|^x$ are infinite. Therefore we can always throw out a finite set of labels to create a new model $\mu^-$ which belongs to the closure of $\mu$.}
\end{proof}

\begin{lemma}[Reassignment lemma]\label{reassignment}
Suppose that $\mu$ is a labelled model, $k^x$ and $l^y$ are disjoint sequences of parameters and $b\in|M_{\mu}|^y$ is a sequence of elements in $\mu$. Then there is another labelling $\nu$ with the same underlying model, $M_{\mu}=M_{\nu}$, such that $\mu(k)=\nu(k)$ and $\nu(l)=b$
\end{lemma}
\begin{proof}
It is possible that some of the $l$-parameters may already be defined at $\mu$, so we begin by removing them from each domain $K_{\mu}$. Because $v_{\mu}$ is infinite-to-one and $l$ is finite, the modified labelling is again infinite-to-one and we call the resulting labelled model $\mu\setminus l$. Because $k$ is disjoint from $l$, we have $v_{\mu\setminus l}(k)=v_{\mu}(k)\in|M_{\mu}|^x$. 

Now we can freely reassign $l$ to $b$. Specifically, when $y=\<y_j:B_j\>$ we set
$$K_{\nu}^A=K_{\mu\setminus l}^A\cup\{ l_j\ |\ B_j=A\}.$$
Then we extend $v_{\nu}$ to this larger domain by setting setting $v_{\nu}(l_j)=b_j$. Clearly $\nu(l)=b$. As an extension of $\mu\setminus l$, we also have $\nu(k)=v_{\mu\setminus l}(k)=\mu(k)$, as desired.
\end{proof}

Now we shift from languages and structures to theories and models. In categorical logic, axioms are always expressed relative to a context of variables $x:A$; in this way we can interpret an outer layer of ``universal quantification'' even though $\forall$ is not a coherent symbol. Similarly, we express our axioms as sequents ${\varphi(x)\underset{x:A}{\vdash} \psi(x)}$, allowing for ``one layer'' of implication (or negation) even though $\to$ is not coherent.

A first-order theory $\bTT$ is \emph{coherent} if it has an axiomatization using sequents of coherent formulas. To indicate that a sequent is valid relative to $\bTT$ we write
$$\bTT\tri \varphi(x)\vdash_{x:A} \psi(x).$$
In an $\LL$-structure $M$, satisfaction of sequents is dictated by definable sets:
$$M\models \varphi(x)\vdash_{x:A}\psi(x)\ \Iff\ \varphi^M\subseteq \psi^M\subseteq |M|^x.$$

\begin{defn}
Consider a coherent theory
$$\bTT=\{\varphi_i(x_i)\underset{x_i:A_i}{\vdash} \psi_i(x_i)\}_{i\in I}$$
written in the language $\LL$. The \emph{(spatial) spectrum of $\bTT$} is the subspace $\Spec_0(\bTT)\subseteq\MM_0(\LL)$ consisting of those labelled $\LL$-structures $\mu$ which are models of $\bTT$: i.e., for all $i\in I$,
$$\varphi_i^\mu\subseteq\psi_i^\mu\subseteq |M_{\mu}|^{x_i}$$
\end{defn}

\begin{lemma}\label{inclusion_lemma}
Now let $\MM_0=\Spec_0(\bTT)$ denote the spectrum of $\bTT$. Any (non-empty) basic open set $U\subseteq V_{\varphi(k)}$ has the form $U=V_{\psi(k,l)}$ where $\bTT,\psi(x,y)\vdash \varphi(x)$.
\end{lemma}
\begin{proof}
There are two claims here. The first is that any basic open set contained in $V_{\varphi(k)}$ also depends on the parameters $k$. The second asserts the $\bTT$-provability of the sequent $\psi(x,y)\vdash\varphi(x)$.

For the first claim, suppose that $U=V_{\psi(l)}$ and that there is some $k_0\in k$ with $k_0\not\in l$. Fix a model $\mu'\in V_{\psi(l)}$ and consider the modified labelling $\mu=\mu'\setminus k_0$ as in the reassignment lemma. We have not adjusted the $l$-parameters so we again have $\mu\in V_{\psi(l)}$. However $k_0$ (and hence $k$) is not even defined at $\mu$ and therefore $\mu\not\in V_{k^x}\supseteq V_{\varphi(x)}$. This is contrary to the assumption that $V_{\psi(l)}\subseteq V_{\varphi(k)}$.

As we assumed $U$ was a basic open set, it must have the form $V_{\psi(k,l)}$ for some formula $\psi$ and some additional parameters $l$. Now we need to see that $\bTT$ proves $\exists y.\psi(x,y)\vdash \varphi(x)$ (which is equivalent to the sequent above). This is an easy application of completeness for $\kappa$-small models.

If the sequent were not valid then we could find a labelled model $\mu$ and a parameter $k^x$ such that $\mu\models \exists y.\psi(k,y)$ but $\mu\not\models \varphi(k)$. Let $b\in |M_{\mu}|^y$ witness the existential. By the reassignment lemma, there is an isomorphism $\alpha:\mu\iso\nu$ with $\nu\not\models \varphi(k)$ and $\nu(l)=\alpha(b)$. But then $\nu\models \psi(k,l)$ so that $V_{\psi(k,l)}\not\subseteq V_{\varphi(k)}$. This is again contrary to the assumption that $U\subseteq V_{\varphi(k)}$.
\end{proof}

\section{Sheaves on $\MM_0$}\label{sec_CovSpSh}

In this section we will define, for every $\bTT$-formula $\varphi(x)$, a sheaf $\ext{\varphi}$ over $\MM_0$.

Fix a basic sort $A\in\AA$ and a variable $x:A$. In a labelled model $\mu$ this corresponds to a labelling of the underlying set $K^x_{\mu}\to |M_{\mu}|^x$. Now we will reencode this data into the category of sheaves over $\MM_0$ by providing a diagram whose stalk at $\mu$ recovers the subquotient up to isomorphism
$\Sh(\MM_0)$
$$\xymatrix{
K^x \ar@{>->}[r] \ar@{->>}[d]_-q&\Delta(\kappa)\\
\ext{A}&\\
}\ \ \raisebox{-.7cm}{$\overset{\stalk_{\mu}}{\longmapsto}$}\ \  \xymatrix{
K^x_{\mu} \ar@{>->}[r] \ar@{->>}[d]_-q&\kappa\\
|M_{\mu}|^x&.\\
}$$

Here $\Delta(\kappa)$ is the constant sheaf $\coprod_{\kappa} 1$; we usually abuse terminology and simply write $\kappa$. As a subobject of the constant sheaf, $K^x$ must be a coproduct of open sets. At the $k$th index, we take the basic open set $V_{k^x}$:
$$K^x=\coprod_{k\in\kappa} V_{k^x} \mono \coprod_{k\in\kappa} \MM_0=\Delta(\kappa).$$
Clearly the stalk of $K^x$ at $\mu$ is isomorphic to the set of parameters $k$ such that $\mu\in V_{k^x}$; this is exactly the domain $K^x_{\mu}$.

\begin{comment}
Alternatively, the open set lattice $\OO(\MM_0)$ is the subobject classifier in $\Sh(\MM_0)$. We can define a map $\kappa\to\OO(\MM_0)$ by sending each $k_0\mapsto V_{k_0}$ and define $K$ as a pullback of the subobject classifier:
$$\xymatrix{
K \pbcorner \ar[r]^\! \ar@{>->}[d] & 1 \ar[d]^\top\\
\kappa \ar[r] & \OO(\MM_0)\\
}$$
\end{comment}

Now we define a sheaf $\ext{A}$ which encodes the underlying set $|M_{\mu}|^x$. We build it as an \'etale space, generalizing the construction of $\MM_0=\ext{\top}$.

\begin{defn}[The semantic realization $\ext{A}$]\label{extA}
\mbox{}\begin{itemize}
\item The points of $\ext{A}$ are pairs $\big\<\mu,a\in |M_{\mu}|^x\big>$.
\item Each $k\in\kappa$ determines a \emph{canonical (partial) section} over the open set $V_{k^x}$:
$$\hat{k}^x:V_{k^x}\to\ext{A}$$
sending $\mu\mapsto \<\mu,\mu^x(k)\>$. This defines an open set $W_{x=k}\subseteq\ext{A}$, homeomorphic to $V_{k^x}$, and these form an open cover of $\ext{A}$. 
\end{itemize}\end{defn}
It is obvious from the definition that the fiber of $\ext{A}$ over $\mu$ is isomorphic to the underlying set $|M_{\mu}|^x$.

Although the sections $\hat{k^x}$ give an open cover of $\ext{A}$, it is convenient to have a richer collection of basic open sets. For any formula $\varphi(x)$ and parameters $k^x$, the basic open set $V_{\varphi(k)}$ has a homeomorphic image in $\ext{A}$:
$$W[\varphi(x)\wedge(x=k)]\cong V_{\varphi(k)}.$$

More generally, any formula $\varphi(x,y)$ together with parameters $l^y$ defines an open set
$$\begin{array}{rcl}
W_{\varphi(x,l)} &=&\{\ \<\mu,a\>\ |\ M_{\mu}\models\varphi(a,\mu^y(l))\ \}\\
&=&\displaystyle\bigcup_{k\in\kappa} W\big[\varphi(x,l)\wedge(x=k)\big].
\end{array}$$
The sets in this union are of the basic open form listed above, so $W_{\varphi(x,l)}$ is again open. The last equality follows from the fact that each $a\in|M_{\mu}|^{x}$ must be labelled by some parameter $k\in\kappa$.

\begin{prop}\label{Asheaf}
$\ext{A}$, as defined above, is a sheaf over $\MM_0$.
\end{prop}
\begin{proof}
A convenient reformulation of the sheaf condition (cf. Joyal-Tierney, \cite{JT}) says that $\FF$ is a sheaf just in case the projection $\FF\to\MM_0$ and the fiber-diagonal ${\delta_{\FF}:\FF\to \FF\underset{\MM_0}{\times} \FF}$ are both open maps. The necessity of first condition is obvious because any sheaf projection is a local homeomorphism; the diagonal condition characterizes discreteness in the fibers of $\FF$.

It is enough to check both these conditions on basic open sets. The projection of a basic $\ext{A}$-open is given by quantifying out the $x$-variable:
$$\begin{array}{rcl}
\pi(W_{\varphi(x,l)}\big)&=&\displaystyle\pi\left(\bigcup_{k\in\kappa} W_{\varphi(x,l)\wedge(x=k)}\right)\\
&=&\displaystyle\bigcup_{k\in\kappa}V_{\varphi(k,l)}\\
&=&V_{\exists x. \varphi(x,l)}\\
\end{array}$$
Here the second equality follows from the definition of $W[\dots]$ as a section of $\pi$.

For the second map, we know that boxes $W_{\varphi(x,j)}\underset{\MM_0}{\times} W_{\psi(x,l)}$ give a basis for the topology of the fiber product. As in $\ext{A}$, unions provide for a richer collection of open sets. Suppose $\gamma(x,x',y)$ is a formula where $x$ and $x'$ are variables in the same context and $l^y$ is an additional parameter; from these we define an open set
$$\begin{array}{rcl}
W_{\gamma(x,x',l)}&=& \{\ \<\mu,a,a'\>\ |\ \mu\models \gamma(a,a',\mu(l))\ \}\subseteq\ext{A}\overtimes{\MM_0}\ext{A}\\
&=&\displaystyle\bigcup_{k,k'\in\kappa}\Big( W\big[\gamma(x,k',l)\wedge(x=k)\big]\overtimes{\MM_0} W\big[\gamma(k,x',l)\wedge(x'=k')\big]\Big).
\end{array}$$
This is a union of boxes of the basic open form above, and equality again follows from the fact that every $a,a'\in|M_{\mu}|^x$ is labelled by some pair $k,k'\in\kappa$.

Now the image of $\ext{A}$ under the fiber diagonal is the open set $W_{x=x'}$. Similarly, a basic open subset $W_{\varphi(x,l)}\subseteq\ext{A}$ maps to $W_{\varphi(x,l)\wedge(x=x')}$. These are both of the form $W_{\gamma(x,x',l)}$, so the fiber diagonal is an open map.
\end{proof}

Together the canonical sections $\hat{k}^x$ describe a map of sheaves $v:K^x\to\ext{A}$
$$\xymatrix{
K^x \cong \coprod_{k\in\kappa} V_{k^x}\ar@{->>}[r]^-v &\ext{A}\\
V_{k^x} \ar@{>->}[u] \ar[ur]_{\hat{k}^x}
}$$
Since these sections cover $\ext{A}$, $v$ is an epimorphism of sheaves. Furthermore, $v$ sends $k\in K^x_{\mu}$ to $\hat{k}^x(\mu)=\<\mu,\mu(k)\>$. This realizes our goal for this section, and altogether we have the following theorem.

\begin{thm}\label{ext_labels}
For each context $x:A$ for a basic sort $A\in\AA$ there is a span in $\Sh(\MM_0)$ of the form
$$\xymatrix{
K^x \ar@{>->}[r] \ar@{->>}[d]_-v&\kappa\\
\ext{A}&}$$
such that
\begin{itemize}
\item[(i)] $K^x$ is a subsheaf of the constant sheaf $\kappa$ and the stalk of $K^x$ at $\mu$ is isomorphic to the domain $K_{\mu}^x=\{k\in\kappa\ |\ k\rm{\ is\ defined\ for\ }x\rm{\ at\ }\mu\}$.
\item[(ii)] The stalk of $\ext{A}$ at $\mu$ is isomorphic to the underlying set $|M_{\mu}|^x=A^\mu$.
\item[(iii)] The map $v: K^x\to\ext{A}$ is an epimorphism of sheaves and the stalk of $v$ at $\mu$ is isomorphic to the subquotient map $v_{\mu}^x:K_{\mu}^x\epi|M_{\mu}|^x$ (cf. definition \ref{M0points}).
\end{itemize}
\end{thm}

\section{The generic model}

In the last section we defined a family of sheaves $\ext{A}$ for the basic sorts $A\in\AA$. In this section we will show that this family carries the structure of a $\bTT$-model in the internal logic of $\Sh(\MM_0)$. Moreover, the model is ``generic'' in the sense that it satisfies all and only the sequents which are provable in $\bTT$.

First we review the interpretation of $\LL$-structures and coherent theories in the internal logic of a topos $\SS$ (cf. Johnstone \cite{elephant}, D1 or Mac Lane \& Moerdijk \cite{SGL}, VI). In this and later chapters we will restrict our attention to Grothendieck toposes; we typically omit the adjective, so ``topos'' will always mean ``Grothendieck topos.''.

To define an $\LL$-structure $M$ in $\SS$ one must give, first of all, an ``underlying object'' $A^M$ for each basic sort $A\in\AA$. Given a compound context $A=A_1\times\ldots\times A_n$ we let $A^M$ denote the product $A_1^M\times\ldots\times A_n^M$. Given a context of variables $x:A$ we use the notation $|M|^x:= A^M$.

$M$ must also interpret basic relations and functions. Each relation $R(x)$ is associated with an interpretation $R^M$ which is a subobject of $|M|^x$. Similarly, a function symbol $f:x\to y$ is interpreted as an $\SS$-arrow $f^M:|M|^x \to |M|^y$.

From these we can define an interpretation $\varphi^M\leq |M|^x$ for each coherent formula $\varphi(x)$; we construct these in essentially the same way that we build up the definable sets of a classical model in $\Sets$. We use limits (intersections, pullbacks) in $\SS$ to interpret meet and substitution. Epi-mono factorization gives us existential quantification while factorization together with coproducts gives us joins.

\noindent\begin{tabular}{cccc}\label{coh_logic}
\\
\textbf{Subst.} &$\xymatrix{(\varphi[f(x)/y])^M \ar[r] \pbcorner \ar@{>->}[d] & \varphi^M \ar@{>->}[d]\\
|M|^y \ar[r]_{f^M} & |M|^x\\}$ &
\textbf{Meet} & $\xymatrix{(\varphi\wedge\psi)^M \pbcorner \ar@{>->}[r] \ar@{>->}[d] & \psi^M \ar@{>->}[d]\\
\varphi^M \ar@{>->}[r] & |M|^x\\}$\\\\
\textbf{Exist} & $\xymatrix{\varphi^M \pbcorner \ar@{->>}[r] \ar@{>->}[d] & (\exists y.\varphi)^M \ar@{>->}[d]\\
|M|^{\<x,y\>} \ar[r] & |M|^x\\}$ & \textbf{Join} & $\xymatrix{\varphi^M + \psi^M \ar@{>->}[d] \ar@{->>}[r] & (\varphi\vee\psi)^M \ar@{>->}[d]\\
|M|^x+|M|^x \ar[r] & |M|^x}$\\\\
\end{tabular}

Remember that an axiom in categorical logic has the form of a sequent $\varphi(x)\vdash_{x:A} \psi(x)$. We say that $M$ \emph{satisfies} this sequent if $\varphi^M\leq\psi^M$ in the subobject lattice $\Sub_{\SS}(|M|^x)$. As usual, $M$ \emph{satisfies $\bTT$} or is a \emph{$\bTT$-model} if it satisfies all the sequents in $\bTT$. We denote the class of $\bTT$-models in $\SS$ by $\bTT\Mod(\SS)$.

In the last section we defined a sheaf $\ext{A}$ for each basic sort $A\in\AA$; these are the underlying objects refered to above. We have already had occasion to consider the fiber product $\ext{A^2}=\ext{A}\overtimes{\MM_0}\ext{A}$ in the proof of proposition \ref{Asheaf}. There we saw that each parameterized formula $\gamma(x,x',l)$ involving two $A$-variables and an additional parameter $l^y$ determines a basic open set $W_{\gamma(x,x',l)}\subseteq\ext{A^2}$.\label{defin_products}

We can generalize this to any context $x=\<x_i:A_i\>_{i\leq n}$ by setting $\ext{A^x}=\ext{A_1}\overtimes{\MM_0}\ldots\overtimes{\MM_0}\ext{A_n}$. Given a formula $\varphi(x,y)$ and parameters $l^y$, the fiber product contains an open set
$$W_{\varphi(x,l)}=\{ \<\mu, a\>\ |\ M_{\mu}\models\varphi(a,\mu^y(l))\}.$$
Just as before, we show this set is open by representing it as a union of boxes indexed over the set of parameter sequences $k^x$.

Given a basic relation $R(x)$ we extend the $\ext{-}$ notation by setting
$$\ext{R}=W_{R(x)}=\{\<\mu,a\>\in\ext{A^x}\ |\ M_{\mu}\models R(a)\}.$$
This provides an interpretation in $\Sh(\MM_0)$ for each basic relation.

Similarly, each function symbol $f:x\to y$ induces a map of sheaves $\ext{f}:\ext{A^x}\to\ext{B^y}$. This is defined fiberwise, sending each $a\in |M_{\mu}|^x$ to $f^\mu(a)\in |M_{\mu}|^y$. This is continuous because the inverse image of a basic open set corresponds to substitution
$$\ext{f}^{-1}(V_{\varphi(x,k)})= V_{\varphi[f(y)/x,k]}\subseteq \ext{B^y}.$$

This specification defines an $\LL$-structure $M^*$ in $\Sh(\MM_0)$. Now suppose that $\varphi(x)\vdash_{x:A} \psi(x)$ is an axiom of $\bTT$ proves a coherent sequent $\varphi(x)\vdash_{x:A}\psi(x)$. Then for each labelled model there is an inclusion of definable sets $\varphi^\mu\subseteq\psi^\mu\subseteq A^\mu$ and consequently $\ext{\varphi}\subseteq\ext{\psi}$. Hence our $\LL$-structure satisfies the sequent, and $M^*$ is a model of $\bTT$ in $\Sh(\MM_0)$.

In fact, every coherent formula $\varphi(x)$ determines a \emph{definable sheaf} $\ext{\varphi}\subseteq\ext{A^x}$. On one hand, $\ext{\varphi}$ is the interpretation of $\varphi$ in the sheaf model $M_0$; as such, it can be constructed inductively from the interpretation of basic relations and functions. Alternatively, we may define $\ext{\varphi}$ semantically by setting
$$\ext{\varphi}=W_{\varphi(x)} = \{\<\mu,a\>\in\ext{A^x}\ |\ M_{\mu}\models \varphi(a)\}.$$
\label{def_sheaf}

\begin{prop}
$M_0$ is a generic model for $\bTT$ in the sense that a coherent sequent $\varphi(x)\vdash_{x:A} \psi(x)$ is provable in $\bTT$ if and only if it is satisfied in $M_0$:
$$\bTT\tri\varphi(x)\vdash_{x:A}\psi(x) \Iff M_0\models \varphi(x)\vdash_{x:A}\psi(x).$$
\end{prop}
\begin{proof}
The left to right implication is an immediate consequence of soundness. If $\bTT$ proves the sequent then for every $\bTT$-model $M$ (and hence every labelled model $\mu$), $\varphi^M\subseteq \psi^M$. But then $\ext{\varphi}\subseteq \ext{\psi}$ so that $M^*$ satisfies the sequent as well.

The converse follows from the fact that $\bTT$ is complete for $\kappa$-small models (see \pageref{M0points}). If $M_0$ satisfies the sequent $\varphi(x)\vdash_{x:A}\psi(x)$ then $\ext{\varphi}\subseteq\ext{\psi}$. Then $\varphi^\mu\subseteq \psi^\mu$ for every labelled model $\mu$, and every $\kappa$-small model is isomorphic to a labelled model. Thus satisfaction in $M_0$ implies satisfaction in every $\kappa$-small model, and completeness ensures provability in $\bTT$.
\end{proof}

\section{The spectral groupoid $\MM=\Spec(\bTT)$}\label{sec_SpGpd}

In the last section we saw that each formula $\varphi(x)$ defines a subsheaf $\ext{\varphi}\subseteq \ext{A^x}$. Now, following Butz \& Moerdijk \cite{butz_thesis} \cite{BM_article}, we will characterize those subsheaves which are definable in this sense. To do so we introduce a space of $\bTT$-model isomorphisms together with continuous domain, codomain, composition and inverse operations. This topological groupoid $\MM=\Spec(\bTT)$ is the (groupoid) spectrum of $\bTT$. This groupoid acts naturally on the definable sheaves and we will show in the next section that the existence of such an action, together with a compactness condition, characterizes definability.

A groupoid in $\Top$ is a diagram of topological spaces and continuous maps like the one below:
$$\xymatrix{
\MM:&\MM_1\underset{\MM_0}{\times}\MM_1 \ar@<-1.5ex>[rr]_-\circ \ar@<.5ex>[rr] \ar@<1.5ex>[rr]^-{p_0,\ p_1}&& \MM_1  \ar@(ul,ur)^{\rm{inv}} \ar@<.5ex>[rr] \ar@<1.5ex>[rr]^-{\rm{dom,\ cod}} &&\MM_0 \ar@<1.5ex>[ll]^{\rm{id}}
}$$
$\MM_0$ and $\MM_1$ are called the object and arrow spaces of $\MM$, respectively. These spaces and continuous maps are required to satisfy the same commutative diagrams as a groupoid in $\Sets$. See \cite{SGL}, section V.7 for a discussion of internal categories and equivariance.

\begin{defn}[The spectral groupoid $\MM=\Spec(\bTT)$]\label{M1points}
\mbox{}
\begin{itemize}
\item An isomorphism of labelled models $\alpha\in\Hom_{\MM}(\mu,\nu)$ is simply a isomorphism of underlying $\bTT$-models $\alpha:M_{\mu}\iso M_{\nu}$. We do not require that these respect the labellings on $\mu$ and $\nu$. Domain, codomain, composition, inverse and identity are computed as in $\bTT\Mod$.
\item For each $x:A$, $\alpha$ defines a component $\alpha^x:|M_{\mu}|^x\to|M_{\nu}|^x$. Given an element $a\in|M_{\mu}|^x$ we often omit the superscript and simply write ${\alpha(a)\in|M_{\nu}|^x}$.
\item Each basic open set $V_{\varphi(k)}\subseteq \MM_0$ determines two basic open sets in $\MM_1$, the inverse images of $\underline{\dom}$ and $\underline{\cod}$:
$$V_{\varphi(k)^d} = \{\alpha:\mu\iso\nu\ |\ \mu\models \varphi(k)\}$$
$$V_{\varphi(k)^c} = \<\alpha:\mu\iso\nu\ |\ \nu\models \varphi(k)\}.$$
We refer to these basic opens as domain and codomain conditions in $\MM_1$.
\item For any context $x:A$ and any two parameter sequences $k^x$ and $l^x$ in the same arity, there is a basic open set
$$V_{k\overset{x}{\mapsto} l}=\{\alpha:\mu\iso\nu\ |\ \alpha(\mu^x(k))=\nu^x(l)\}.$$
\end{itemize}\end{defn}

\begin{prop}
$\MM$, as defined above, is a topological groupoid.
\end{prop}
\begin{proof}
As the compositional structure on $\MM$ is inherited from $\bTT\Mod$, the internal category conditions on $\MM$ are immediate. A bit less obvious is that all of the structure maps are continuous. For the domain and codomain maps this is built into the definition; it follows that the fiber projections $p_0$ and $p_1$ are continuous as well.

After these inverse is the easiest, as it obviously swaps the basic open sets in pairs.
$$\begin{array}{rcl}
\alpha\in V_{\varphi(k)^d}&\Iff& \alpha^{-1}\in V_{\varphi(k)^c}\\
\alpha\in V_{k\overset{x}{\mapsto} l}&\Iff&\alpha^{-1}\in V_{l \overset{x}{\mapsto} k}.\\
\end{array}$$
For $\underline\id$ we have
$$\begin{array}{ccccc}
1_{\mu}\in V_{\varphi(k)^d}&\Iff &1_{\mu}\in V_{\varphi(k)^c}&\Iff& \mu\in V_{\varphi(k)}\\
1_{\mu}\in V_{k \overset{x}{\mapsto} l}&\Iff& \mu^x(k)=\mu^x(l)&\Iff& \mu\in V_{k\underset{x}{=}l}.\\
\end{array}$$
These latter sets are open in $\MM_0$, so $\underline\id$ is continuous as well.

Lastly, composition. Note that if $\alpha$ satisfies a codomain condition $V_{\varphi(k)^c}$ then so does any composite $\alpha\circ \beta$. This means that the inverse image of $V_{\varphi(k)^c}$ along $\circ$ is just a fiber product with $\MM_1$
$$\xymatrix{
V_{\varphi(k)^c}\underset{\MM_0}{\times} \MM_1 \ar[r]^-{p_0} \ar@{>->}[d]& V_{\varphi(k)^c} \ar@{>->}[d]\\
\MM_1\underset{\MM_0}{\times} \MM_1 \ar[r]^-\circ& \MM_1.\\
}$$
This is clearly open, and the same reasoning applies to domain conditions.

Lastly, suppose that we have composable maps ${\mu\stackrel{\beta}{\lto}\lambda\stackrel{\alpha}{\lto}\nu}$ and a neighborhood $\alpha\circ\beta\in V_{k\overset{x}{\mapsto} l}$. Choose any parameter $j$ such that $\alpha(\mu^x(k))=\lambda^x(j)$; this determines an open box in the fiber product:
$$\<\alpha,\beta\>\in\big(V_{k\overset{x}{\mapsto}  j}\big)\times_{\MM_0} \big(V_{j\overset{x}{\mapsto}  l}\big)\subseteq\circ^{-1}(V_{k\overset{x}{\mapsto}  l}).$$
Hence composition and the other structure maps are continuous, and $\MM$ is a groupoid in $\Top$.
\end{proof}

\begin{prop}
Suppose that we have two isomorphisms $\alpha:\mu_0\cong\mu_1$ and $\beta:\nu_0\cong\nu_1$. Then $\alpha$ belongs to the closure of $\beta$ just in case $\mu_i$ belongs to the closure of $\nu_i$ ($i=1,2$) and the canonical homomorphisms $h_i:M_{\mu_i}\to M_{\nu_i}$ induced by these closures form a commutative square
$$\xymatrix{
M_{\nu_0} \ar[r]^\beta & M_{\nu_1}\\
M_{\mu_0} \ar[r]_{\alpha} \ar[u]^{h_0} & M_{\mu_1} \ar[u]_{h_1}\\
}$$
\end{prop}

\begin{proof}\label{iso_closure}
We have $\alpha\in\overline{\beta}$ just in case $\beta$ belongs to every $\MM_1$-open set which $\alpha$ belongs to. Applying this observation to the domain and codomain conditions tells us that $\nu_i$ belongs to every $\MM_0$-open set which $\mu_i$ does, and this tells us that $\mu_i\in\overline{\nu_i}$.

Recall from proposition \ref{spectral_closure} that $\mu$ belongs to the closure of $\nu$ just in case $K_{\mu}\subseteq K_{\nu}$ and this inclusion induces a model homomorphism $h(\mu(k))=\nu(k)$:
$$\xymatrix{
K_{\mu} \ar@{}[r]|{\textstyle\subseteq} \ar@{->>}[d] & K_{\nu} \ar@{->>}[d]\\
|M_{\mu}| \ar@{-->}[r]_{h} & |M_{\nu}|.\\
}$$

For any element $a\in|M_{\mu_0}|$ we can find a parameter $k$ such that $a=\mu_0(k)$ and a parameter $l$ such that $\mu_1(l)=\alpha(a)$. This tells us that $\alpha$ belongs to $V_{k\mapsto l}$, and therefore $\beta$ must as well. It follows that
$$\begin{array}{rcl|l}
\beta(h_0(a))&=&\beta(h_0(\mu_0(k))& k\ \mathrm{labels}\ a\\
&=& \beta(\nu_0(k))& \mathrm{def'n\ of\ }h_0\\
&=& \nu_1(l) & \beta\in V_{k\mapsto l}\\
&=& h_1(\mu_1(l)) & \mathrm{def'n\ of\ }h_1\\
&=& h_1(\alpha(\mu_0(k))) & \alpha\in V_{k\mapsto l}\\
&=& h_1(\alpha(a)) & k\ \mathrm{labels\ }a\\
\end{array}$$
Therefore $\beta\circ h_1=h_0\circ\alpha$, as asserted.
\end{proof}

In what follows we will often need to pull back along the domain map $\MM_1\to\MM_0$ so we give this operation a special notation $\MM*(-)$. Given a sheaf $F\in\Sh(\MM_0)$, an element $\MM* F$ consists of a map $\alpha:\mu\to\nu$ together with an element of the fiber $f\in F_{\mu}$. Similarly, $\MM*\MM*F$ denotes the pullback of $\MM* F$ along the second projection $p_1:\MM_1\overtimes{\MM_0}\MM_1\to\MM_1$. This is the space of composable pairs ${\mu\stackrel{\alpha}{\lto}\nu\stackrel{\beta}{\lto}\lambda}$ together with an element $f\in F_{\mu}$. 

With this shorthand, an \emph{equivariant sheaf} is an object $F\in\Sh(\MM_0)$ (viewed as an \'etale space) together with an action $\rho:\MM* F\to F$ commuting with the codomain:
$$\xymatrix{
\MM*F \ar[d]\ar[rrr]^{\rho} &&& F \ar[d]\\
\MM*\MM_0 \ar@{=}[r]^-{\sim}&\MM_1\ar[rr]_-{\cod} && \MM_0.\\
}$$

This says that each $\alpha:\mu\to\nu$ defines an action on fibers, $\rho_{\alpha}:F_{\mu}\to F_{\nu}$. We additionally require that this action satisfies the following \emph{cocycle conditions}, ensuring that $\rho$ respects the groupoid structure in $\MM$:
$$\begin{array}{ccc}
\rho_{1_{\mu}}(f)=f && \rho_{\alpha\circ\beta}(f)=\rho_{\alpha}\circ\rho_{\beta}(f)\\\\
\xymatrix@C=35pt{F \ar[r]^-{\id\times 1_{F}} \ar@{=}[dr]& \MM* F \ar[d]^\rho\\& F\\} &&
\xymatrix@C=35pt{\MM*\MM* F \ar[r]^-{\circ\times 1_{F} } \ar[d]_{1_{\MM}\times\rho} & \MM* F \ar[d]^\rho\\
\MM* F \ar[r]_{\rho} & F. }\end{array}$$

\begin{prop}
For each context $x:A$, the assignment
$$\rho_{x,\alpha}=\alpha^x:|M_{\mu}|^x\to|M_{\nu}|^x.$$
defines a canonical $\MM$-equivariant structure $\rho_x$ on $\ext{A}$.
\end{prop}
\begin{proof}
The cocycle conditions for $\rho_x$ reduce to associativity and identity conditions for composition of $\bTT$-model homomorphisms. As for continuity, suppose that $\alpha:\mu\to\nu$ and $a\in|M_{\mu}|^x$. For any neighborhood $\alpha(a)\in W_{\varphi(x,l)}$ pick some $k$ such that $\alpha:k\mapsto l$.

Because $\nu\models\varphi(\alpha(a),l)$ and $\alpha$ is an isomorphism, $\mu\models\varphi(a,k)$. Thus the pair $\<\alpha,a\>$ belongs to an open neighborhood inside the inverse image
$$V_{k\overset{x}{\mapsto} l}\overtimes{\MM_0} W_{\varphi(x,k)}\subseteq \rho_x^{-1}(W_{\varphi(x,l)}).$$
Therefore $\rho$ is continuous.
\end{proof}

\section{Stability, compactness and definability}\label{sec_StabDef}

\begin{comment}

In this section we will prove the following theorem, giving a semantic characterization of definability:

\begin{thm}[Topological definability theorem]\label{TopDef}
Fix a context $x:A$ and suppose that for each $\bTT$-model $M$ we have a subset $S_M\subseteq |M|^x$. The family ${S_M}$ is definable by a coherent formula just in case
\begin{itemize}
\item[(i)] $S$ is stable under isomorphisms. For all $\alpha:M\iso N$, $\alpha^x(S_M)=S_{N}$.
\item[(ii)] For all $M$ and $a\in S_M$, there is some formula $\varphi(x)$ such that $M\models\varphi(a)$ and for all $N$, $\varphi^N\subseteq S_N$.
\item[(iii)] For each $M$ and each cover $S_M=\displaystyle\bigcup_{i\in I} \varphi_i^M$ there is a finite subcover $S_M=\varphi_{i_1}^M\cup\ldots\cup\varphi_{i_n}^M$.
\end{itemize}
\end{thm}

\end{comment}

Now we are ready to characterize the definable sheaves $\ext{\varphi}$ in terms of equivariance together with a further condition: compactness. We say that an equivariant sheaf $\<E,\rho\>$ is \emph{compact} if any cover by \emph{equivariant} subsheaves has a finite subcover. Equivalently, any cover has a finite subfamily whose orbits under $\rho$ cover $E$. 

\begin{thm}[Stability for subobjects, cf. Awodey \& Forssell \cite{forssell_thesis,FOLD}]\label{stab_thm}
A subsheaf $S\subseteq\ext{A}$ is definable if and only if it is equivariant and compact.
\end{thm}

\begin{prop}
Each definable subsheaf $\ext{\varphi}\leq\ext{A}$ is equivariant and compact.
\end{prop}
\begin{proof}
If $\alpha:\mu\to\nu$ is an isomorphism and $\mu\models \varphi(a)$ then clearly $\nu\models\varphi(\alpha(a))$. This means that the restriction of $\rho_A$ factors through $\ext{\varphi}$, making it an equivariant subsheaf
$$\xymatrix{
\MM*\ext{A} \ar[rr]^{\rho_A} && \ext{A} \\
\MM*\ext{\varphi} \ar@{>->}[u] \ar@{-->}[rr]_{\rho_{\varphi}} && \ext{\varphi} \ar@{>->}[u]\\
}$$

Now suppose that $\ext{\varphi}=\bigcup_{i\in I} \ext{\psi_i}$ and let us abbreviate the sequent ${\psi(x)\vdash_{x:A}\bot}$ by $\neg\psi(x)$. Suppose that $\ext{\varphi}$ were not compact. Then for any finite subset $I_0\subseteq I$, $\bigcup_{i\in I_0}\ext{\psi_i}\subsetneq \ext{\varphi}$; thus there is a labelled model $\mu$ and an element $a\in|M_{\mu}|^x$ such that $\mu$ satisfies $\{\varphi(a)\}\cup\left\{\neg \bigvee_{i\in I_0} \psi_i(a)\right\}$.

Extend the language of $\bTT$ by a constant $\bf{c}_0$ and let
$$\bTT'=\bTT\cup\{\varphi(\bf{c}_0)\}\cup\{\neg\psi_i(\bf{c}_0)\}.$$
The pairs $\<M_0,a_0\>$ witness the consistency of each finite subtheory, so $\bTT'$ is also consistent. Let $\<\mu_*,a_*\>$ be a labelled model of $\bTT'$. Then $a_*\in\varphi^\mu$ but $a_*\not\in \psi_i^\mu$ for each $i\in I$. This contradicts the assumption that $\ext{\varphi}=\displaystyle\bigcup_{i\in I} \ext{\psi_i}$.
\end{proof}

This shows that every formula $\varphi(x)$ defines a compact equivariant subsheaf of $\ext{A}$. Now we argue the converse. First we show that every equivariant subsheaf is a union of definable pieces and from this compactness easily implies definability.

Recall from definition \ref{extA} that each parameter $k^x$ determines a canonical open section $\ext{k^x}:V_{k^x}\to\ext{A}$ sending $\mu\mapsto \<\mu,\mu^x(k)\>$. When $\mu\models\varphi(k)$ this factors through $\ext{\varphi}$, giving a subsection
$$\xymatrix{
\ext{\varphi} \ar@{}[r]|\subseteq & \ext{A}\\
V_{\varphi(k)} \ar@{-->}[u]^{\ext{k\models\varphi}} \ar@{}[r]|\subseteq & V_{k^x} \ar[u]_{\ext{k}}\\
}$$

\begin{lemma}\label{equiv_ext}
Given any equivariant sheaf $\<E,\rho\>$ and a partial section $e:V_{\varphi(k)}\to E$ there is a unique equivariant extension of $e$ along the canonical section $\ext{k\models\varphi}$:
$$\xymatrix{
\ext{\varphi} \ar[rr]^{\tilde{e}} && E\\
&V_{\varphi(k)} \ar[ul]^{\ext{k\models\varphi}} \ar[ur]_e&.\\
}$$
\end{lemma}
\begin{proof}
This is an easy application of the reassignment lemma \ref{reassignment}. For any point $\<\mu,a\>\in\ext{\varphi}$ we find another model $\nu$ and an isomorphism $\alpha:\mu\cong\nu$ so that $\nu(k)=\alpha(a)$. Then equivariance forces us to set
$$\begin{array}{rcl}
\tilde{e}\<\mu,a\>&=& \rho_{\alpha^{-1}}\circ\tilde{e}\circ\rho_{\alpha}\<\mu,a\>\\
&=&\rho_{\alpha^{-1}}\circ\tilde{e}\ \<\nu,\alpha(a)\>\\
&=&\rho_{\alpha^{-1}}\circ\tilde{e}\circ\ext{k\models\varphi}(\mu)\\
&=&\rho_{\alpha^{-1}}\circ e(\nu).
\end{array}$$
\end{proof}

\begin{lemma}
If $S\subseteq\ext{A}$ is an equivariant subsheaf, then $S$ is a union of definable subsheaves.
\end{lemma}
\begin{proof}
First we show that when a basic open set $W_{\varphi(x,k)}$ is contained in $S$ then so is the definable sheaf $\ext{\exists y.\varphi}$. Fix a model $\mu$ and a pair $\<a,b\>\in|M_{\mu}|^{\<x,y\>}$ such that $\mu\models\varphi(a,b)$. We need to see that $\<\mu,a\>$ belongs to $S$.

By the reassignment lemma we can find an isomorphism $\alpha:\mu\iso\nu$ such that $\alpha(b)=\nu(k)$. Then $\nu\models\varphi(\alpha(a),k)$, so that $\<\nu,\alpha(a)\>\in W_{\varphi(x,k)}\subseteq S$. Because $S$ is equivariant along $\alpha$, this implies that $\<\mu,a\>\in S$ as well.

For each $\<\mu,a\>\in S$ choose a basic open neighborhood $W[\varphi_a(x,k_a)]\subseteq S$. By the foregoing, $\ext{\exists y_a.\varphi_a}$ is contained in $S$ as well, so we have
$$S=\displaystyle \bigcup_{\<\mu,a\>\in S} \ext{\exists y_a.\varphi_a}.$$
\end{proof}

\begin{thm}\label{coh_def}
An equivariant subsheaf $S\subseteq \ext{A}$ is definable if and only if it is compact.
\end{thm}
\begin{proof}
We have just seen that $S=\bigcup_i\ext{\varphi_i}$ for some set of formulas $\{\varphi_i(x)\}_{i\in I}$. If $S$ is compact, then we can reduce this to a finite subcover $$S=\ext{\varphi_{i_1}}\cup\ldots\cup\ext{\varphi_{i_n}},$$ in which case $S$ is definable by the finite disjunction:
$$S=\bigext{A}$$
$$S=\bigext{\bigvee_{j=1,\ldots,n} \varphi_{i_j}}.$$

Conversely, suppose that $S=\ext{\varphi}$ is definable and that $S=\bigcup_{i\in I} T_i$ for some equivariant subsheaves $T_i$. By the previous lemma, we may express each of these as a union of definable pieces: $T_i=\bigcup_{j\in J_i} \ext{\psi_{ij}}$. Then $S=\bigcup_{i,j}\ext{\psi_{ij}}$ and, by completeness, it follows that
$$\bTT\tri \{\psi_{ij}(x)\}_{i\in I, j\in J_i}\vdash_{x:A} \varphi(x).$$

Logical compactness ensures that only finitely many of these formulas is required to prove $\varphi$. Consequently, finitely many of the definable sheaves $\ext{\psi_{ij}}$ cover $S$. These are contained in finitely many of the subsheaves $T_i$, which must also cover $S$, so $S$ is compact for equivariant covers.
\end{proof}

\begin{cor}[cf. Butz \& Moerdijk \cite{butz_thesis} \cite{BM_article}]\label{ext_full}
Any equivariant map between definable sheaves $s:\ext{A}\to\ext{B}$ is definable by a formula $\sigma(x,y)$, in the sense that for any $\<\mu,a\>\in\ext{A}$, $s(a)=b$ iff $\mu\models\sigma(a,b)$.
\end{cor}
\begin{proof}
Because $s$ is equivariant, the graph $\Gamma_s=\{\<\mu,a,b\>|s(a)=b\}$ defines an equivariant subsheaf of $\ext{A}\overtimes{\MM_0}\ext{B}$. Because it is a graph, this subsheaf is isomorphic to $\ext{A}$ which is compact.

It then follows that the graph is definable: $\Gamma_s=\ext{\sigma}$, so
$$s(a)=b\Iff \<\mu,a,b\>\in\ext{\sigma}\Iff \mu\models\sigma(a,b).$$
\end{proof}

In the next chapter we will define the syntactic pretopos $\EE_{\bTT}$ associated with a coherent theory $\bTT$. Roughly speaking, the object of $\EE_{\bTT}$ are formulas in context (supplemented with formal coproducts and quotients). The arrows are provably functional relations. A $\bTT$-model is equivalent to a pretopos functor $\EE_{\bTT}\to\Sets$, taking each formula $\varphi$ to the definable set $\varphi^M$. We close this section with a few topos-theoretic results connecting $\EE_{\bTT}$ and $\MM_{\bTT}$.

For any pretopos $\EE$, the family of finite, jointly epimorphic families defines a Grothendieck topology called the \emph{coherent topology} $\JJ_c$. Whenever we talk about sheaves on a pretopos we will mean the coherent sheaves, so we write $\Sh(\EE)$ rather than $\Sh(\EE,\JJ_c)$.

It is well-known that $\Sh(\EE_{\bTT})$ is the classifying topos for $\bTT$, and for any other topos $\SS$ a geometric morphism $\epsilon:\SS\to\Sh(\EE_{\bTT})$ is essentially determined by a pretopos functor $\epsilon_0:\EE_{\bTT}\to\SS$. From this, one defines the inverse image $\epsilon^*:\Sh(\EE_{\bTT})\to\SS$ by left Kan extension; each $\EE_{\bTT}$-sheaf is a colimit of representables $F\cong\colim_j yA_j$ and $\epsilon^*(F)\cong\colim_j \epsilon_0(A_j)$.

\begin{thm}[Butz-Moerdijk \cite{butz_thesis} \cite{BM_article}]\label{eq_top_equiv}
$\EqSh(\MM_{\bTT})$ is the classifying topos for $\bTT$-models. Specifically, the classifying geometric morphism of the generic $\bTT$-model in $\EqSh(\MM_{\bTT})$ is an equivalence of categories.
\end{thm}
\begin{proof}
The following argument comes from Awodey \& Forssell, \cite{forssell_thesis, FOLD}.

The definable sheaf construction induces a pretopos functor $\ext{-}:\EE_{\bTT}\to\EqSh(\MM)$. Finite limits are preserved a the level of fibers because these are definable sets (where limits are preserved) and so so $\big|\ext{\lim A_i}\big|\cong\big|\lim \ext{A_i}\big|$. Then observe (cf. proposition \ref{Asheaf}) that every open set in $\ext{\lim A_i}$ is a union of boxes from the factors. Similarly, coproducts and quotients are preserved on the stalks and, since colimits in equivariant sheaves are computed stalk-wise, they are also preserved at the level of sheaves.

By the Grothendieck comparison lemma, the induced geometric morphism $\EqSh(\MM)\to\Sh(\EE_{\bTT})$ is an equivalence just in case $\ext{-}$ is full, faithful and generating, in the sense that for any equivariant sheaves and maps $h,h':F\rightrightarrows G$ there is a definable sheaf $\ext{\varphi}$ and an equivariant map $e:\ext{\varphi}\to F$ such that $h\circ e\not=h'\circ e$. Corollary \ref{ext_full} shows that $\ext{-}$ is full.

Completeness ensures that it is also faithful. Suppose that $\sigma\not=\tau:A\rightrightarrows B$. Then there is a model $M$ (which we may assume is $\kappa$-small) and an element $a\in A^M$ such that $\sigma^M(a)\not=\tau^M(a)$. Choose any labelled model $\mu$ with $M_{\mu}=M$. Then $\ext{\sigma}(\<\mu,a\>)=\sigma^\mu(a)\not=\tau^\mu(a)=\ext{\tau}(\<\mu,a\>)$, so $\ext{\sigma}\not=\ext{\tau}$.

Now suppose that $\<F,\rho\>$ is any equivariant sheaf on $\MM$. For any basic open set there is a canonical section $\hat{k}:V_{\varphi(k)}\to\ext{\varphi}$, and every partial section $s:V_{\varphi(k)}\to F$ has a unique equivariant extension to $\ext{\varphi}$ such that $\hat{s}\circ\hat{k}=s$ (lemma \ref{equiv_ext}). Any $h,h'$ can be distinguished by some partial section $s$ (as these cover $F$), so these are also distinguished by $\hat{s}$. Therefore $\ext{-}$ generates $\EqSh(\MM)$.
\end{proof}

\begin{prop}\label{essential_gm}
The geometric morphism $\epsilon:\xymatrix{\Sh(\MM_0)\to\EqSh(\MM)\simeq\Sh(\EE)}$ is both surjective and essential (and therefore open).
\end{prop}
\begin{proof}
The inverse image of $\epsilon$ is given by the forgetful functor $\epsilon^*:\Sh(\EE)\simeq\EqSh(\MM)\to\Sh(\MM_0)$. We extend the semantic bracket notation by writing $\epsilon^*E=\ext{E}$ for any $E\in\Sh(\EE)$. By definition, $\epsilon$ is surjective just in case $\epsilon^*$ is faithful. Since an equivariant map $E\to E'$ is just a sheaf map which preserves the equivariant action, this is certainly true.

According to the adjunction, the direct image $\epsilon_*$ is is canonically determined by maps out of the definable sheaves. Given a sheaf $F\in\Sh(\MM_0)$,
$$\begin{array}{rcl}
(\epsilon_*F)(A)&\cong & \Hom_{\Sh(\EE)}(yA,\epsilon_*F)\\
&\cong& \Hom_{\Sh(\MM_0)}(\ext{A},F).\\
\end{array}$$

The essential left adjoint $\epsilon_!\dashv\epsilon^*$ is determined by colimits (which it must preserve) together with the requirement $\epsilon_! V_{\varphi(k)}=y\varphi$. Since any $F\in\Sh(\MM_0)$ is a colimit of these basic opens, $F\cong\colim_i V_{\varphi_i(k_i)}$, we are forced to define $\epsilon_!$ by
$$\epsilon_!F\cong\colim_i y\varphi_i.$$

The adjunction then follows from the universal property of the canonical sections (referenced in the previous theorem). These induce a sequence of canonical isomorphisms
$$\begin{array}{rcl}
\Hom_{\Sh(\MM_0)}(F,\ext{E})&\cong& \Hom_{\Sh(\MM_0)}(\colim_i V_{\varphi_i(x_i)},\ext{E})\\
&\cong & \lim_i \Hom_{\Sh(\MM_0)}(V_{\varphi_i(x_i)},\ext{E})\\
&\cong & \lim_i \Hom_{\EqSh(\MM)}(\ext{\varphi_i},\ext{E})\\
&\cong & \lim_i \Hom_{\Sh(\EE)}(y\varphi_i,E)\\
&\cong & \Hom_{\Sh(\EE)}(\colim_i y\varphi_i,E)\\
&\cong & \Hom_{\Sh(\EE)}(\epsilon_!F,E).
\end{array}$$

\end{proof}

\section{Classical first-order logic}\label{FOL}

The foregoing chapter has been concerned with the coherent fragment of (intuitionistic) first-order logic. In this section we discuss the relevant generalization to full first-order (classical) logic. In order to extend the coherent definability theorem we will replace the given language $\LL$ by a theory $\overline{\bTT}$, written in an extended language $\overline{\LL}$, such that $\MM(\LL)^{fo}\cong\MM(\bTT)^{coh}$.

Supplementing coherent logic by negation yields a complete set of quantifiers, so we may assume that all first-order formulas are written using $\{\top,\wedge,\exists,\neg\}$. Apart from this, the definition of the first-order spectrum is completely analagous to that of the coherent spectrum.

\begin{defn}[First-order spectra]
Fix a language $\LL$ and a regular cardinal $\kappa\geq |\LL|+\aleph_0$. The first-order spectrum of the language $\MM^{fo}=\MM(\LL)^{fo}$ is defined by
\begin{itemize}
\item The underlying sets of $\MM^{fo}$ are the same as those of $\MM^{coh}$ (cf. definitions \ref{M0points} \& \ref{M1points}):
$$|\MM_0^{fo}|=|\MM_0^{coh}| \hspace{1cm}|\MM_1^{fo}|=|\MM_1^{coh}|$$
\item The topology of $\MM^{fo}$ is defined in the same fashion as that of $\MM^{coh}$ (cf. definitions \ref{M0opens} \& \ref{M1points}), except that the basic open sets $V_{\varphi(k)}$ range over all \emph{first-order} formulas $\varphi(x)$.
\item The spectrum of a first-order theory $\bTT$ is the full subgroupoid $\MM(\bTT)^{fo}\subseteq\MM(\LL)^{fo}$ consisting of labelled models $\mu$ which satisfy the axioms of $\bTT$.
\end{itemize}\end{defn}

The first-order and coherent spectra are nearly identical; they share exactly the same groupoid structure of labelled models and isomorphisms. However, the first-order topology is finer and, in particular, more separated. Recall that the Stone space for a Boolean propositional theory has a basis of clopen sets. Points of the Stone space are valuations of the algebra and these are completely disconnected in the topology. Here we have the related fact

\begin{lemma}
The connected components of $\MM_0(\LL)^{fo}$ correspond to complete theories in $\LL$.
\end{lemma}

\begin{proof}
For any collection of sentences $\Delta$ we let $V^\Delta=\bigcap_{\varphi\in\Delta} V_{\varphi}$. Every model $\mu$ defines a complete theory $\Delta_{\mu}=\{\varphi\ |\ M_{\mu}\models\varphi\}$. If $\Delta\not=\Gamma$ are distinct complete theories then there is some formula with $\varphi\in\Delta$ and $\neg\varphi\in\Gamma$, so $V^\Delta\cap V^\Gamma\subseteq V_\varphi\cap V_{\neg\varphi}=\emptyset$. Thus we have a partition
$$\MM_0(\LL)^{fo}=\coprod_{\Delta\rm{\ compl.}} V^\Delta.$$

It remains to see that each subset $V^\Delta$ is itself connected. This follows from the fact that complete theories satisfy the joint embedding property: if $M_1,M_2\models \Delta$ there is another $\Delta$-model $N$ and a pair of elementary embeddings $M_1\to N$ and $M_2\to N$. This is easy to see using the method of diagrams introduced in the next chapter (cf. proposition \ref{diag_models}). By Lowenheim-Skolem, we may assume $|N|=|M_1|+|M_2|$.

Now suppose that $\mu_1$ and $\mu_2$ are labelled models of $\Delta$. The labellings on $\mu_1$ and $\mu_2$ are infinite so we may restrict to sublabellings $\mu_1'$ and $\mu_2'$ (still infinite) which use disjoint sets of labels. Thus we have $\mu_1'\in\overline{\mu_1}$ and $\mu_2'\in\overline{\mu_2}$. By the previous observation regarding joint embedding, together with the disjointness of labels, we can find another labelled model $\nu$ with $\mu_1',\mu_2'\in\overline{\nu}$. Thus we have a chain
$$\overline{\mu_1} \ni \mu_1' \in \overline{\nu} \ni \mu_2' \in \overline{\mu_2}.$$

Now suppose that $V^\Delta$ has a clopen partition $V^\Delta= U_1+U_2$. If $\mu\in U_1$ then we also have the closure $\overline{\mu}\subseteq U_1$. Given the chain above, this implies that $\mu_1\in U_1$ if and only if $\mu_2\in U_1$. Hence $U_2=\emptyset$, and $V^\Delta$ is connected.

\end{proof}

In order to apply the coherent definability theorem \ref{coh_def} to the first-order case, we give a translation from classical to coherent logic. This is a process called Morleyization and a discussion can be found at \cite{elephant}, D1.5.13.

\begin{lemma}\label{bool_comp}
For each first-order theory $\bTT$ in a language $\LL$ there is a coherent theory $\bTT^*$, written in an extended language $\LL^*$, such that $\MM(\bTT)^{fo}\cong\MM(\bTT^*)^{coh}$.
\end{lemma}
\begin{proof}
First consider the case of first-order $\LL$-structures (i.e., $\bTT$ is the empty theory over $\LL$). We obtain the extended language as a union $\LL^*=\bigcup \LL_n$ where $\LL_0=\LL$. Let $\rm{Coh}(\LL_n)$ denote the set of coherent formulas in $\LL_n$ and define $\LL_{n+1}=\LL_n\cup\{N_\varphi(x)\ |\ \varphi(x)\in\rm{Coh}(\LL_n)\}$. Similarly, $\bTT^*=\bigcup \bTT_n$ where $\bTT_{n+1}$ is $\bTT_n$ together with the $\LL_{n+1}$-coherent axioms
$$\begin{array}{c}
\vdash \varphi(x)\vee N_\varphi(x)\\
\varphi(x)\wedge N_\varphi(x)\vdash \bot.\\
\end{array}$$
Of course, many of these new symbols are redundant; for example, one of de Morgan laws forces an equivalence $N_{\varphi\wedge\psi}\equiv N_{\varphi}\vee N_{\psi}$. However, there is little benefit in a more parsimonious approach.

Suppose that $M$ is an $\LL$-structure. Given the axioms in $\bTT_1$, each $\LL_1$-symbol $N_\varphi$ has a unique interpretation as the complement $N_\varphi^M=|M|^x\setminus\varphi^M$. As usual, these interpretations for $\LL_1$-symbols extend uniquely to an interpretation for any $\LL_1$-formula. Then each $\LL_2$-symbol has a unique interpretation satisfying $\bTT_2$ (as the complement of an $\LL_1$-formula), and so on.

This shows that each $\LL$-structure $M$ has a unique extension to a model $M^*\models\bTT^*$. As $\LL$ and $\LL^*$ have the same basic sorts we can also lift a labelling of $M$ to a labelling of $M^*$; this gives a map $\mu\mapsto \mu^*$ inducing a bijection
$$j:|\MM_0(\LL)^{fo}|\cong|\MM_0(\LL)^{coh}|\cong|\MM_0(\bTT^*)^{coh}|.$$

Every classical formula $\varphi(x)$ over $\LL$ is $\bTT^*$-equivalent to a coherent formula over $\LL^*$. This is easily proved by induction. If $\varphi(x)$, $\psi(x)$ and $\gamma(x,y)$ are all coherent over $\overline{\LL}$ then they are already coherent over some $\LL_n$. But then $(\varphi\wedge\psi)(x)$ and $\exists y.\gamma(x)$ are also coherent over $\LL_n$ and $\neg\varphi(x)\equiv N_{\varphi}$ is coherent over $\LL_{n+1}$.

From this it is easy to see that the bijection $j$ is actually a homeomorphism:
$$\mu\in V_{\neg\varphi(k)}\subseteq\MM_0(\LL)^{fo}\Iff \mu^*\in V_{N_\varphi(k)}\subseteq\MM_0(\bTT^*)^{coh}.$$ 
Similarly, an isomorphism $\alpha:M\cong N$ lifts uniquely to $\alpha^*:M^*\cong N^*$ and this induces a homeomorphism at the level of $\MM_1$. This establishes the claimed isomorphism of spectra $\MM(\LL)^{fo}\cong\MM\left(\bTT^*\right)^{coh}$

Now suppose that $\bTT$ is not empty. Each formula $\varphi\in\bTT$ corresponds to a coherent formula $\varphi^*$ over $\LL^*$, and we simply add these into $\bTT^*$. Then $\mu\models\varphi$ just in case $\mu^*\models\varphi^*$ and so $\MM(\bTT)^{fo}\cong\MM(\bTT^*)^{coh}$.
\end{proof}

As in the coherent case, one can proceed to define the a generic first-order model in the equivariant sheaves over $\MM^{fo}$. For each context $x:A$ there is a sheaf $\ext{A}^{fo}$ and for each formula $\varphi(x)$ there is a definable subsheaf $\ext{\varphi}^{fo}\subseteq \ext{A}^{fo}$. Moreover, the homeomorphism $j:\MM(\bTT)^{fo}\iso \MM(\bTT^*)^{coh}$ induces corresponding maps $j_A:\ext{A}^{fo}\iso\ext{A}^{coh}$ between these.

\begin{thm}[First-order definability]
Suppose that $x:A$ is a context and for every first-order labelled model $\mu$ we have a subset $S_{\mu}\subseteq|M_{\mu}|^x$. These subsets are first-order definable just in case the union $\bigcup_{\mu} S_{\mu}$ defines a compact equivariant subsheaf of $\ext{A}^{fo}$.
\end{thm}
\begin{proof}
Clearly the subsets $S_{\mu}$ are first-order definable just in case the corresponding subsets $S_{\mu}^*=j(S_{\mu})$ are coherently definable from $\overline{\bTT}$. By the coherent definability theorem, this is true just in case $j_A(S)$ is a compact equivariant subsheaf of $\ext{A}^{coh}$.

A homeomorphism preserves open sets, so $S$ is a subsheaf just in case $j_A(S)$ is. We have already observed that a first-order isomorphism $\mu\iso\nu$ is the same as a coherent isomorphism $\mu^*\iso\nu^*$ and these act on $\ext{A}^{fo}$ and $\ext{A}^{coh}$ in the same way. Thus $S$ is equivariant just in case $j_A(S)$ is. Finally, as $j_A$ preserves covers and equivariant subsheaves, it also preserves (equivariant) compactness. Thus 
\begin{tabular}[t]{rcl}
$S$ is definable &$\Iff$& $j_A(S)$ is definable\\
&$\Iff$& $j_A(S)$ is a compact equivariant subsheaf\\
&$\Iff$& $S$ is a compact equivariant subsheaf\\
\end{tabular}

\end{proof}

This provides a full solution to the problem of ``logicality'', the question of when a family of subsets $\{S_M\subset|M|\}_{M\models \bTT}$ is definable in first-order logic. This question goes back to Tarski, who proved (with Lindenbaum) that definable sets must be invariant under isomorphism \cite{LT1936}. However, this is clearly not a sufficient condition for (first-order) definability because, for example, infinitary quantifiers and connectives are also isomorphism stable. This is a question which he returned to throughout his career \cite{tarski1986}.

Most attempts to answer this question have centered around stability under some (typically) larger class of morphisms. For example, McGee has shown that the sets which are stable under isomorphisms are exactly those which are definable in $\LL_{\infty\infty}$ (i.e., definable by formulas with any (infinite) number of variables and arbitrarily large conjunctions \cite{mcgee}. Bonnay showed that a family of sets is stable under arbitrary homomorphism if and only if it is definable in $\LL_{\infty\infty}$ without equality \cite{bonnay}. Feferman has shown that sets are $\lambda$-definable from monadic operations if and only if they are definable in first-order logic without equality \cite{feferman}.

The foregoing discussion suggests a different approach, relying on topology rather that morphism invariance. We have given a precise topological characterization of definability in coherent logic and, via Morleyization, we may regard any classical first-order theory as a coherent theory. Together, this gives an exact characterization of definability in first-order logic.

Similar methods can also be applied to yield a topological definability theorem for intuitionistic logic, so long as we modify our spectra to range over Henkin models as well as the intended interpretations. To make sense of this we will need to employ the categorical logic discussed in the next chapter.

%% file: Ch2.tex
\chapter{Pretopos Logic}

In this chapter we will recall and prove a number of well-known facts about pretoposes, most of which can be found in either the Elephant \cite{elephant} or Makkai \& Reyes \cite{MakkaiReyes}. We recall the definitions of coherent categories and pretoposes and the pretopos completion which relates these two classes of categories. We go on to discuss a number of constructions on pretoposes, including the quotient-conservative factorization and the machinery of slice categories and localizations. In particular, we define the elementary diagram, a logical theory associated with any $\bTT$-model, and its interpretation as a colimit of pretoposes. These are Henkin theories: they satisfy existence and disjunction properties which can be regarded as a sort of ``locality'' for theories. We will also show that this machinery interacts well with complements, so that the same methods may be applied to the study of classical theories.

\section{Coherent logic and pretoposes}\label{sec_Ptop}

Categorical logic is founded upon two principles: (i) logical theories are structured categories and (ii) models and interpretations are structure-preserving functors. These premises were first developed by Lawvere in his PhD. thesis, \emph{Functorial Semantics of Algebraic Theories} \cite{lawvereFuncSem}. The first slogan means that the syntactic entities of a theory $\bTT$ (i.e., formulas and terms) can be arranged into a \emph{syntactic category} $\CC_{\bTT}$. Any $\bTT$-model then defines a functor $\CC_{\bTT}\to\Sets$ which is unique up to isomorphism. The ambient logic of $\bTT$ (e.g., classical or intuitionistic, propositional or first-order) corresponds to the relevant categorical structure which $\CC_{\bTT}$ possesses (and which its models preserve).

We will assume that the reader is familiar with the basic contours of functorial semantics, in particular the interpretation of finite limits as pairing and conjunction and regular factorization as existential quantification. These ideas were discussed briefly in the last chapter (\pageref{coh_logic}), and a more complete discussion can be found in \cite{MakkaiReyes}, \cite{SGL} or \cite{elephant}. In this chapter we will present the extension of these ideas to coherent categories and pretoposes. The material in this section and the next is well-known; a standard reference is Johnstone \cite{elephant} A1. 

A category $\CC$ is \emph{coherent} if it has (i) finite limits, (ii) regular epi-mono factorization and (iii) finite joins of subobjects and, moreover, the latter two constructions should be stable under pullbacks. Coherent structure is sufficient to interpret coherent theories; this structure is exactly that which was required, in the last chapter, for our interpretation of the generic model in $\Sh(\MM_0)$ (page \pageref{coh_logic}).

Now fix a (multi-sorted) first-order language $\LL$. An \emph{interpretation} $I$ of $\LL$ in a coherent category $\bEE$ begins with an underlying object $A^I\in\bEE$ for each basic sort $A\in\LL$. A compound context $B=\<A_1,\ldots,A_n\>$ is interpreted by the product $B^I=\prod_i A_i^I$.

To each basic relation $R(x)$ (in a context $x:A$) we assign a subobject $R^I\leq A^I$; similarly, we assign an arrow $f^I:A^I\to B^I$ each function symbol $f:x\to y$ (where $y:B$). Given this data, we can construct an interpretation of any coherent formula $\varphi(x)$ as a subobject $\varphi^I\leq A^I$. The construction proceeds inductively based on the formulaic structure of $\varphi(x)$, where each logical operation is interpreted as on page \pageref{coh_logic}. Given such an interpretation, we asy that $I$ \emph{satisfies} a sequent $\varphi(x)\overprove{x:A}\psi(x)$ just in case $\varphi^I\leq\psi^I$ in $\Sub_{\bEE}(A^I)$. 

Each coherent theory $\bTT$ determines a \emph{syntactic category} $\bEE_{\bTT}$ which is itself coherent. For each context $x:A$ there is an object $A^x\in\EE_{\bTT}$ and these are related inductively by
$$A^{\<x,y\>}\cong A^x\times A^y\hspace{1cm} A^{\<\>}\cong 1_{\EE_{\bTT}}.$$
Every formula $\varphi(x)$ determines a subobject $[x|\varphi]\leq A^x$ and we have $[x|\varphi]\leq[x|\psi]$ if and only if $\bTT$ proves the corresponding sequent, which we notate
$$\bEE_{\bTT}\tri \varphi(x)\vdash_{x:A} \psi(x).$$
A morphism $\sigma:[\varphi(x)]\to[\psi(y)]$ is a formula $\sigma(x,y)$ which is \emph{provably functional} (p.f.) on $\varphi$:
$$\begin{array}{c}
\sigma(x,y)\vdash \varphi(x)\wedge\psi(y)\\
\sigma(x,y)\wedge\sigma(x,y')\vdash y=y'\\
\varphi(x)\vdash \exists y. \sigma(x,y).\\
\end{array}$$
When $\sigma$ is provably functional, we will often write $\sigma(x)=y$ in place of $\sigma(x,y)$.

If $I$ in an interpretation of $\bTT$ in a coherent category $\bFF$, this induces a coherent functor $\widetilde{I}:\bEE_{\bTT}\to\bFF$ by sending $[x|\varphi]\mapsto \varphi^I$. One can also define $\LL$-structure homomorphisms in $\bFF$ and these correspond to natural transformations $\widetilde{I_0}\to\widetilde{I_1}$. This defines an equivalence of categories (natural in $\FF$)
$$\bTT\Mod(\bFF)\sim\Coh(\bEE_{\bTT},\bFF)$$
and this classification property fixes $\bEE_{\bTT}$ up to equivalence. In particular, a classical model of $\bTT$ corresponds to a coherent functor $M:\bEE_{\bTT}\to\Sets$.

\begin{defn}
An initial object $0\in\bEE$ is called \emph{strict} if, for any other object $A\not\cong0$, $\Hom_{\bEE}(A,0)=\emptyset$.

Suppose that $\bEE$ has a strict initial object. A coproduct $A+B\in\bEE$ is called \emph{disjoint} if the pullback of the coprojections is 0:
$$\xymatrix{
0 \pbcorner \ar[r] \ar[d] & A \ar[d]^{i_A}\\
B \ar[r]_{i_B} & A+B.}$$

A left-exact category with a strict initial object and all disjoint coproducts, both stable under pullbacks, is called \emph{extensive}.
\end{defn}

\begin{defn}
A subobject $R\leq A\times A$ in $\bEE$ is called an \emph{equivalence relation} or \emph{congruence} if it satisfies diagrammatic versions of reflexivity, transitivity and symmetry axioms: 

\begin{tabular}{lll}
\textbf{Refl.} & \textbf{Trans.} & \textbf{Sym.}\\
$\xymatrix{A \ar@{-->}[r] \ar@{>->}[rd]_{\Delta} & R \ar@{>->}[d] &\\& A\times A\\}$&
$\xymatrix{\raisebox{-3.5ex}{$R\overtimes{A}R$} \ar@{-->}[r] \ar[rd]_{\<p_1,p_3\>} & R \ar@{>->}[d] \\& A\times A\\}$&
$\xymatrix@R=5ex{ R \ar@{-->}[r] \ar@{>->}[rd] &A\times A \ar[d]^{\rsim}_{\rm{tw}}\\& A\times A}$
\end{tabular}

An object $Q=A/R$ is the \emph{(exact) quotient} of $A$ by $R$ if there is a coequalizer $e:A\epi Q$ (necessarily a regular epimorphism) such that $R$ is the kernel pair of $e$:
$$\xymatrix{R \pbcorner \ar[r] \ar[d] & A \ar[d]^e \\ A \ar[r]_e & Q.\\}$$

If a left-exact category $\bEE$ has exact quotients for all equivalence relations and these are stable under pullback then $\bEE$ is called an \emph{exact category}.
\end{defn}

\begin{defn}
A \emph{pretopos} is a category which is both extensive and exact.
\end{defn}

Note that a pretopos is automatically regular and coherent. The first claim follows from the fact that each kernel pair is an equivalence relation, so that an exact category already has a quotient for every kernel pair. For the second claim, simply note that the join of two subobjects $R,S\leq A$ can be computed as the epi-mono factorization $R+S \epi R\vee S \mono A$. A coherent functor automatically preserves any disjoint coproducts and quotients in its domain, so pretopos functor is nothing more than a coherent functor between pretoposes.

\begin{lemma}\label{ptop_nice}
If $\EE$ is a pretopos, then
\begin{enumerate}
\item every monic in $\EE$ is regular (an equalizer).
\item $\EE$ is balanced: epi + mono $\To$ iso.
\item every epic in $\EE$ is regular (a coequalizer).
\end{enumerate}
Furthermore, for any epimorphism $f:A\epi B$,
\begin{enumerate}
\item[4.] $f^*:\Sub(B)\to\Sub(A)$ is injective.
\item[5.] $\exists_f\circ f^*$ is the identity on $\Sub(B)$.
\end{enumerate}
\end{lemma}

\begin{proof}
Suppose that $m:R\mono A$ is monic. This induces an equivalence relation $E_R\mono A+A$ and its quotient is a pushout: $Q\cong A\overplus{R} A$. Moreover, the coprojections $A\rightrightarrows Q$ are again monic and their intersection is $R$. Consequently, $R\mono A\rightrightarrows Q$ is an equalizer diagram and $m$ is regular.

Now suppose that $m$ is both epic and monic. Because it is epic and equalizes the coprojections $A\rightrightarrows Q$, these coprojections must be equal. But then $m$ is the equalizer of identical maps, and therefore an isomorphism. Hence $\EE$ is balanced.

Now suppose $f:A\to B$ is epic. Factor $f$ as a regular epi $e$ followed by a monic $m$. Since $f$ is epic, $m$ is as well. But then balance implies that $m$ is an isomorphism. Since $e$ and $f$ are isomorphic maps and $e$ is regular, $f$ too.

Note that $f^*$ injective follows immediately from the identity $\exists_f\circ f^*=1_{\Sub(B)}$. Because regular epis are stable under pullback, any subobject $S\leq B$ induces a pullback square in which the left-hand side is a reg epi-mono factorization
$$\xymatrix{
f^*S \pbcorner \ar@{>->}[r] \ar@{->>}[d] & A \ar@{->>}[d]^f\\
S \ar@{>->}[r] & B.\\
}$$
This is exactly the definition of $\exists_f$, so uniqueness of factorizations guarantees that $S\cong \exists_f(f^*S)$.\end{proof}

\begin{defn}\label{ptop_morph_props}
Suppose that $I:\bEE\to\bFF$ is a coherent functor we say that $I$ is:
\begin{itemize}
\item \emph{conservative} if it reflects isomorphisms.
\item \emph{full on subobjects} if, for every object $A\in\EE$ and every subobject $S\leq IA\in \FF$ there is a subobject $R\leq A$ with $IR\cong S$.
\item \emph{subcovering} if for every object $B\in\FF$ there is an object $A\in\EE$, a subobject $S\leq IA\in\FF$ and a (regular) epimorphism $S\epi B$.
\end{itemize}
\end{defn}

\begin{lemma}[Pretopos completion]\label{ptop_comp}
The forgetful functor $\Ptop\to\Coh$ has a left adjoint sending each coherent category $\bEE$ to its \emph{pretopos completion} $\EE$. Moreover, the unit is a coherent functor $I:\bEE\to\EE$ which is conservative, full, full on subobjects and (finitely) subcovering.\footnote{Because $\bEE$ may not have coproducts, we are allowed to cover each $B\in\EE$ by a finite family of subobjects $S_i\leq IA_i$.}
\end{lemma}
\begin{proof}[Sketch]
Please see \cite{elephant}, A3 or \cite{shulman_completion} for a more detailed presentation.

A \emph{finite-object equivalence relation} $A_i/R$ in $\mathbb{E}$ is a family of objects $\<A_i\>_{i\leq n}$ together with binary relations $R_{ij}\mono A_i\times A_j$ satisfying

\begin{tabular}{cc}
\textbf{Refl.} & $\xymatrix{A_i \ar@/^3ex/[rr]^{\Delta} \ar@{-->}[r] & R_{ii} \ar@{>->}[r] & A_i\times A_i,}$\\\\
\textbf{Trans.} & $\xymatrix{\raisebox{-3.5ex}{$R_{ij}\overtimes{[\varphi_j]} R_{jk}$} \ar@/^3ex/[rr]^{\<p_i,p_k\>} \ar@{-->}[r] & R_{ik} \ar@{>->}[r] & A_i\times A_k,}$\\\\
\textbf{Sym.} & \raisebox{.7cm}{$\xymatrix@R=5ex{ R_{ij} \ar@{=}[d]^{\rsim} \ar@{>->}[r] &A_i\times A_j \ar@{=}[d]^{\rsim}_{\rm{tw}}\\
			      R_{ji} \ar@{>->}[r] & A_j\times A_i}$}.\\
\end{tabular}

The objects of $\EE$ are the finite-object equivalence relations in $\bEE$. Given two such objects $A_i/R$ and $B_k/S$ an arrow $f:A_i/R\to B_j/S$ is a family of relations $F_{ik}\mono A_i\times B_k$ which is provably functional modulo the equivalences $R$ and $S$. For example, the condition $\sigma(x,y)\wedge\sigma(x,y')\vdash y=y'$ becomes the following requirement: for all $Z\in\bEE$ and all $x_i:Z\to A_i$ and $y_k:Z\to B_k$ we have
$$\begin{array}{c}
\left(\raisebox{.8cm}{\xymatrix{
Z \ar[r]^-{\<x_i,x_j\>} \ar@{-->}[rd] & A_i\times A_j & Z \ar[r]^-{\<x_i,y_k\>} \ar@{-->}[rd] & A_i\times B_k & Z \ar[r]^-{\<x_j,y_l\>} \ar@{-->}[rd] & A_j\times B_l \\ 
& R_{ij} \ar[u] && F_{ik} \ar[u] && F_{jl} \ar[u]\\
}}\right)\\
\xymatrix{\ar@{}[rd]|{\mbox{\huge{$\To$}}} && Z \ar[r]^-{\<y_k,y_l\>} \ar@{-->}[rd]& B_k\times B_l\\
 &&& S_{kl} \ar[u] \\
}\\
\rm{i.e.,}\ \big(R_{ij}(x_i,x_j)\wedge F_{ik}(x_i,y_k)\wedge F_{jl}(x_j,y_l) \big)\To S_{jl}(y_j,y_l)
\end{array}
$$

Each object $A\in\bEE$ defines a trivial quotient $A/\!\!=$ and the comparison functor $I:\bEE\to\EE$ acts by sending $A\mapsto A/\!\!=$. Two maps are equal modulo the identity relation just in case they are provably equal, so this displays $\bEE$ as a full subcategory inside $\EE$. It follows at once that $I$ is conservative. Henceforth we will not distinguish between an object $A\in\bEE$ and its image $IA\in\EE$.

Suppose $A_i$ is a family of objects in $\bEE$. Define an equivalence relation $\Delta_{ij}=0$ if $i\not=j$ and $\Delta_{ii}=\Delta_{A_i}$. In $\EE$ this defines a (disjoint) coproduct $\coprod_i A_i=A_i/\Delta_{ij}$. From these, any binary coproduct $(A_i/R)+(B_j/S)$ can be defined as a quotient of $\coprod_i A_i+\coprod_j B_j$, so $\EE$ is closed under $+$. Moreover, any object $A_i/R$ comes equipped with a family $A_i\in\bEE$ and a presentation $\coprod_i A_i\epi A_i/R$, so $I$ is finitely subcovering.

To see that $I$ is full on subobjects, suppose that $A, B_i\in\bEE$ and $B_i/R\leq A$. Each composite $B_i\to B_i/R\mono A$ belongs to the full subcategory $\bEE$, so its epi-mono factorization $B_i\epi \im(B_i) \mono A$ does as well. Since $B_i/R$ is the epi-mono factorization of a map $\coprod_i B_i\to A$ it is isomorphic to $\bigvee_i \im(B_i)$, which again belongs to $\bEE$.
\end{proof}

\begin{defn}\label{class_ptop}
The \emph{classifying pretopos} $\EE_{\bTT}$ of a coherent theory $\bTT$ is the pretopos completion of the syntactic category $\bEE_{\bTT}$. We will say that $\bTT$ is a \emph{pretopos theory} if $I:\bEE_{\bTT}\to\EE_{\bTT}$ is an equivalence of categories.
\end{defn}

Every pretopos classifies a theory $\bTT_{\EE}$. To form this theory, we let every object $A\in\EE$ determines a basic sort, every subobject $S\leq A$ a basic relation and every arrow $f:A\to B$ a function symbol. The axioms of $\bTT_{\EE}$ are defined by the subobject ordering in $\EE$: $R\underset{\bTT_{\EE}}{\vdash} S \Iff R\leq S\in\Sub_{\EE}(A)$. If $\EE=\EE_{\bTT}$ is already the classifying category for some coherent theory $\bTT$, then the extension $\bTT\subseteq\bTT_{\EE}$ is essentially the same as Shelah's model-theoretic closure $\bTT\subseteq\bTT^{\rm{eq}}$ (see e.g. \cite{harnik} for a discussion).

Because coproducts and quotients are definable in $\Sets$, every $\bTT$-model $M$ has an essentially unique extension to a $\bTT^{\rm{eq}}$ model $M^{\rm{eq}}$. Category-theoretically, this amounts to an equivalence of categories
$$\Coh(\bEE_{\bTT},\Sets)\simeq\Ptop(\EE_{\bTT},\Sets)$$
and this follows immediately from the fact that pretopos completion is left adjoint to the forgetful functor $\Ptop\to\Coh$. Therefore it is reasonable to regard $\bTT$ and $\bTT^{\eq}$ as the ``same'' theory in the sense that they are semantically equivalent.

$\bTT^\rm{eq}$ is a conservative extension which allows for a syntactic operation called elimination of imaginaries. An \emph{imaginary element} $a/E$ in a model $M$ is a definable equivalence relation $E(x,x')$ (where $x,x':A$) together with an equivalence class $[a]\in A^M/E^M$. $M$ has \emph{uniform elimination of imaginaries} (u.e.i) if for every equivalence relation $E(x,x')$ there is a sort $y:B$ and a p.f. formula $\epsilon(x,y)$ such that
$$M\ \models\ \  E(a,a') \ \dashv\vdash\ \exists y.\big(\epsilon(a)=y=\epsilon(a')\big).$$

Pretopos structure supports elimination of imaginaries in the sense that every $\bTT^{\rm{eq}}$-model has u.e.i. Given a definable equivalence relation $E(x,x')$, simply take $B=A/E$ and let $\epsilon(x,y)$ be the p.f. relation corresponding to the quotient $A\epi A/E$.

When we study categorical semantics we sometimes fix a pretopos $\SS$ (often $\SS=\Sets$) to serve as a ``semantic universe''. In that case, we distinguish between \emph{interpretations} $I:\EE\to \FF$ (between arbitrary pretoposes) and \emph{models} $M:\EE\to\SS$ (into $\SS$). We let $A^M$ denote the value of a model $M$ at an object $A\in\EE$. We (loosely) refer to $A^M$ as a definable ``set'' in $M$ and write $a\in A^M$ as a shorthand for $a\in\Hom_{\SS}(1,A^M)$.

In this terminology, interpretations act on models by precomposition (contravariantly):
$$\xymatrix@R=2ex{
\FF\Mod(\SS) \ar[rr]^{I^*} && \EE\Mod(\SS)\\
N \ar@{|->}[rr] && I^*N\\
\FF \ar[rdd]_N && \EE \ar[ll]_I \ar[ldd]^{I^*N}\\\\
&\SS&\\
}$$
For each $A\in\EE$, $A^{I^*N}=(IA)^N$ and we call $I^*N$ the \emph{reduct} of $N$ along $I$. This generalizes the classical terminology; if $I$ is induced by an inclusion of theories $\bTT\subseteq\bTT'$ and $N\models\bTT'$, $I^*N\cong N\!\!\upharpoonright\!\LL(\bTT)$.

From this point forward we will work in the doctrine of pretoposes and define our logical terminology accordingly: ``theory'' is synonymous with ``pretopos'', a formula is an object and a model is a pretopos functor to $\Sets$. Please see the table on page \pageref{dictionary} for a full list of categorical definitions and their logical equivalents.

\section{Factorization in $\Ptop$}

In this section we review the quotient-conservative factorization system in $\Ptop$. A discussion of this factorization system can be found in \cite{makkai}. Given a pretopos $\EE$, a sequent in $\EE$ is an object $A\in \EE$ (a context) together with an (ordered) pair of subobjects $\varphi,\psi\in\Sub(A)$. We say that $\EE$ satisfies $\varphi\vdash\psi$ just in case $\varphi\leq \psi$.

An interpretation $I:\EE\to\FF$ is conservative in the logical sense if $\EE$ satisfies any sequents which $I$ forces in $\FF$
$$\FF\tri I(\varphi)\vdash I(\psi) \To \EE\tri \varphi\vdash\psi.$$
Categorically, this says that $I$ is \emph{injective on subobjects}: $R\lneq S\leq A$ implies $IR\lneq IS$. We will see below that, in a pretopos, this is equivalent conservativity in the categorical sense.

\begin{lemma}\label{eq_cons}
For a pretopos functor $I:\EE\to\FF$, the following are equivalent:
\begin{itemize}
\item[(i)] $I$ is conservative (reflects isomorphisms).
\item[(ii)] $I$ is injective on subobjects.
\item[(iii)] $I$ is faithful.
\end{itemize}\end{lemma}
\begin{proof}

The implication (i)$\Rightarrow$(ii) is immediate. Suppose that $I$ identifies two subobjects $IR \cong IS\in\Sub_{\FF}(IA)$; then $IR\cong I(R\wedge S)\cong IS$ and, since $I$ reflects both these isomorphisms, $R\cong S$ already in $\EE$. To see that (ii)$\Rightarrow$(iii), suppose that $I$ is injective on subobjects and that $f\not=g:A\rightrightarrows B\in\CC$. Then $\Eq(f,g)\lneq A$, which implies that
$$\Eq(If,Ig)=I(\Eq(f,g))\lneq IA.$$
But then $If\not=Ig$, so $I$ is faithful.

For (iii)$\Rightarrow$(i), suppose that $I$ is faithful and that $If$ is an isomorphism. $If$ is monic and epic, and both the properties are reflected by faithful functors. For example, if $If$ is monic and $f\circ g= f\circ h$ then $If\circ Ig=If\circ Ih$, whence $Ig=Ih$ by monicity and $g=h$ by faithfulness. Thus $f$ is also monic and epic, and so by balance (lemma \ref{ptop_nice}) $f$ is an isomorphism.
\end{proof}

Now we introduce a second class of pretopos functors, the quotients. Suppose that $\bTT\subseteq\bTT'$ is an extension of theories written in the same language. Let $\EE$ and $\EE'$ denote the classifying pretoposes of $\bTT$ and $\bTT'$, respectively. From the pretopos completion we know that each object of $\EE$ is a $\bTT$-definable equivalence relation, and therefore also $\bTT'$-definable. Similarly, $\EE$-arrows are $\bTT$-provably functional relations (modulo equivalence), and these are also $\bTT'$-provably functional.

This means that the inclusion $\bTT\subseteq\bTT'$ induces an interpretation $I:\EE\to\EE'$. Since equality of arrows in $\EE'$ is $\bTT'$-provable equality, $I$ may not be faithful. Similarly, the new axioms in $\bTT'$ may introduce provable functionality or equivalence, so in general $I$ is neither full nor surjective on objects. However, $I$ does satisfy some related surjectivity conditions defined in the last section.

\begin{lemma}
If $\bTT\subseteq\bTT'$ is an extension of theories in the same language $\LL$ then the induced functor $I:\EE\to \EE'$ is subcovering and full on subobjects.
\end{lemma}
\begin{proof}
Suppose that $Q$ is an object of $\EE'$. By the pretopos completion we know that $Q$ has the form $Q=B_i/E$, where $B_i=\varphi_i(x_i)$ is a formula in context $x_i:A_i$ and $E$ is a $\bTT'$-provable finite-object equivalence relation. Since $\bTT$ and $\bTT'$ share the same language, $A_i=IA_i$ already belongs to $\EE$ and the presentation $\coprod_i IA_i\geq \coprod_i B_i \epi Q$ shows that $I$ is subcovering.

Now suppose $A=A_i/E$ is an object in $\EE$ and $S\leq IA$. Then $S=S_i/IE$ for some family of subobjects $S_i\leq IA_i$. Each $S_i$ corresponds to a $\bTT'$-formula $\varphi_i(x_i)$ in context $x_i:A_i$. Since $\bTT$ and $\bTT'$ have the same language, $\varphi_i(x_i)$ also defines a subobject $R_i\leq A_i$ and $I(R_i/E)\cong S_i/IE$. Therefore $I$ is full on subobjects.
\end{proof}

\begin{defn}
A pretopos functor $I:\EE\to\FF$ is called a \emph{quotient} if it is both subcovering and full on subobjects (cf. definition \ref{ptop_morph_props}). In particular, for every $B\in\FF$ there is an $A\in\EE$ and an epimorphism $IA\epi B$.
\end{defn}

\begin{prop}\label{ptop_factor}
Every pretopos functor $I:\EE\to\FF$ has an factorization into a quotient map followed by a conservative functor:
$$\xymatrix@R=3ex{
\EE \ar[rr]^I \ar@{->>}[rd]_{Q} && \FF\\
& \EE\!/I \ar@{>->}[ur]_C &\\
}$$
\end{prop}

\begin{proof}
Begin by factoring $I$ as a composite $\EE\labelarrow{Q_0}\bGG\labelarrow{C_0}\FF$, where $Q_0$ is essentially surjective on objects and $C_0$ is full and faithful. $\bGG$ is not a pretopos but it is coherent, and its pretopos completion is our desired factorization
$$\xymatrix{
\EE \ar[rrr]^I \ar[rd]_{Q_0} \ar[rrd]^{Q} &&& \FF\\
&\bGG \ar[r]_J & \GG \ar[ur]_C\\
}$$

Since $Q_0$ is e.s.o., it is both subcovering and full on subobjects. $J$ satisfies the same properties by lemma \ref{ptop_comp}. Since both classes of functors are closed under composition, $Q$ must be a quotient.

Note that if $S\lneq B$ is a proper subobject and $p:A\epi B$ is a (regular) epi then $p^*S\lneq A$ is again proper. This is because the composite $p^*S\to S\mono A$ is \emph{not} epic (because $S$ is proper) and this is equal to the composite $p^*S\mono B\epi A$. If $p^*S\iso B$ were not proper then the composite would be epic, in contradiction to the previous observation.

Suppose $S\leq B$ is a subobject in $\GG$. Because $Q$ is a quotient, there is an object $A\in \EE$ and an epimorphism $p:QA\epi B$. Moreover, the pullback $p^*S\cong S\overtimes{B} QA$ is a subobject of $QA$ and therefore also in the image of in $Q$: $\exists R \leq A$ with $Q(R)\cong S\overtimes{B} QA$.

In order to see that $C$ is conservative, suppose $CS\cong CB$. Since $C$ preserves pullbacks and $C\circ Q\cong I$, this implies
$$IR\cong C(S\overtimes{B} QA)\cong CS\overtimes{CB} C(QA)\cong CB\overtimes{CB} IA \cong IA.$$
As $C_0$ is full and faithful, it follows that $Q_0 R\cong Q_0 A$ in $\bGG$ (and hence $QR\cong QA$ in $\GG$). But $p:QA\epi B$ is an epimorphism, so $p^*$ is injective on subobjects. Thus $QR\cong p^*S\cong QA$ implies $S\cong B$, so that $C$ is injective on subobjects and hence conservative.
\end{proof}

\begin{prop}\label{qc_ortho}
Quotient and conservative pretopos functors are orthogonal. I.e., for any diagram of pretopos functors as below with $C$ conservative and $Q$ a quotient there is a diagonal filler
$$\xymatrix{
\EE \ar@{->>}[d]_Q \ar[rr]^I && \GG \ar@{>->}[d]^{C}\\
\FF \ar[rr]_J \ar@{-->}[urr]^{D} && \HH.\\
}$$
The factorization is unique up to natural isomorphism.
\end{prop}

\begin{proof}[Sketch]
Consider an object $B\in\FF$. Since $Q$ is a quotient, there is an $A\in \EE$ together with an epimorphism $e:QA\epi B$. The kernel pair of $e$ is a subobject of $Q(A\times A)$ and so has a preimage $K\mono A\times A$ in $\EE$. Applying $J$ to $B$ and observing that $JQ\cong CI$ gives us a coequalizer diagram in $\HH$:
$$CI(K)\rightrightarrows CI(A) \epi J(B).$$

Parallel arrows $k_1,k_2:K\rightrightarrows A$ define an equivalence relation just in case certain monics defined from $k_1$ and $k_2$ are, in fact, isomorphisms. For example, the reflexivity and transitivity conditions (where $\tau:A\times A\iso A\times A$ the twist isomorphism) can be expressed by requiring that the marked arrows in the following pullback diagrams are isos:
$$\xymatrix{
A\wedge K \ar@{>->}[d] \ar@{>->}[r]^-{\sim} & A \ar@{>->}[d] && \tau(K)\wedge K \ar@{>->}[d]^{\vertsim} \ar@{>->}[r]^-{\sim} & \tau(K) \ar@{>->}[d] \\
K \ar@{>->}[r] & A\times A && K \ar@{>->}[r] & A\times A\\
}$$
A more complicated diagram involving subobjects of $A\times A\times A$ expresses transitivity in much the same way. Since a conservative morphism reflects isomorphisms, it must also reflect equivalence relations.

Since $C$ is conservative, this shows that $IK\rightrightarrows IA$ is already an equivalence relation in $\GG$, and we set $D(B)$ equal to its quotient. Arrows can be handled similarly. Each morphism $b:B\to B'$ induces a subobject $S_b\mono IA\times IA'$ which is provably functional modulo the equivalence relations defining $B$ and $B'$. This has a preimage $R_b\in \EE$ and the property of being provably functional modulo equivalence is reflected by conservative morphisms, so that $I(R_b)$ induces a morphism $D(B)\to D(B')$ as desired.

Essential uniqueness of the functor follows from the fact that $D$ must preserve quotients; if the kernel pair $QK\rightrightarrows QA$ maps to the kernel pair $IK\rightrightarrows IA$ (so that $D\circ Q\cong I$), then $D$ must send $B$ to the quotient of $IK$, as above. A similar argument applies to arrows. For further details, please see Makkai \cite{makkai_ultra}.
\end{proof}

\begin{cor}
The factorization in lemma \ref{ptop_factor} is unique up to equivalence in the category $\EE\!/I$ and also functorial: for any square in $\Ptop$ there is a factorization (unique up to natural isomorphism)
$$\xymatrix{
\EE \ar@/^3ex/[rr]^I \ar@{->>}[r] \ar[d] & \EE\!/I \ar@{>->}[r] \ar@{-->}[d] & \FF \ar[d]\\
\EE' \ar@/_3ex/[rr]_J \ar@{->>}[r] & \EE'/J \ar@{>->}[r] & \FF'\\
}$$
\end{cor}

\begin{proof}
First suppose that we have two quotient-conservative factorizations of $I$, through $\GG$ and $\GG'$. By the previous proposition, there are essentially unique diagonal fillers
$$\xymatrix{
\EE \ar@{->>}[d] \ar@{->>}[rr] && \GG \ar@{>->}[d] \ar@{-->}@/_/[lld]_D \\
\GG' \ar@{>->}[rr] \ar@/_/@{-->}[rru]_{D'} && \FF\\
}$$
The composite $D'\circ D$ is a diagonal filler between $\GG$ and itself so, by uniqueness of these diagonals, it must be naturally isomorphic to $1_{\GG}$. Similarly, $D\circ D'\cong 1_{\GG'}$, so $\GG$ and $\GG'$ are equivalent categories.

In order to establish functoriality, simply observe that there is a diagonal filler
$$\xymatrix{
\EE \ar@{->>}[d] \ar[r] & \EE' \ar@{->>}[r] & \EE'/J \ar@{>->}[d] \\
\EE\!/I \ar@{>->}[r] \ar@{-->}[urr] & \FF \ar[r] & \FF'.\\
}$$

\end{proof}

It is worth noting that the (coherent) sheaf functor sends the quotient/conservative factorization on pretoposes to the surjection/embedding factorization of geometric morphisms.
$$\xymatrix{
\EE \ar@{->>}[r]&\EE/I \ar@{>->}[r] & \FF& \rm{in\ }\Ptop\\
\Sh(\EE) & \ar@{>->}[l] \Sh(\EE/I) & \ar@{->>}[l] \Sh(\FF)& \rm{in\ }\Top\\
}$$
See \cite{SGL} VII.4 for details of this construction in $\Top$.

\section{Semantics, slices and localization}

In this section we translate some basic model-theoretic concepts into the context of pretoposes. Specifically, we will consider certain \emph{model-theoretic} extensions $\bTT\subseteq\bTT'$ which involve new constants and axioms, but do not involve new sorts, functions or relations. Although some of the specific definitions may be new, the material is well-known and folklore. Most can be found, at least implicitly, in Makkai \& Reyes \cite{MakkaiReyes}.

Fix a pretopos $\EE$ and an object $A\in\EE$. Observe that the slice category $\EE\!/A$ is again a pretopos. Most of the pretopos structure--pullbacks, sums and quotients--is created by the forgetful functor $\EE\!/A\to\EE$. The remaining structure, products, exist because the identity $1_A$ is terminal in $\EE\!/A$. Moreover, the forgetful functor has a right adjoint $A^\times: \EE\to\EE\!/A$ sending an object $B$ to the second projection $B\times A\to A$. As a right adjoint $A^\times$ preserves limits and the definition of pretoposes assumes pullback-stability for sums and quotients. Therefore $A^\times$ is a pretopos functor.

\begin{lemma}
$\EE\!/A$ classifies $A$-elements. More specifically, there is an equivalence of categories
$$\Ptop(\EE\!/A,\SS)\cong \Sigma_{M\models\EE} \Hom_{\SS}(1,A^M).$$
\end{lemma}
\begin{proof}
An object of the latter category in the statement of the lemma is a pair $\<M,a\>$ where $M$ is a model $\EE\to\SS$ and $a$ is a (global) element of the $\EE$-definable set $A^M$. A morphism $\<M,a\>\to\<M',a'\>$ is an $\EE$-model homomorphism $h:M\to M'$ such that $h_A(a)=a'$.

Suppose that $N:\EE\!/A\to\SS$; we recover $M$ from $N$ by taking the reduct along $A^\times$: $\varphi^M\cong (\varphi\times A)^N$. To recover $a$, note that the diagonal is a global section in $\EE\!/A$
$$\xymatrix{ \Delta_A:1_{\EE\!/A} \cong A \ar[r] & A^\times(A)}.$$ 
Therefore its interpretation in $N$ is a global element $(\Delta_A)^N:1\lto (A\times A)^N=A^M$. Similarly, we can recover an $\EE$-model homomorphism by composition $\xymatrix{\EE \ar[r]^{A^\times} & \EE\!/A \rtwocell^M_N{h} & \SS}$. The following naturality square shows that $h(a)=a'$:
$$\xymatrix{
1 \ar@{=}[d] \ar[rr]^{a'=(\Delta_A)^{M'}} && A^{M'} \ar[d]^{h_A} \\
1 \ar[rr]_{a=(\Delta_A)^M} && A^M.\\
}$$

Conversely, given $M$ and $a$ we may define $N:\EE\!/A\to\SS$ by sending each map $\sigma:E\to A$ to the fiber of $\sigma^M$ over $a$. The same construction applies to morphisms: when $h(a)=a'$ the map $h_E$ restricts to the fiber $\sigma^N$:
$$\xymatrix@=2ex{
\sigma^{N'} \ar[rr] \ar[dd] \ar@{-->}[rd] &&E^{M'} \ar[dd]|\hole \ar[rd]^{h_E}\\
&\sigma^N \pbcorner \ar[rr] \ar[dd] && E^M \ar[dd]^{\sigma^M}\\
1 \ar@{=}[rd] \ar[rr]^(.3){a'}|\hole && A^{M'} \ar[rd]^{h_A}\\
&1 \ar[rr]_a && A^M.\\
}$$

We leave it to the reader to check that these maps define an equivalence of categories.
\end{proof}

More generally, the pullback along any arrow $\tau:A\to B$ is a pretopos functor. If $N:\EE\!/A\to\SS$ classifies an element $a\in A^M$ then the composite $N\circ\tau^*$ classifies the element $\tau^M(a)\in B^M$. In particular, an element $a\in A^M$ satisfies a formula $\varphi\mono A$ just in case the associated functor $\<M,a\>:\EE\!/A\to \SS$ factors through the pullback $\EE\!/A\to \EE\!/\varphi$.

A (model-theoretic) \emph{type} in context $A$ is a filter of formulas $p\subseteq\Sub_{\EE}(A)$. This is precisely a set of mutually consistent formulas $\varphi(x)$ in a common context $x:A$. Given an element $a\in A^M$ we write $a\models \varphi$ if $M\models\varphi(a)$. In that case, there is an essentially unique factorization of $\<M,a\>$ through $\EE\!/\varphi$.
$$\xymatrix{
\EE\!/A \ar[rr]^{\<M,a\>} \ar[d]_{\varphi\vdash A} && \SS\\
\EE\!/\varphi \ar@{-->}[urr]_{a\models\varphi} &\\
}$$

Similarly, we write $a\models p$ if $a\models \varphi$ for every $\varphi\in p$. This induces a functor out of the filtered colimit
$$\xymatrix{
\EE\!/A \ar[r]^{\<M,a\>} \ar[d]_{\varphi\vdash A} & \SS\\
\EE\!/\varphi \ar@{-->}[ur] \ar[r]_{\psi\vdash\varphi} & \EE\!/\psi \ar[r]_-{p\vdash\psi} \ar@{-->}[u] & \colim_{\varphi\in p} E/\varphi. \ar@{-->}[ul] \\
}$$
It follows that we can define the classifying pretopos for $p$-elements by taking a directed (2-)colimit\footnote{See Lack \cite{lack} for the definition of (pseudo-)colimits in a 2-category.}
$$\EE_p\simeq\colim_{\varphi\in p} \EE\!/\varphi.$$
We will see in a moment that $\EE_D$ is a pretopos. Although $\EE_D$ is only defined up to equivalence of categories, the same is true for the classifying pretopos of a theory.

More generally we have the following definition:

\begin{defn}\label{def_localization}
Given a filtered diagram $D:J\to\EE^{\op}$ the \emph{localization} of $\EE$ at $D$ is the colimit (in $\Cat$) of the composite $j\mapsto D_j\mapsto \EE\!/D_j$:
$$\EE_D\simeq \colim_{j\in J}\EE\!/D_j$$
\end{defn}

Given a map $i\to j$ in $J$ we let $d_{ij}:D_j\to D_i$ so that $d_{ij}^*:\EE\!/D_i\to\EE\!/D_j$. For each $j\in J$, $\tilde\jmath$ denotes the colimit injection $\EE\!/D_j\to \EE_D$. 

We may take $\Ob(\EE_D)=\coprod_j \Ob(\EE\!/D_j)$. Given $\sigma\in\EE\!/D_i$ and $\tau\in\EE\!/D_j$, an arrow $f:\sigma\to\tau$ is defined by a span $D_i\from D_k \to D_j$ together with a map $\overline{f}:d_{ki}^*\sigma\to d_{kj}^*\tau$. Similarly, two arrows $\overline{f}\in\EE\!/D_i$ and $\overline{g}\in\EE\!/D_j$ are identified in the colimit just in case there is a span as above such that $d_{ki}^*\overline{f}=d_{kj}^*\overline{g}$. 

\begin{lemma}\label{localization_limits}
If $\EE$ is a pretopos and $D$ is a filtered diagram in $\EE^{\op}$ then $\EE_D$ is again a pretopos and a finite (co)cone in $\EE_D$ is (co)limiting just in case it is the image of a (co)limit in one of the slice categories.
\end{lemma}
\begin{proof}
Because any two indices $D_i$ and $D_j$ have a common span $D_i\from D_k\to D_j$, representatives from the two categories can be compared in $D_k$. Therefore, for any finite family of objects and arrows $\EE_D$ may choose representatives belonging to a single slice category $\EE\!/D_k$. Pulling back along a further map $D_l\to D_k$ we may realize all of the (finite number of) equations holding among this family. Thus any finite diagram in $\EE_D$ has a representative in a single slice.

This allows us to compute limits in $\FF$ using those in the slices. Suppose that $F$ is a finite diagram in $\EE_D$ and find a diagram of representatives $\overline{F}\in\EE\!/D_j$. Let $\overline{L}$ denote the limit of these representatives in $\EE\!/D_j$. Now suppose that $Z\to F$ is a cone in $\EE_D$. By the previous observation there is a further map $D_k\to D_j$ such that the cone and the diagram together have representatives $\overline{Z}\to d_{jk}^*\overline{F}$ in $\EE\!/D_k$.

Since pullbacks preserve limits, this induces a map $\overline{Z}\to d_{jk}^*\overline{L}$. This makes $\tilde\jmath(\overline{L})$ a limit for $F$ in $\EE_D$; uniqueness in $\EE_D$ follows from uniqueness in each of the slices. Essentially the same argument shows that coproducts and quotients may be computed in slices, so $\EE_D$ is a pretopos and each map $\tilde\jmath:\EE\!/D_j\to\EE_D$ preserves pretopos structure. 

\label{ptop_alg}
From a more sophisticated perspective, we may observe that the theory of pretoposes is itself quasi-algebraic. This means that we may write down the theory of pretoposes using sequents of Cartesian formulas (using only $\{=,\wedge\}$). This follows from the facts that pretopos structure amounts to the existence of certain adjoints, and the theory of adjoints is equational (see \cite{awodeyCT}, ch. 9). It is also important that provable functionality and equivalence relations are Cartesian-definable.

The categorical interpretation $\{\wedge, =\}$ involves only finite limits, so any functors which preserve these must also preserve models of Cartesian theories. In $\Sets$ filtered colimits commute with finite limits, so a filtered colimit of pretoposes in $\Sets$ is again a pretopos.
\end{proof}

\section{The method of diagrams}\label{sec_diagrams}

Classically the Henkin diagram of a $\bTT$-model $M$ is an extension $\bTT\subseteq\rm{Diag}(\bTT)$ constructed by:
\begin{itemize}
\item Extending the language $\LL(\bTT)$ to include new constants $\bf{c}_a$ for each $a\in|M|$.
\item Adding the axiom $\vdash \varphi(\bf{c}_a)$ for any $\bTT$-formula $\varphi(x)$ such that $a\in \varphi^M$.
\end{itemize}
In this section we review the categorical interpretation for Henkin diagrams. Here we work in the semantic context $\SS=\Sets$.

Fix a model $M:\EE\to\Sets$. From this one defines the \emph{category of elements} $\int_{\EE} M$. For the objects of $\int_{\EE} M$ we take the disjoint union $\displaystyle\coprod_{A\in\EE} A^M$. A morphism $\<b\in B^M\>\to\<a\in A^M\>$ is an arrow $f:B\to A$ such that $f^M(b)=a$. We call the pair $\<f,b\>$ a \emph{specialization} for $a$, and write $f:b\mapsto a$ to indicate that $\<f,b\>$ is a specialization of $a$.

Composition in $\int_{\EE} M$ is computed as in $\EE$. The existence of products and equalizers in $\EE$ guarantees that $(\int_{\EE} M)^{\op}$ is a filtered category. A specialization for a finite family of elements $a_i\in A_i^M$ is a specialization of the tuple, i.e. a family of maps $\<f_i: b\mapsto a_i\>$.

\begin{defn}[The pretopos diagram]\label{def_diagram}
Given a model $M$, the \emph{localization} $\EE_M$ is the localization of $\EE$ along the filtered diagram $(\int_{\EE} M)^{\op}\to\Ptop$ defined by $\<a\in A^M\>\mapsto A\mapsto \EE\!/A$. We also write $\DD(M)\simeq \underset{a\in\int M}{\colim} \EE/A$ and call this category the \emph{diagram} of $M$.
\end{defn}

An object in $\DD(M)$ is called a \emph{parameterized set} in $M$ or a p-set over $M$; such an object is defined by a triple $\<A,\sigma,a\>$ where $\sigma\in\EE\!/A$ and $a\in A^M$. The \emph{context} of the p-set is $A$ and its \emph{domain} is the domain of $\sigma$. We usually suppress the context and denote this object $\<\sigma,a\>$.

A \emph{morphism of p-sets} $\<\sigma,a\>\to\<\tau,b\>$ consists of a specialization $\<f,g\>:c\mapsto\<a,b\>$ together with a map between the pullbacks in $\EE\!/C$:
$$\xymatrix{
D \ar[d]_{\sigma} & \ar[l] \raisebox{1.5ex}{$f^*\sigma \stackrel{h}{\lto} g^*\tau$} \ar[r] \ar/va(-8) 1.6cm/;[d] \ar/va(-6) 2.7cm/;[d] &E \ar[d]^{\tau}\\
A & \ar[l]^{f} C  \ar[r]_{g} & B\\
}$$
Though technically $\<h,c,f,g\>:\<\sigma,a\>\to\<\tau,b\>$, we write $h/c:f^*\sigma\to g^*\tau$ as a shorthand for this data.

Two parallel morphisms $h/c$ and $h'/c'$ from $\<\sigma,a\>\to\<\tau,b\>$ induce a diagram:
$$\xymatrix@=2ex{
&&& f^*\sigma \ar[llldd] \ar[rd] \ar[rr]^h &&g^*\tau \ar[rrrdd] \ar[ld]&\\
&&&&C \ar[rrd]^g \ar[lld]_f &&\\
D \ar[rr]^{\sigma} && A &&&& B && E \ar[ll]_{\tau}\\
&&&& C' \ar[urr]_{g'} \ar[ull]^{f'} &&\\
&&& f'^*\sigma \ar[llluu] \ar[ru] \ar[rr]_{h'} &&g'^*\tau \ar[rrruu] \ar[lu]&\\
}$$
These two maps are equal if there is a further specialization $\<k,k'\>:d\to \<c,c'\>$ such that (i) $\<k,k'\>$ factors through the pullback $C\overtimes{A\times B} C'$ and (ii) $k^*h$ and $k'^*h'$ are equal (up to canonical isomorphism).

More specifically, the first condition ensures that $f\circ k=f'\circ k'$, inducing a canonical isomorphism $k^*f^*\cong (f\circ k)^*=(f'\circ k')^*\cong k'^*f'^*$ (and similarly for $k^*g^*\cong k'^*g'^*$). Then $h/c$ and $h'/c'$ are equal if the following diagram commutes:
$$\xymatrix{
k^*f^*\sigma \ar@{=}[d]^{\vertsim} \ar[r]^{k^*h} & k^*g^*\tau \ar@{=}[d]^{\vertsim}\\
k'^*f'^*\sigma \ar[r]_{k'^*h'} & k'^*g'^*\tau\\
}$$

Now suppose that $h/b$ and $k/c$ are composable at $\<\sigma,a\>$. The definitions of $h$ and $k$ involve maps $f:B\to A$ and $g:C\to A$; let $p_B$ and $p_C$ denote the projections from their pullback. Then we can define a map in $\EE\!/(B\overtimes{A} C)$ by
$$\xymatrix{
l:p_B^*(\dom(h)) \ar[r]^-{p_B^*h} & p_B^*f^*\sigma\cong p_C^*g^*\sigma \ar[r]^-{p_C^*k} & p_C^*(\cod(k))\\
}$$ 
Since $h/b$ and $k/c$ are morphisms of p-sets, $f(b)=a=g(c)$ and the pair $\<b,c\>$ belongs to $(B\overtimes{A} C)^M$. Then we may define the composite by $(k/c)\circ(h/b)=l/\<b,c\>$.

\begin{comment}
$\DD(M)$ is a localization (a 2-colimit) and so it is only defined up to equivalence of categories. In the next section we will use diagrams to define the structure sheaf of a theory. To show that we have a sheaf (rather than a stack or 2-sheaf) it will be important to have identity conditions for the objects of $\DD(M)$. Fortunately, equality in $\EE$ induces a natural equality in $\DD(M)$.

\begin{defn}
Two p-sets $\<\sigma,a\>$ and $\<\tau,b\>$ are \emph{equal} if \bf{(i)} they have the same domain $E=\dom(\sigma)=\dom(\tau)$ and \textbf{(ii)} there is a specialization $\<f,g\>:c\mapsto \<a,b\>$ such that the pullbacks $f^*\sigma$ and $g^*\tau$ are equal as subobjects of $E\times C$:
$$\xymatrix{
A \ar@{>->}[d]_{\Delta_A} & \ar[l] f^*\sigma = g^*\tau \ar[r] \ar@{>->}/va(-8) 1.85cm/;[d] \ar@{>->}/va(-6) 2.82cm/;[d] &E \ar@{>->}[d]^{\Delta_B}\\
A\times A & \ar[l]^{\sigma\times f} E\times C  \ar[r]_{\tau\times g} & B\times B\\
}$$
\end{defn}
\end{comment}

\begin{prop}
$\DD(M)$ is the classifying pretopos for the classical diagram $\rm{Diag}(M)$.
\end{prop}

\begin{proof}
The unique element $\top\in 1^M$ induces an interpretation of $\tilde{\top}:\EE\to\DD(M)$; let $\tilde{A}=\tilde{\top}(A)$ denote the interpretation of $A$ in $\DD(M)$. For each $a\in A^M$, applying $\tilde{a}$ to the generic constant $\Delta_A$ defines a global section (i.e., a constant) $1\to\tilde{A}$ in $\DD(M)$.

Now suppose that $i:R\mono A$ is a basic relation and $M\models R(a)$. To see that $\DD(M)\vdash R(\bf{c}_a)$ we must show that the constant $\bf{c}_a:1\to \tilde{A}$ factors through the inclusion $\tilde{R}\mono\tilde{A}$.

Because $M\models R(a)$ we have the following factorizations, where $a_R=i(a)$ is $a$ itself, regarded as an element of $R^M$:
$$\xymatrix{ \EE \ar[rr]|-{A^\times} \ar@/^5ex/[rrrrrr]^(.3){\tilde{\top}} \ar@/_3ex/[rrrr]_(.4){R^\times} && \EE\!/A \ar[rr]|{\ i^*\ } \ar@/^3ex/[rrrr]^(.2){\tilde{a}}&& \EE\!/R \ar@{-->}[rr]_{\tilde{a_R}} && \DD(M)\\}$$
Therefore, the following diagrams are identified (up to isomorphism) in $\DD(M)$, yielding the desired factorization
$$\xymatrix@R=2ex{
A^\times(R)\ \ar@{>->}[r] & **[l]A^\times(A) & R^\times(R)\ \ar@{>->}[r] & **[l]R^\times(A) &\ \ \tilde{R}\ \ar@{>->}[r] &\ \tilde{A}\ \  \\
& \ar@{}[r]^(.2){}="a"^(.7){}="b" \ar@{|->} "a";"b"^{i^*} && \ar@{}[r]^(.25){}="a"^(.75){}="b" \ar@{|->} "a";"b"^{\tilde{a_R}} &\\
 A \ar[ruu]_{\Delta_A} && R \ar[ruu]_{i^*(\Delta_A)} \ar[uu]|{\Delta_R} && 1 \ar@{-->}[uu] \ar[uur]_{\bf{c}_a}\\
}$$
\end{proof}

\begin{lemma}\label{diag_models}
$\DD(M)$ classifies $\EE$-models under $M$. A model of $\DD(M)$ consists of an $\EE$-model $N$ together with a homomorphism $h:M\to N$. A morphism of $\DD(M)$-models is a commutative triangle under $M$

In particular, the identity on $M$ induces a canonical model $I_M:\DD(M)\lto\Sets$.
\end{lemma}
\begin{proof}
Suppose that $H:\DD(M)\to\Sets$ is a model of $\DD(M)$. The composite $H\circ\tilde{\top}$ provides the asserted $\EE$-model $N$. Similarly, $H\circ\tilde{a}$ defines an interpretation of each $\bf{c}_a$ in $N$:
$$\xymatrix{
\EE\!/A \ar@/^3ex/[rr]^{\<N,\bf{c}_a^N\>} \ar[r]_-{\tilde{a}} & \DD(M) \ar[r]_H & \Sets\\
\EE \ar[u]^{A^\times} \ar[ur]_{\tilde{\top}} \ar@/_3ex/[urr]_{N} \\
}$$

From this we can define a family of functions $h_A: A^M\to A^{N}$ by setting $h_A(a)=\bf{c}_a^N$. Fix an $a\in A^M$ and for each $\sigma:A\to B$ let $b_\sigma=\sigma^M(a)$. This corresponds to an axiom $\vdash \sigma(\bf{c}_a)=\bf{c}_{b_\sigma}$ in $\DD(M)$ and therefore
$$\sigma^N\circ h_A(a) = \sigma^N(\bf{c}_a^N) = \big(\sigma(\bf{c}_a)\big)^N = \bf{c}_{b_\sigma}=h_B(\sigma^M(a))$$
This ensures that $h$ defines a natural transformation $M\to N$ and hence an $\EE$-model homomorphism.

This construction is reversible. Given a homomorphism $h:M\to N$, simply define $H:\DD(M)\to \Sets$ by sending each pair $\<b,\sigma\>\in\DD(M)$ to the definable set $(\sigma^N)^{-1}\big(h_B(b)\big)$

For a natural transformation $\theta:H_1\to H_2$, the composite $\theta\cdot\tilde{\top}$ induces a homomorphism $g_\theta:N_1\to N_2$. Similarly, $\theta\cdot\tilde{a}$ induces an $\EE\!/A$-natural transformation $H_1\cdot\tilde{a}\to H_2\cdot\tilde{a}$. From this it follows that $g_\theta\big(\bf{c}_a^{N_1}\big)=\bf{c}_a^{N_2}$. Then $g_\theta$ commutes with the maps $h_i:M\to N_i$, making a commutative triangle under $M$.
\end{proof}

The diagram $\DD(M)$ is closely related to the notion of a definable set. Recall that a set $S\subseteq A^M$ is \emph{definable} if there is some binary formula (subobject) $\varphi\mono A\times B$ and an element $b\in B^M$ such that $S=\{a\in A^M\ |\ M\models\varphi(a,b)\}$. We often specify a definable set in $M$ by writing $S=\varphi(x,a)^M$. A function $f:S\to T$ is \emph{definable} just in case its graph is. Let $\DS(M)$ denote the category of definable sets in $M$.

\begin{defn}
When $\sigma:E\to A$, the \emph{realization} of $\<\sigma,a\>$ in $M$ is a definable subset $(\sigma^M)^{-1}(a)\subseteq E^M$, the fiber of $\sigma^M$ over $a$. Given any homomorphism $f:M\to N$, the \emph{realization under} $f$ is the definable set $(\sigma^N)^{-1}(f(a))$. We denote this set by $\<\sigma,a\>^f$ or $\<\sigma,f(a)\>^N$.
\end{defn}

Suppose $S=\varphi(x,b)$ for some formula $\varphi\mono A\times B$. If $\sigma$ denotes the composite map $\varphi\mono A\times B\to B$, then $S$ is the realization of $\<\sigma,b\>$ in $M$. Thus every definable set is a realization of some p-set (though this requires changing the ``ambient set'' of $S$ from $A^M$ to $\varphi^M$).

Similarly, given two p-sets and a definable function between their realization, the formula which defined its graph induces a morphism of p-sets. Thus realization defines a functor $\DD(M)\to\DS(M)$ which is essentially surjective on objects and full.

However, $\DD(M)$ and $\DS(M)$ are not identical. A definable set can have many definitions and these definitions might induce induce different p-sets in $\DD(M)$. This is analogous to the situation in algebraic geometry, where distinct polynomials sometimes induce the same vanishing set (e.g. $xy$ and $x^2y$ both vanish on the same set $\{\<x,y\>|\ x=0\vee y=0\}$).

In fact, $\DS(M)$ is the quotient-conservative factorization of the canonical functor $I_M$ defined above:
$$\xymatrix{
\DD(M) \ar[rr]^{I_M} \ar@{->>}[rd]_{\rm{quot.}} &&\Sets\\
& \DS(M) \ar@{>->}[ur]_{\rm{cons.}}&\\
}$$

\begin{lemma}
Suppose that we have two p-sets $\<\sigma,a\>$ and $\<\tau,b\>$ and that $\dom(\sigma)=E=\dom(\tau)$. The following are equivalent:
\begin{itemize}
\item $\<\sigma,a\>\leq\<\tau,b\>$ in the subobject lattice $\Sub_{\DD(M)}(\tilde{E})$.
\item There is a binary formula $\gamma\mono A\times B$ such that $M\models \gamma(a,b)$ and $\EE$ proves
$$(\sigma(w)=x)\wedge \gamma(x,y)\vdash \tau(w)=y.$$
\item For every homomorphism $f:M\to N$, the realization of $\<\sigma,a\>$ in $N$ is contained in the realization of $\<\tau,b\>$:
$$(\sigma^N)^{-1}\big(f(a)\big)\subseteq (\tau^N)^{-1}\big(f(b)\big).$$
\end{itemize}
\end{lemma}

\begin{proof}
First suppose that $\<\sigma,a\>\leq\<\tau,b\>$. Because $\DD(M)$ classifies the diagram of $M$, there must be a proof
$$\rm{Diag}(M),\sigma(w)=\bf{c}_a \vdash \tau(w)=\bf{c}_b.$$
Let $\vdash\varphi_i(\bf{c}_{c_i})$ be the (finitely many) axioms from $\Diag(M)$ involved in the proof and set $\varphi(\bf{c}_a,\bf{c}_b,\bf{c}_c))=\bigwedge_i \varphi_i(\bf{c}_{c_i})$, where $\bf{c}_c$ is disjoint from $\bf{c}_a$ and $\bf{c}_b$. This leaves us with 
$$\varphi(\bf{c}_a,\bf{c}_b,\bf{c}_c)\wedge(\sigma(z)=\bf{c}_a)\vdash \tau(z)=\bf{c}_b.$$

The right-hand side of this sequent does not involve the constants $\bf{c}_c$, so we may replace them by an existential quantifier on the left. Let $\gamma(x,y)=\exists z.\varphi(x,y,z)$. The existence of $c\in C^M$ shows that $M\models\gamma(a,b)$. Now replacing the constants $\bf{c}_a$ and $\bf{c}_b$ by free variables $x$ and $y$, we are left with the desired statement:
$$\EE\tri\gamma(x,y)\wedge(\sigma(z)=x)\vdash \tau(z)=y.$$

Now suppose that the second condition holds and that $f:M\to N$. Since $M\models\gamma(a,b)$, we also have $N\models\gamma(f(a),f(b))$. If $c\in \sigma^{-1}(f(a))$ then the $\<f(a),f(b),c\>$ satisfies the left-hand side of the sequent above. By soundness, we also have $\tau^N(c)=f(b)$, so the realization of $\<\sigma,a\>$ under $f$ is contained in that of $\<\tau,b\>$.

For the last equivalence, simply note that $\DD(M)$ classifies homomorphisms under $M$; let $N_f:\DD(M)\to\Sets$ denote the model associated to a homomorphism $f:M\to N$. Then in every model we have 
$$N_f(\<\sigma,a\>)=\<\sigma,f(a)\>^N\leq \<\tau,f(b)\>^N=N_f(\<\tau,b\>).$$
Then completeness guarantees that this inclusion must hold in the theory, so $\<\sigma,a\>\leq\<\tau,b\>$ in $\DD(M)$.
\end{proof}

We close with a discussion of the locality property satisfied by the diagram $\DD(M)$. This will be an important condition when we define the structure sheaf of a logical scheme in the next chapter.
\begin{defn}\label{def_local}
Fix an object $A\in\EE$.
\begin{itemize}
\item $A$ is \emph{projective} if every epi $p:B\epi A$ has a section $s:A\to B$, so that $p\circ s=1_A$.
\item $A$ is \emph{indecomposable} if for any subobjects $R,S\leq A$ we have $R\vee S=A$ implies $R\cong A$ or $S\cong A$.
\item $\EE$ is \emph{local} if the terminal object $1\in\EE$ is projective and indecomposable.
\end{itemize}
\end{defn}

\begin{lemma}\label{local_functor}
When $\EE$ is a local pretopos, the co-representable functor $\Hom(1,-):\EE\to\Sets$ is a pretopos functor (and hence a model of $\EE$).
\end{lemma}
\begin{proof}
The co-representables preserve limits, essentially by the definition of limits: $\Hom(1,\lim_i A_i)\cong\lim_i \Hom(1,A_i)$. Therefore it is enough to check that $\Hom(1,-)$ preserves epimorphisms and coproducts.

For the first, suppose that $f:A\epi B$ is epic and that $b:1\to B$. Pulling $f$ back along $b$ yields another epi $b^*A\epi 1$ which, by locality, has a section $s:1\to b^*A\to A$. Then $f\circ s=b$, so the induced map $\Hom(1,A)\to\Hom(1,B)$ is surjective.

Similarly, given a point $c:1\to A+B$ we may pull back the coproduct inclusions to give a partition $1\cong c^*A+c^*B$. Since $1$ is indecomposable, either $c^*A\cong 1$ (giving a factorization of $c$ through $A$), or vice versa. This defines an isomorphism $\Hom(1,A+B)\cong\Hom(1,A)+\Hom(1,B)$.

When $\EE=\Diag(M)$ is the diagram of a $\bTT$-model $M$, one easily checks that $\Hom(1,-)$ is the canonical $\EE$-model associated with the identity homomorphism $M\to M$ (cf. lemma \ref{diag_models})
\end{proof}

\begin{lemma}\label{local_stalks}
Each diagram $\DD(M)$ is a local pretopos. Logically speaking, $\DD(M)$ satisfies the existence and disjunction properties.
\end{lemma}
\begin{proof}
Suppose $E=\<\sigma,b\>$ is an object in $\DD(M)$ and that the projection $E\epi 1$ is epic. Applying the canonical interpretation $I_M$ sends this epi to a surjection $E^M\to 1\in \Sets$. This means the definable set $E^M$ is non-empty, so pick an element ${a\in E^M\subseteq A^M}$. We need to see that the constant $\bf{c}_a$ factors through $E$.

In $\DD(M)$ we have a section $\<\bf{c}_a,\bf{c}_b\>:1\to \tilde{A}\times \tilde{B}$ and $E$ is a pullback along this section:
$$\xymatrix{
E \pbcorner \ar[d] \ar[r] & \tilde{A}\times \tilde{A} \ar[d]^{1\times\tilde{\sigma}}\\
1 \ar[r]_-{\<\bf{c}_a,\bf{c}_b\>} \ar[ur]|{\<\bf{c}_a,\bf{c}_a\>} & \tilde{A}\times \tilde{B}\\
}$$
Because $\sigma^M(a)=b$, the diagram $\DD(M)$ contains an axiom $\vdash \sigma(\bf{c}_a)=\bf{c}_b$. This implies that the bottom triangle in this diagram commutes and therefore the map $\<\bf{c}_a,\bf{c}_a\>$ induces a section $1\to E$, as desired.

Now suppose that $R,S\leq 1$ are subterminal objects and that $R\vee S=1$. Applying the canonical model $I_M$, this gives a surjection $R^M+S^M\epi 1$ in $\Sets$, so either $R^M$ or $S^M$ is non-empty. Without loss of generality, suppose $a\in R^M$. Just as above, this defines a section $1\to R$ in $\DD(M)$. Since we are in a pretopos and the projection $R\to 1$ is both epic and monic, $R\cong 1$.

Logically, projectivity and indecomposability correspond to the existence and disjunction properties for $\DD(M)$:
$$\begin{array}{rcl}
\DD(M)\models \exists x. \varphi(x) & \Iff & \DD(M)\models\varphi(t)\textrm{ for some (definable) closed term }t.\\
\DD(M)\models \varphi\vee\psi & \Iff & \DD(M)\models \varphi\textrm{ or }\DD(M)\models \psi.\\
\end{array}$$

The first equivalence follows from the fact that sections $1\to \varphi\leq\tilde{A}$ are exactly the $\DD(M)$-definable singletons in $A$ which provably satisfy $\varphi$ (where $\varphi(x)=\varphi(x,\bf{c}_b)$ is an $\EE$-formula in context $x:A$ which may also contain parameters from $M$).

For the second, suppose that $\varphi,\psi\leq 1$ are subterminal object in $\DD(M)$ (i.e. sentences). If $\DD(M)\models\varphi\vee\psi$ then the projection $\varphi+\psi\epi 1$ is an epi. But then either $\varphi\to 1$ or $\psi\to 1$ is both epi and mono, hence an isomorphism. Thus either $\DD(M)\models\varphi$ or $\DD(M)\models\psi$.
\end{proof}

\section{Classical theories and Boolean pretoposes}\label{sec_bool_ptop}

\begin{defn}
A coherent category (or pretopos) $\bEE$ is \emph{Boolean} if every subobject in $\bEE$ is complemented: for any object $A\in\bEE$ and subobject $S\leq A$, there exists another subobject denoted $\neg S$ such $A\cong S+\neg S$. We denote the (2-)category of Boolean pretoposes by $\BPtp$.
\end{defn}

\begin{lemma}
If $\bEE$ is a coherent category which is Boolean, then its pretopos completion $\EE$ is also Boolean.
\end{lemma}
\begin{proof}
Fix an object $Q\in \EE$ and a subobject $S\leq Q$. Because the unit of the completion $I:\bEE\to\EE$ is finitely subcovering, there is a finite coproduct $A\cong \coprod_i A_i$ with $A_i\in\bEE$ and a cover $f:A\epi Q$. Let $R=f^*S$ and $R_i=R\cap A_i$. 

$I$ is full on subobjects, so each $R_i$ belongs to $\bEE$. Since $\bEE$ is Boolean, this means that $R$ has a complement $\neg R\cong\coprod_i\neg R_i$. Its image $\exists_f(\neg R)$ is the complement of $S$.

Since $S=\exists_f(f^*S)=\exists_f(R)$ and $R+\neg R\epi Q$, the two subobjects cover $Q$: $S\vee\exists_f(\neg R)=\exists_f(R+\neg R)=Q.$ They are disjoint because $S\cap \exists_f(\neg R)=\exists_f(R\cap\neg R)=\exists_f(0)$. Thus any subobject $S\leq Q$ has a complement $\neg S=\exists_f(\neg R)$, so $\EE$ is Boolean.
\end{proof}

\begin{prop}
There is a Boolean completion operation $\hat{(-)}:\Ptop\to\BPtp$ which is right adjoint to the forgetful functor $\BPtp\to\Ptop$.
\end{prop}

\begin{proof}
For any classical theory $\bTT$ the first-order formulas and (classically) provably functional relations form a Boolean coherent category. Its pretopos completion is the classifying pretopos of the (classical) theory. Moreover, any coherent theory may be regarded as a classical theory with the same language and axioms (but more connectives and inference rules).

Now replace any pretopos $\EE$, by its associated theory $\bTT_{\EE}$ (page \pageref{class_ptop}) and then regard $\bTT_{\EE}$ as a classical theory. The associated Boolean pretopos is the Boolean completion $\hat{\EE}$. The unit $\EE\to\hat{\EE}$ acts by sending each coherent $\EE$-formula to the same formula regarded classically. This operation extends to maps and to quotients because intuitionistic proofs (of provable functionality or equivalence relations) are all classically valid.

When $\BB$ is Boolean we can lift any pretopos functor $I:\EE\to\BB$ along the unit:
$$\xymatrix{
\hat{\EE} \ar@{-->}[r]^{\hat{I}} & \BB\\
\EE \ar[u] \ar[ur]_{I}&.\\
}$$
We build the lift using the usual inductive scheme $\hat{I}(A)=I(A)$ when $A\in\EE$ and while for other objects
$$\begin{array}{cc}
\hat{I}(A \overtimes{C} B)\cong \hat{I}(A)\overtimes{\hat{I}(C)} \hat{I}(B) & \hat{I}(\top_{\hat{E}})=\top_{\BB} \\
\hat{I}(A+B)\cong \hat{I}(A)+\hat{I}(B) & \hat{I}(A/R)\cong \hat{I}(A)/\hat{I}(R)\\
\hat{I}(\neg A)\cong \neg\hat{I}(A)
\end{array}$$
On one hand, these operations suffice to construct all the classical pretopos formulas which are the objects of $\hat{E}$. The functor is well-defined because $I$ is a pretopos functor, meaning that the inductive clauses agree with the base cases wherever they overlap.
\end{proof}

For example, in the last chapter we showed how to replace any classical theory $\bTT$ by an equivalent coherent theory $\overline\bTT$. To do so we needed to extended the language $\LL\subseteq\overline\LL$; at the level of pretoposes, this corresponds to Boolean completion: $\EE_{\LL} \lto \widehat{\EE_{\LL}}\simeq \EE_{\overline\LL}$. Since $\Sets$ is Boolean, the universal mapping property of $\widehat{\EE_{\LL}}$ guarantees that every $\LL$-structure has a unique extension to an $\overline\LL$-structure. Similarly, when $\bTT$ is already coherent we have $\EE_{\overline\bTT}\simeq\widehat{\EE_{\bTT}}$.

\begin{lemma}
\mbox{}
\begin{itemize}
\item Any quotient of a Boolean pretopos is Boolean.
\item Any localization of a Boolean pretopos is Boolean.
\end{itemize}
\end{lemma}
\begin{proof}
First consider a quotient map $I:\EE\to\FF$. Any pretopos functor necessarily preserves joins and disjointness (a limit condition). Therefore, it also preserves coproducts and hence complementation: $I(S)+I(\neg S)\cong I(S+\neg S)\cong IA$.

Since quotients are full on subobjects, any $S\leq IA$ has a preimage $IR\cong S$ and therefore a complement $\neg S=I(\neg R)$. For any other object $B\in\FF$ and $S\leq B$, we define $\neg S$ just as in the the previous lemma: $\neg S=\exists_q(\neg q^*S)$. Exactly the same argument shows that $S+\neg S\cong B$.

As for localizations, note that subobjects and complements in $\EE\!/A$ are just subobjects and complements in $\EE$. Thus $\EE\!/A$ is Boolean so long as $\EE$ is. Now suppose that $\FF\simeq\colim_j \EE\!/A_j$ is a localization of $\EE$ over a filtered category $J$ and that $S\leq F$ belong to $\FF$. This subobject has a representative $\overline{S}\leq\overline{F}$ in $\EE\!/A_j$ for some $j\in J$ and the representative is complemented. Since $\tilde\jmath:\EE\!/A_j\to\FF$ is a pretopos functor, $\tilde\jmath(\neg \overline{S})$ is a complement to $S$.
\end{proof}

Recall that $\EE$ is \emph{well-pointed} if for any distinct maps $f\not= g:A\to B$ there is a global element $a:1\to A$ such that $f\circ a\not=g\circ a$. This is the appropriate notion of locality for Boolean pretoposes.

\begin{lemma}\label{bool_stalks}
If a pretopos $\DD$ is Boolean and local, then it is two-valued and well-pointed. In particular, if $\EE$ is Boolean and $M:\EE\to\Sets$ is a model then the diagram $\DD(M)$ satisfies these conditions.
\end{lemma}
\begin{proof}
Consider any subterminal object $S\leq 1$ in $\DD$. Since $\DD$ is Boolean, $S+\neg S\cong 1$. But $\EE$ is local and $1$ is indecomposable, so either $S=1$ or $\neg S=1$ (and hence $S=0$). Thus $\DD$ is two-valued.

Now suppose that $f\not=g: A\to B$ in $\DD$. Let $R\leq A$ denote the complement of the equalizer of $f$ and $g$. Since $f\not= g$, $R$ is non-zero and therefore its image $R\epi \exists R \mono 1$ in non-zero as well. Since $\DD$ is two-valued, we must have $\exists R=1$, so the projection $R\epi 1$ is epic.

Since $\DD$ is local, $1$ is projective and we can find a section $s:1\to R\leq A$. If $f\circ s=g\circ s$ then $s$ would factor through the equalizer; this is impossible since $s$ factors through $R$, which is the (disjoint) complement of the equalizer. Therefore $f\circ s\not= g\circ s$, so $\DD$ is well-pointed.

In particular, we know that the diagram $\DD(M)$ is always local (lemma \ref{local_stalks}). By the previous lemma, when $\EE$ is Boolean so are its localizations. Therefore the assumptions of the lemma apply, so $\DD(M)$ is two-valued and well-pointed.
\end{proof}

%% file: Ch3.tex
\chapter{Logical Schemes}

In this chapter we will define the 2-category of logical schemes and prove some basic facts about them. We begin by defining the logical structure sheaf $\OO_{\bTT}$, a sheaf of local pretoposes over the spectrum $\MM_{\bTT}$ defined in chapter 1. The pair $(\MM_{\bTT},\OO_{\bTT})$ is an affine logical scheme. Following familiar definitions in algebraic geometry, we go on to define more general logical schemes by gluing, and we use the affine case to guide the definition of morphisms and natural transformations of schemes. This defines a reflective 2-adjunction with the 2-category of pretoposes, and we use this fact to show that schemes are closed under all finite 2-limits.

\section{Stacks and Sheaves}

We begin by defining three closely related notions of a ``category over $\EE$;'' see Vistoli \cite{fibrations} or Moerdijk \cite{moerdijk_stacks} (for the topological case) for a thorough treatment of fibrations and stacks. The simplest of these three is a \emph{presheaf of categories}, which is simply a (strict) functor $\bf{C}:\EE^{\op}\to\Cat$. The second is that of an \emph{$\EE$-indexed category}, which is a pseudofunctor $\bf{C}:\EE^{\op}\to\Cat$. Here the transition morphisms $f^*:\bf{C}(B)\to\bf{C}(A)$ are only expected to commute up to a canonical natural isomorphism, $(g\circ f)^*\cong f^*\circ g^*$, and these satisfy further coherence conditions. See \cite{bunge_pare} for further details.

The third, fibrations, requires a preliminary definition. Given a fixed projection functor $\bf{p}:\CC\to\EE$, an arrow $g:C'\to C$ is \emph{Cartesian over $f$} if $\bf{p}(g)=f$ and for any $h:C''\to C$ and any factorization $\bf{p}(h)=f\circ q$, there exists a unique lift $\bf{p}\left(\bar{q}\right)=q$ such that $h=g\circ \bar{q}$. Diagramatically,
$$\xymatrix@=12pt{
C'' \ar[rrrrd]^{h} \ar@{-->}[rrd]_{\exists !\bar{q}}\\
&& C' \ar[rr]_{g} && C & g\rm{ is Cartesian}\\
\bf{p}(C'') \ar[rrd]_{q} \ar[rrrrd]^{\bf{p}(h)} &&&&& \rm{if } \forall h,q \ \exists !\ \overline{q}  \\
&& \bf{p}(C') \ar[rr]_{f=\bf{p}(g)} && \bf{p}(C).\\
}$$
The functor $\bf{p}:\CC\to\EE$ is a \emph{fibration} if every arrow in $\EE$ has a Cartesian lift in $\CC$. Usually we will suppress the functor $\bf{p}$ and refer to a fibration by its domain $\CC$.

A \emph{fibered functor} over $\EE$ is a functor between fibrations which commutes with the projections to $\EE$ and sends Cartesian arrows to Cartesian arrows. We let $\uHom_{\EE}(\CC,\DD)$ denote the category of fibered functors and natural transformations between them.

Given a fibration $\CC\to\EE$ and an object $A\in\EE$, we write $\CC(A)$ for the subcategory of objects and arrows sent to $A$ and $1_A$, respectively. If we fix an $f:B\to A$ and choose a Cartesian lift $f^*C\to C$ for every $C\in\CC(A)$, this defines a ``pullback'' functor $f^*:\CC(A)\to\CC(B)$
$$\xymatrix@=12pt{
&f^*D \ar@{-->}[rd]_{\exists! f^*h} \ar[rrr] &&& D  \ar[rd]^{\forall h} \\
&& f^*C \ar[rrr] &&& C  &\\\\
&& \CC(B) &&& \ar[lll]_{f^*} \CC(A).
}$$

A \emph{cleavage} on $\CC$ is a choice of Cartesian lift $f^*C\to C$ for every pair $\<f,C\>$. Modulo the axiom of choice, these always exist. Cartesian arrows compose and any two ``pullbacks'' are canonically isomorphic (by the usual argument) yielding canonical isomorphisms $(g\circ f)^*\cong f^*\circ g^*$. This shows that the data necessary to specify a cloven fibrations is exactly the same as that necessary to give an indexed category.

The choice of cleavage is essentially irrelevant; different choices induce canonically isomorphic pullbacks. Therefore, we abuse terminology by referring to ``the'' transition functor $\CC(A)\to\CC(B)$ associated with a fibration without specifying a chosen cleavage. This is harmless so long as we make no claims about object identity in $\CC(A)$. However, note that the identity of arrows in $\Hom_{\CC(A)}(f^*D,f^*C)$ does not depend on the choice of cleavage.

We say that $\CC$ is \emph{fibered in $\Sets$} (or $\Grpd$, $\Ptop$, etc.) if for every $A\in\EE$, $\CC(A)$ is a set (or a groupoid, pretopos, $\ldots$). A fibration in $\Sets$ is essentially an ordinary presheaf $\EE^{\op}\to\Sets$. This follows from the fact that the Cartesian lifts of a fibration in $\Sets$ are unique: any two must be isomorphic and isomorphisms in a set (regarded as a discrete category) are identities. In particular, the representable functors $yA$ correspond to the forgetful functor $\EE\!/A\to\EE$ and there is a fibrational analogue of the Yoneda lemma.

\begin{defn}
For every object $A\in\EE$ the forgetful functor $\EE\!/A\to\EE$ is a fibration in $\Sets$, call the \emph{representable fibration} $yA$:
$$\xymatrix{
\EE\!/A \ar[dd]_{yA} & E' \ar[rr]^{\rm{Cartesian}} \ar[rd] && E \ar[ld]\\
&&A&\\
\EE & E' \ar[rr] && E\\
}$$
\end{defn}

\begin{prop}[2-Yoneda Lemma]\label{2yoneda}

For any fibration $\CC\to\EE$ and any object $A\in\EE$ there is an equivalence of categories (natural in $\CC$ and $A$)
$$\CC(A)\ \simeq\ \uHom_{\EE}(yA,\CC).$$ 
\end{prop}

\begin{proof}
The argument is essentially the same as in $\Sets$. Evaluation at $1_A\in yA(A)$ defines a functor $\uHom_{\EE}(yA,\CC)\to\CC(A)$. Conversely, given an object $C\in\CC(A)$ we define a fibered functor $\EE\!/A\to\CC$ by sending $f\mapsto f^*C$. We always have $1_A^*C\cong C$, so the composite $\CC(A)\to\uHom(yA,\CC)\to\CC(A)$ is clearly an equivalence of categories.

On the other hand, suppose $F:yA\to\CC$ is a fibered functor. Every $g:E\to A$ can be viewed as a map $g\to 1_A$ in $\EE\!/A$. This map is Cartesian (every map is Cartesian, since $yA$ is fibered in $\Sets$). Therefore its image $F(g)\to F(1_A)$ is also Cartesian, so $F(g)\cong g^*(F(1_A))$. This guarantees that the composite $\uHom_{\EE}(\EE\!/A,\CC)\to\CC(A)\to \uHom_{\EE}(\EE\!/A,\CC)$ is again an equivalence of categories.
\end{proof}

A \emph{splitting} of $\CC$ is a coherent cleavage $(g\circ f)^*= f^*\circ g^*$, corresponding to a presheaf of categories. Not every fibration has a splitting, but we can always find split fibration which is pointwise equivalent to $\CC$. Indeed, if we set $\widehat{\CC}(A)=\uHom_{\EE}(\EE\!/A,\CC)$ then, because composition of functors is strict, this defines a presheaf of categories over $\EE$. Pointwise equivalence follows from the Yoneda lemma.

Now we turn to the definition of stacks, which have roughly the same relation with fibrations as sheaves have to presheaves.

\begin{defn}
Suppose that $J=\{A_i\to Q\}$ is a covering family in $\EE$. Let $A_{ij}=A_i\overtimes{Q} A_j$. An object of \emph{descent data for $J$ over $Q$}, denoted $C_i/\alpha$, consists of the following:
\begin{itemize}
\item a family of objects $C_i\in \CC(A_i)$
\item together with a family of gluing isomorphisms $\alpha_{ij}:p_i^*C_i\cong p_j^*C_j$ in $\CC(A_{ij})$
\item which satisfy the following \emph{cocycle conditions} in $\CC(A_i)$ and $\CC(A_{ijk})$, respectively:
$$\Delta^*(\alpha_{ii})=1_{A_i} \hspace{1cm} p_{ij}^*(\alpha)\circ p_{jk}^*(\alpha)=p_{ik}^*(\alpha).$$
\end{itemize}
The situation is summarized in the following diagram (where, e.g., $\displaystyle p_{12}$ is the aggregate of the projections $A_{ijk}\to A_{ij}$).
$$\xymatrix@R=3ex{
p_{ij}^*(\alpha_{ij})\circ p_{jk}^*(\alpha_{jk})=p_{ik}^*(\alpha_{ik}) & p_i^*C_i\stackrel{\alpha_{ij}}{\cong}p_j^*C_j & C_i & C/\alpha\\
\raisebox{-4ex}{$\displaystyle\coprod_{i,j,k} A_{ijk}$} \ar@<1.5ex>[r]^-{p_{12}} \ar[r]|-{p_{13}} \ar@<-1.5ex>[r]_-{p_{23}} &
\raisebox{-4ex}{$\displaystyle\coprod_{i,j} A_{ij}$} \ar@<-1.5ex>[r]_-{p_2} \ar@<1.5ex>[r]^-{p_1} & \raisebox{-4ex}{$\displaystyle\coprod_{i} A_{i}$}  \ar@{->>}[r]^{q}  \ar[l]|-{\Delta} & Q.\\
}$$
\end{defn}

This defines a descent category $\Desc(\CC,J)$, where a morphism $C_i/\alpha\to D_i/\beta$ is a family of maps $h_i:C_i\to D_i$ which respect the descent isomorphisms: $\beta_{ij}\circ p_i^*(h_i)=p_j^*(h_j)\circ \alpha_{ij}$. Roughly speaking, we can think of $\alpha_{ij}$ as instructions for gluing the pieces $E_i$ together to form an object over $Q$; the cocycle conditions guarantee these gluings are compatible. Moreover, there is a functor $i_Q:\CC(Q)\to\Desc(\CC,J)$ sending each object $D\in\CC(Q)$ to $q^*D$, together with the canonical isomorphism $p_1^*(q^*D)\cong p_2^*(q^*D)$.

\begin{defn}
$\CC$ is \emph{stack} if the map $i_Q$ is an equivalence of categories for every object $Q$ and every covering family $J$. We denote the 2-category of stacks by $\Stk(\EE)$ (when the topology $\JJ$ is implicit).
\end{defn}

When we are working with the coherent topology on a pretopos, we can factor any covering family into an extensive (coproduct) cover $\{A_i\to\coprod_i A_i\}$ and a regular (singleton) cover $\{A\epi Q\}$. The stack condition for the former simply asserts that $\CC(\coprod_i A_i)\simeq\prod_i \CC(A_i)$ and this is typically easy to verify. Therefore we usually restrict consideration to singleton covers, leaving the general case to the reader.

We note that any representable functor $yA$ is a sheaf of sets (and hence a stack) for the coherent topology. 

\begin{lemma}[cf. Awodey \cite{awodey_thesis}]\label{stack_to_sheaf}
$ $\begin{itemize}
\item If $\CC$ is a presheaf of categories (i.e. $\CC_0$ and $\CC_1$ are fibered in $\Sets$) and $\CC$ is a stack, then it is equivalent in $\Stk(\EE)$ to the sheafification
$$\bf{a}\CC=\bf{a}\CC_1\rightrightarrows\bf{a}\CC_0.$$
\item For every stack $\CC$ there is a sheaf of categories $\overline{\CC}$ such that $\CC(A)\simeq\overline{\CC}(A)$ (naturally in $A$).
\end{itemize}

\end{lemma}
\begin{proof}
Recall that for a presheaf $P$, the sheafification is defined by $\bf{a}(P)=P^{++}$, where $P^+$ is the presheaf of matching families in $P$. A (strictly) matching family for a singleton cover $q:A\to Q$ is simply an object $C\in\widehat{\CC}(A)$ such that $p_1^*C=p_2^*C$. Similarly, a map $h:C'\to C$ matches for $q$ if $p_1^*(h)=p_2^*(h)$. This is precisely a descent map $C'/\!\!=\lto C/\!\!=$.

This displays $\widehat{\CC}^+(Q)$ as a full subcategory inside $\Desc(\CC,q)$. As $\CC$ is a stack, any object of descent data $C/\alpha$ has a representative $D\in\CC(Q)$ with $C/\alpha\cong q^*D/\!\!=$. Therefore the inclusion $\widehat{\CC}^+(Q)\subseteq\Desc(\CC,q)$ is essentially surjective and
$$\widehat{\CC}^+(Q)\simeq\Desc(\CC,q)\simeq\widehat{\CC}(Q).$$

Now for any stack $\CC$ we set $\overline{\CC}=\bf{a}\widehat{\CC}$, the sheafification of the strict fibration $\widehat{\CC}(A)=\uHom_{\EE}(yA,\CC)$. Now we have an equivalence $\widehat{\CC}\simeq\CC$ where $\widehat{\CC}$ is a presheaf and $\CC$ is a stack, so the previous statement applies: $\CC(A)\simeq\widehat{\CC}(A)\simeq\overline{\CC}(A)$.
\end{proof}

Because of this many sheaf constructions (e.g., limits and colimits, direct and inverse image) have ``up to isomorphism'' analogues for stacks. Moreover, many of the same formulas will hold so long as we weaken equalities to isomorphisms and isomorphisms to equivalence of categories.

In particular, limits of stacks are computed pointwise. Since adjunctions, disjointness and equivalence relations can be expressed in finite limits, $\CC$ is a pretopos in $\Stk(\EE)$ just in case $\CC(A)$ is a pretopos for each $A\in\EE$ and each transition map $f^*:\CC(A)\to\CC(B)$ is a pretopos functor.

Shifting attention to topological spaces, a stack on $X$ is a stack on the open sets $\OO(X)$. Any continuous function $f:X\to Y$ induces an adjoint pair $\xymatrix{\Stk(X)\rtwocell^{f_*}_{f^*}{'\top} & \Stk(Y)}$. These are defined just as for sheaves.

\begin{comment}

The direct image is precomposition with $f^{-1}$:
$$f_*\CC:\OO(Y)^{\op}\stackrel{f^{-1}}{\lto}\OO(X)^{\op}\stackrel{\CC}{\lto}\Cat.$$
For the inverse image we present a stack $\DD\in\Stk(Y)$ as a (weighted) colimit of open sections $\DD\simeq\colim_j U_j$ and set $f^*\DD\simeq\colim_j f^{-1} U_j$.

When $X$ is a topological groupoid $X_1\overset{d}{\underset{c}{\rightrightarrows}} X_0$ and $\CC$ is a stack on $X_0$, an \emph{equivariant structure} for $\CC$ is a map $\rho:d^*\CC\to c^*\CC$. For each $\alpha:x\to x'$ this defines a map of stalks $\rho_\alpha:\CC_{x}\to\CC_{x'}$ and these are subject to the familiar cocycle conditions $\rho_{1_x}\simeq 1_{\CC_x}$ and $\rho_{\alpha\circ\beta}\simeq\rho_\alpha\circ\rho_\beta$.
\end{comment}

\section{Affine schemes}

Now we define a sheaf of pretoposes over the topological groupoid $\MM_{\EE}$ (from chapter 1, definitions \ref{M0points}, \ref{M0opens}, \ref{M1points}) which will act as the structure sheaf in our scheme construction. This arises most naturally as a stack over $\EE$. The ``walking arrow'' is the poset $\2=\{0\leq 1\}$ regarded as a category, and the exponential $\EE^{\2}$ is the arrow category of $\EE$. Its objects are $\EE$-arrows $E\to A$ and its morphisms are commutative squares. The inclusion $\{1\}\subseteq \2$ induces the \emph{codomain fibration} $\EE^{\2}\to\EE$.

To justify the name, note that an arrow in $\EE^{\2}$ (i.e. a commutative square) is Cartesian just in case it is a pullback square; therefore $\EE^{\2}\to\EE$ is a fibration so long as $\EE$ has finite limits. The associated pseudofunctor $\EE^{\op}\to\Cat$ sends $A\mapsto\EE\!/A$, with the contravariant action given by pullback. We will usually refer to an object in $\EE\!/A$ by its domain $E$; when necessary, we generically refer to the projection $E\to A$ by $\pi$ or $\pi_E$.

\begin{prop}[cf. Bunge \& Pare \cite{bunge_pare}]
When $\EE$ is a pretopos, $\EE^{\2}$ is a stack for the coherent topology.
\end{prop}
\begin{proof}
Suppose that we have a covering map $q:A\epi Q$ with its kernel pair $K=A\overtimes{Q} A\overset{p_1}{\underset{p_2}{\rightrightarrows}} A$. Consider an object $E\to A$ together with descent data $\alpha:p_1^* E\iso p_2^*E$.

From this we define an equivalence relation $R_\alpha$ on $E$ by setting $e\stackrel{\alpha}{\sim} e'$ whenever $\alpha(e)=e'$ (which makes sense so long as $\pi(e)\stackrel{q}{\sim} \pi(e')$). More formally,
$$\xymatrix{
p_1^*E \ar@{-->>}[r] \ar@{>->}[d]_{\<1,\alpha\>} & R_\alpha \ar@{>-->}[d]\\
p_1^*E\times p_2^*E \ar[r] & E\times E.
}$$
Invertibility of $\alpha$ implies that $R_\alpha$ is symmetric, while reflexivity and transitivity follow from the cocycle conditions on $\alpha$.

Moreover, the quotient $E/\alpha$ is compatible with $q$ in the following sense. Since $\alpha$ is a map over $K$, we can factor through the kernel pair:
\raisebox{2ex}{$\xymatrix@R=1.5ex@C=2ex{p_1^*E \ar@{->>}[r] \ar@{-->}[rd] & R_\alpha \ar@<.5ex>[r] \ar@<-.5ex>[r]  & E \ar[d] \\
&K \ar@<.5ex>[r] \ar@<-.5ex>[r]   & A.
}$}
Therefore these upper maps coequalize $q$ and, since $p_1^*E\epi R_\alpha$ is epic, so do the maps $R_\alpha\rightrightarrows E\to A$. It follows that $R_\alpha$ factors through the kernel pair and this induces a map from $E/\alpha$ into $Q$:
$$\xymatrix{
R_\alpha \ar@<.7ex>[r] \ar@<-.7ex>[r] \ar@{-->}[d] & E \ar[d] \ar@{->>}[r] & E/\alpha \ar@{-->}[d]\\
K \ar@<.7ex>[r] \ar@<-.7ex>[r] & A \ar@{->>}[r]_q & Q.\\
}$$

To see that $\EE^{\2}$ is a stack, we need to check that $E\cong q^*(E/\alpha)$ so that the right-hand square is Cartesian. It is enough to show (following Johnstone \cite{elephant} Prop 1.2.1) that the comparison map $h:E\to q^*(E/\alpha)$ is both epic and monic (and hence an iso). For the first claim, consider the diagram:
$$\xymatrix{
p_1^* E \ar[r] \ar[d] & E \ar[r]^-h & q^*(E/\alpha) \ar[d] \ar[r] & A \ar[d]^q\\
E \ar@{->>}[rr] && E/\alpha  \ar[r] & Q.\\}$$
The right-hand square is evidently a pullback; the outer rectangle is as well, since it equals \raisebox{2.5ex}{$\xymatrix@=2ex{p_1^*E \ar[r] \ar[d] & K \ar[r] \ar[d] & A \ar[d] \\ E \ar[r] & A \ar[r] & Q}$}.

This guarantees that the left-hand square is a pullback, and that the map $p_1^*E\epi q^*(E/\alpha)$ is epic (since pullbacks preserve covers). Therefore the second factor $E\epi q^*(E/\alpha)$ is epic as well.

To show the map is monic consider the composite $E\stackrel{h}{\to} q^*(E/\alpha)\mono E/\alpha\times A$. Clearly, this is monic just in case $h$ itself is. Now suppose that we have parallel maps ${Z\rightrightarrows E}$ whose composites $Z\to E\to E/\alpha\times A$ are equal.

Equality in the first component shows that $Z$ factors through $R_\alpha$; equality in the second guarantees that the composite $Z\to R_\alpha \to K$ equalizes the projections $K\rightrightarrows A$. Therefore $Z$ factors through the equalizer $\Eq(K\rightrightarrows A)$, and this equalizer is just the diagonal $\Delta:A\to K$. This leaves us with the diagram below, which plainly shows that the two maps $Z\rightrightarrows E$ are equal:
$$\xymatrix{
&& Z \ar@{-->}[lld]_{z_1=z=z_2\ \ } \ar@{-->}[d] \ar[rrd]_{z_2} \ar@<1ex>[rrd]^{z_1} &&&\\
**[l] \Delta^*p_1^*E \cong E\ar[r] \ar[d] & p_1^*E \ar@{->>}[r] \ar[rd] & R_\alpha \ar[d] \ar@<.3ex>[rr] \ar@<-.7ex>[rr] && E\ar[d] \ar@{->>}[r] & E/\alpha\\
A \ar@<-1ex>@/_2ex/[rrrr]_{1_A} \ar[rr]^{\Delta} && K \ar@<.5ex>[rr] \ar@<-.5ex>[rr] && A &\\
}$$
\end{proof}

\begin{defn}[The structure sheaf]\label{def_str_sheaf}
Given a pretopos $\EE$, the \emph{structure sheaf} $\OO=\OO_{\EE}$ is the sheaf over $\MM(\EE,\kappa)$ associated with the stack $\EE^{\2}$ via the maps $\Stk(\EE)\stackrel{\widehat{\ }}{\lto}\Cat(\Sh(\EE))\simeq \EqSh(\MM)$ (cf. lemma \ref{stack_to_sheaf} and theorem \ref{eq_top_equiv}).
\end{defn}

As we have defined it, $\OO$ is an equivariant sheaf of pretoposes over $\MM$ but usually we will work up to equivalence, as if $\OO$ were a stack over $\EE$. This eliminates a good deal of bookkeeping and emphasizes the relationship to the codomain fibration (as opposed to its strict replacement $\widehat{\EE^{\2}}$). At the same time, this allows us to justify our formal manipulations while avoiding the additional machinery of equivariant stacks.

We can give a more concrete description of $\OO$ using the notion of relative equivariance. For any open subset $U\subseteq\MM_0$ there is a smallest parameter $k=k(U)$ such that $U\subseteq V_k$; we call $k$ the \emph{context} of $U$. $U$ is \emph{relatively invariant} if it is closed under those isomorphisms sending $k\mapsto k$. We may also call $U$ \emph{$k$-invariant} or, when $k=\emptyset$, simply \emph{invariant}.

By the definability theorem \ref{stab_thm} from Chapter 1, an invariant set $U$ which is compact for invariant covers is definable: $U=V_\varphi$ for some sentence $\varphi$. Any other invariant set is a union of these definable pieces. By the same token, a $k$-invariant $U$ which is compact for $k$-invariant covers has the form $V_{\varphi(k)}$ for some formula $\varphi(x)$. 

When $U$ is relatively invariant, the pair $\<U,V[k\mapsto k]\>$ defines an open subgroupoid of $\MM$. Given an equivariant sheaf $\<E,\rho\>$, we say that a subsheaf $S\leq E$ is \emph{relatively equivariant} if the image $S\epi U\subseteq \MM$ is relatively invariant and $S$ is equivariant with respect to the subgroupoid defined by $U$. In particular, an open section $s:V_{\varphi(k)}\to E$ is relatively equivariant if $s(\nu)=\rho_\alpha(s(\mu))$ whenever $\alpha(\mu(k))=\nu(k)$. We let $E(U)$ denote the set of relatively equivariant sections of $E$ over $U$, while $\Gamma_{\eq}(E)$ denotes the set of equivariant global sections $\MM_0\to E$.

\begin{prop}\label{str_sheaf_secs}
$\Gamma_{\eq}(\OO)$ is equivalent to $\EE$. More generally,
$$\OO(V_{\varphi(k)})\simeq\EE\!/\varphi.$$
\end{prop}

\begin{proof}
For the first claim we have the following sequence of equivalences
$$\begin{array}{rcl}
\EE\simeq \EE\!/1&\simeq&\uHom_{\EE}(y1,\EE^{\2})\\
&\simeq & \uHom_{\EE}(y1,\widehat{\EE^{\2}})\\
&\simeq & \uHom_{\MM}(\ext{1},\OO_{\EE}) = \Gamma_{\eq}(\OO_{\EE}).
\end{array}$$
Here we use the Yoneda lemma together with the fact that $\MM_0=\ext{1}$ is terminal in $\EqSh(\MM)$.

Similarly, suppose that we have a relatively equivariant section $s\in\OO(V_{\varphi(k)})$. By proposition \ref{equiv_ext}, there is a unique equivariant extension along the canonical section $\hat{k}$:
$$\xymatrix{
\ext{\varphi} \ar@{-->}[r]^{\exists ! \overline{s}} & \OO\\
& V_{\varphi(k)} \ar[u]_s \ar[ul]^{\hat{k}}\\
}$$
Since the equivalence $\EqSh(\MM)\simeq\Sh(\EE)$ sends $\ext{\varphi}$ to the representable sheaf $y\varphi$, this gives another chain of equivalences
$$\OO(V_{\varphi(k)})\simeq\uHom_{\MM}(\ext{\varphi},\OO)\simeq\uHom_{\EE}\left(y\varphi,\widehat{\EE^{\2}}\right)\simeq \EE\!/\varphi.$$
\end{proof}

\begin{lemma}\label{struct_stalks}
The stalk of $\OO$ at a labelled model $\mu$ is equivalent to the diagram of the model $M_\mu$ (cf. definition \ref{def_diagram}). An isomorphism $\alpha:\mu\to\nu$ acts on these fibers by sending $\<\sigma, a\>\longmapsto\<\sigma,\alpha(a)\>$. 
\end{lemma}
\begin{proof}
The partial sections $s\in\OO(V_{\varphi(k)})$ cover the structure sheaf, so we can use them to compute its stalks. Recall (lemma \ref{inclusion_lemma}) that an inclusion of basic open sets always has the form $V_{\psi(k,l)}\subseteq V_{\varphi(k)}$ where $\exists z.\psi(x,z)\vdash\varphi(x)$.

This corresponds to a map from $\psi\to\varphi$ in $\EE$; pullback along this map describes the transition functor $\OO(V_{\varphi(k)})\to\OO(V_{\psi(k,l)})$. This means that the stalk $\OO_\mu\simeq\displaystyle\colim_{\mu\models\varphi(k)} \EE\!/\varphi$ is a directed colimit of slices and pullbacks, i.e. a localization as in definition \ref{def_localization}.

This is nearly the definition of the diagram; however, $\DD(M_\mu)$ and $\OO_\mu$ are presented as colimits over different index categories. The former is indexed by elements $a\in\varphi^M$; the latter is indexed by labels $k\in\kappa$ with $\mu\models \varphi(k)$. Evaluation of parameters induces a comparison map $\OO_\mu\to\DD(M_\mu)$ sending $\<\sigma,k\>\mapsto\<\sigma,\mu(k)\>$.

Since $\mu$ is a surjective labelling, this functor is both surjective on objects and full. As we are working with pretoposes, conservativity will imply faithfulness and hence an equivalence of categories.

If $\<\sigma,\mu(k)\>\cong\<\tau,\mu(l)\>$ is an isomorphism in $\DD(M_\mu)$ then there is a diagram
$$\xymatrix{
D \ar[d]_{\sigma} & \ar[l] \raisebox{1.5ex}{$f^*\sigma \cong g^*\tau$} \ar[r] \ar/va(-8) 1.6cm/;[d] \ar/va(-6) 2.4cm/;[d] &E \ar[d]^{\tau}\\
\varphi & \ar[l]^{f} \gamma \ar[r]_{g} & \psi\\
}$$
together with an element $c\in \gamma^\mu$ with $f^\mu(c)=\mu(k)$ and $g^\mu(c)=\mu(l)$. This same diagram forces an isomorphism in $\OO_\mu$: for any label $\mu(j)=c$,
$$\<\sigma,k\>\cong\big\<f^*\sigma,j\big\>\cong\big\<g^*\tau,j\big\>\cong\<\tau,l\>.$$

Finally, we want to see that an isomorphism $\alpha:\mu\iso\nu$ acts by sending $\<\sigma,a\>\mapsto\<\sigma,\alpha(a)\>$. We proceed in two steps, first supposing that there is a parameter $k$ with $a=\mu(k)$ and $\alpha(a)=\nu(k)$. Since $\<\sigma,k\>$ is a $k$-equivariant section sending $\mu\mapsto \<\sigma,\mu(k)\>$ and $\alpha:k\mapsto k$, it follows that 
$$\rho_\alpha(\<\sigma,a\>)\cong\rho_\alpha(\<\sigma,k)(\mu))\cong\<\sigma,k\>(\nu)\cong\<\sigma,\alpha(a)\>.$$ 

For the general case, fix labels $\mu(j)=a$ and $\nu(l)=\alpha(a)$. Using the reassignment lemma \ref{reassignment} we may find a diagram of labelled models satisfying the following conditions (with $k$ disjoint from $j$ and $l$):
$$\xymatrix{
\mu \ar[r]^{\alpha} \ar[d]_{\beta} & \nu & \mu'(k)=\mu'(j)=\beta(a)\\
\mu' \ar[r]_{\alpha'} & \nu' \ar[u]_{\beta'}& \beta'(\nu(k))=\beta'(\nu(l))=\alpha(a)\\
}$$
Now $\beta:j\mapsto j$, $\alpha':k\mapsto k$ and $\beta':l\mapsto l$, the previous observation applied three times gives
$$\begin{array}{rcl}
\rho_\alpha(\<\sigma,a\>)&\cong&\rho_{\beta'}\circ\rho_{\alpha'}\circ\rho_{\beta}(\<\sigma,\mu(k)\>)\\
&\cong&\rho_{\beta'}\circ\rho_{\alpha'}(\<\sigma,\mu'(k)\>)\\
&\cong&\rho_{\beta'}\circ\rho_{\alpha'}(\<\sigma,\mu'(j)\>)\\
&\cong&\rho_{\beta'}(\<\sigma,\nu'(j)\>)\\
&\cong&\rho_{\beta'}(\<\sigma,\nu'(l)\>)\\
&\cong&\<\sigma,\nu(l)\>\cong\<\sigma,\alpha(a)\>\\
\end{array}$$
\end{proof}

\begin{cor}[Subdirect product representation]\label{sub_dir_prod}
Every pretopos $\EE$ embeds (conservatively) into a product of local pretoposes.
\end{cor}

\begin{proof}
Lemma \ref{struct_stalks} showed that the stalks of $\OO=\OO_{\EE}$ are local pretoposes. Thus $\prod_{\mu\in\MM} \OO_\mu$ is a product of local pretoposes. Lemma \ref{str_sheaf_secs} showed that $\EE$ is equivalent the equivariant global sections of $\OO_{\EE}$. Evaluating these sections at each stalk defines a pretopos functor $\EE\to\prod_{\mu\in\MM} \OO_\mu$. This is clearly conservative: if $\varphi^\mu\cong\psi^\mu$ for every $\mu\in\MM$ then the equivalence $\EE\simeq\Gamma_{\eq}\OO$ ensures that $\varphi\cong\psi$ in $\EE$.
\end{proof}

We collect the results of this section into a theorem:
\begin{thm}[Equivariant sheaf representation for pretoposes]\label{aff_schemes}
For every pretopos $\EE$ there is a topological groupoid $\MM=\MM_{\EE}$ and an equivariant sheaf of pretoposes $\OO=\OO_{\EE}$ over $\MM$ such that $\Gamma_{\eq}\OO\simeq\EE$ and, for each $\mu\in\MM_0$, the stalk $\OO_\mu$ is a local pretopos (cf. definition \ref{def_local}).
\end{thm}

\begin{defn}[Affine schemes]
The pair $\<\MM_{\EE},\OO_{\EE}\>$ is the \emph{affine (logical) scheme} associated with $\EE$, denoted $\Spec(\EE)$.
\end{defn}

In particular, we may specialize to the case of classical first-order logic by considering Boolean pretoposes.
\begin{cor}
For every classical first-order theory $\bTT$ there is a topological groupoid of (labelled) $\bTT$-models and isomorphism $\MM=\MM_{\bTT}$ and an equivariant sheaf of Boolean pretoposes $\OO=\OO_{\bTT}$ over $\MM$ such that $\Gamma_{\eq}\OO\simeq\EE_{\bTT}$. For each model $\mu\in\MM_0$, the stalk $\OO_\mu$ is a well-pointed pretopos, the classifying pretopos of the complete diagram of $\mu$.
\end{cor}
\begin{proof}
As discussed in Chapter 2, section \ref{sec_bool_ptop}, every classical theory has a classifying pretopos $\EE_{\bTT}$ which is Boolean and we apply the previous theorem. The sheaf of pretoposes $\OO_{\bTT}$ is Boolean because complements are preserved by slicing. The stalks are well-pointed by lemma \ref{bool_stalks}.
\end{proof}

\begin{cor}
Every Boolean pretopos $\BB$ embeds into
\begin{itemize}
\item a product of well-pointed pretoposes.
\item a power of $\Sets$
\end{itemize}
\end{cor}

\begin{proof}
The first claim follows corollary \ref{sub_dir_prod} together with the fact (lemma \ref{bool_stalks}) that any local, Boolean pretopos is well-pointed.

For the second claim, notice that the global sections $\Hom(1,-)$ in a well-pointed pretopos define a faithful (and hence conservative) functor into sets. Composing these functors with the previous claim yields the asserted embedding:
$$\xymatrix{
\BB \ar@{>->}[r] \ar@{>->}[rd] & \prod_{\mu\in\MM} \OO_{\BB,\mu} \ar@{>->}[d]^{\prod_\mu \Hom_{\OO_{\BB,\mu}}(1,-)}\\
& \prod_{\mu\in\MM} \Sets.
}$$
\end{proof}

\section{The category of logical schemes}

In the previous section we defined the affine scheme $\<\MM_{\EE},\OO_{\EE}\>$ associated with a pretopos $\EE$. This gave a representation of $\EE$ as the equivariant global sections of $\OO_{\EE}$. In this section we will show that this construction is functorial in $\EE$; pretopos functors and natural transformations lift to the structure sheaves, and these are recovered from their global sections. This guides the definition of the 2-category of logical schemes.

\begin{prop}\label{embed_prop}
\mbox{}\begin{itemize}
\item An interpretation $I:\EE\to\FF$ (contravariantly) induces a reduct functor $I_\flat:\MM_{\FF}\to\MM_{\EE}$.
\item $I$ also defines an internal pretopos functor $I^\sharp:\OO_{\EE}\to I_{\flat*}\OO_{\FF}$ such that $\Gamma_{\eq}(I^\sharp)\cong I$.
\item Every stalk of the transposed map $I_\flat^*\OO_{\EE}\to\OO_{\FF}$ is a conservative functor.
\item For any natural transformation $\xymatrix{\EE \rtwocell^{\2}_J{\tau} & \FF}$ there is an internal pretopos functor $\tau^*:J_{\flat*}\OO_{\FF}\to I_{\flat*}\OO_{\FF}$ (in the opposite direction) and a natural transformation
$$\xymatrix{
 \OO_{\EE}\ar[rr]^{I^\sharp} \ar[dr]_{J^\sharp} \rrlowertwocell<\omit>{<3>\ \ \tau^\sharp} && I_{\flat*}\OO_{\FF}\\
& J_{\flat*}\OO_{\FF} \ar[ru]_{{\tau^*}} &  \\
}$$
such that $\Gamma_{\eq}\tau^*\cong 1_{\FF}$ and $\Gamma_{\eq}\tau^\sharp\cong\tau$.
\end{itemize}
\end{prop}

\begin{proof}
Recall that the reduct of an $\FF$-model $N$ along $I$ is defined to be the composite $I_\flat N:\EE\to\FF\stackrel{N}{\to}\Sets$. The reduct of an isomorphism $\xymatrix{\FF \rtwocell^N_{N'}{\omit\vertiso} & **[r]\Sets}$ is defined in the same way. Since $A^{I_\flat N}=(IA)^N$, the reduct $I_\flat N$ inherits labellings from $N$, defining a functor $I_\flat:\MM_{\FF}\to\MM_{\EE}$.

Given a basic open set $V_{\varphi(k)}\subseteq \MM_{\EE}$, the inverse image along $I_\flat$ is $V_{I\varphi(k)}$, so $I_\flat$ is continuous on the space $\Ob(\MM_{\EE})$. The map $\Ar(\MM_{\FF})\to\Ar(\MM_{\EE})$ is continuous because open sets of morphisms are defined by labels, which are inherited. Hence $I_\flat:\MM_{\FF}\to\MM_{\EE}$ is a continuous functor and induces a geometric morphism
$\xymatrix{
\EqSh(\MM_{\FF}) \rrtwocell^{I_\flat^*}_{I_{\flat *}}{`\bot} && \EqSh(\MM_{\EE}).\\
}$

In order to define $I^\sharp$ we consider the relatively equivariant sections of $I_{\flat*}\OO_{\FF}(V_{\varphi(k)})$, These are equivalent to the sections of $\OO_{\FF}(V_{I\varphi(k)})$ and hence to $\FF/I\varphi$. $I^\sharp$ is then defined section-wise by obvious functors $I_\varphi:\EE\!/\varphi\to\FF/I\varphi$. In particular, setting $\varphi=1$ gives $\EE\!/1\simeq\EE$ and $\Gamma_{\eq} I^\sharp\cong I:\EE\to\FF$.

Recall that the stalk of $\OO_{\EE}$ at a labelled model $\mu$ is the diagram of $\mu$; its objects (modulo coproducts and quotients) are parameterized formulas of the form $\varphi(x,b)$ where $\varphi(x,y)$ is a formula in context $A\times B$ and $b$ is an element in $B^\mu$. Since pullbacks preserve fibers, the stalk of $\OO_{\EE}$ at $I_\flat(\nu)$ is the same as the stalk of $I_\flat^*\OO_{\EE}$ at $\nu$. Moreover, by the definition of the reduct $I_\flat(\nu)$ we have $B^{I_\flat(\nu)}=(IB)^\nu$, allowing us to regard the element $b\in B^{I_\flat(\nu)}$ as an element of $\nu$. Thus (at the level of stalks) the transposed map $I_\flat^*\OO_{\EE}\to\OO_{\FF}$ maps the diagram of the reduct $I_\flat(\nu)$ into the diagram of $\nu$ by sending $\varphi(x,b)\mapsto I\varphi(x,b)$. Since the reduct $I_\flat(\nu)$ satisfies a sequent $\varphi(x,b)\vdash\psi(x,b')$ just in case $\nu$ satisfies $I\varphi(x,b)\vdash I\psi(x,b')$, this map is obviously conservative on stalks.

Similarly, suppose $I,J:\EE\to\FF$ and $\tau:I\Rightarrow J$ is a natural transformation. Pullback along $\tau_\varphi$ induces a pretopos functor $\tau_\varphi^*:\FF/J\varphi\to\FF/I\varphi$ while $\tau^\sharp$ is induced by the naturality square for $E\in\EE\!/\varphi$:
$$\xymatrix{
IE \ar@/^3ex/[rr]^{\tau_E} \ar[d] \ar[r]_-{\tau_E^\sharp} & \tau_\varphi^*(JE) \ar[r] \ar[ld]& JE \ar[d]& 
\EE\!/\varphi \ar[rr]^{I/\varphi} \ar[dr]_{J/\varphi} \rrlowertwocell<\omit>{<3>\ \ \tau^\sharp_\varphi} && \FF/I\varphi\\
I\varphi \ar[rr]_{\tau_\varphi} && J\varphi 
&& \FF/J\varphi \ar[ru]_{{\tau_\varphi^*}} &  \\
}$$
Since $I_{\flat*}\OO_{\FF}(V_{\varphi(k)}\simeq\OO_{\FF}(V_{I\varphi(k)})\simeq\FF/I\varphi$ and $J_{\flat*}(V_{\varphi(k)})\simeq\FF/J\varphi$, this gives a section-wise description of the transformation $\tau^\sharp$ asserted in the proposition.

Taking global sections of the transformation amounts to setting $\varphi=1$ in the diagrams above. Since $I$ and $J$ preserve the terminal object, $\tau_1^*$ is the identity $1_{\FF}$ and the component of $\Gamma_{\eq}\tau^\sharp$ at $E\in\EE\!/1$ is just $\tau_E$:
$$\xymatrix{
IE \ar[r]^{\tau_E^\sharp=\tau_E} \ar[d] & JE \ar@{=}[r] \ar[ld]& JE \ar[d]& 
\EE \ar[rr]^{I} \ar[dr]_{J} \rrlowertwocell<\omit>{<3>\ \ \tau} && \FF\\
1_{\FF} \ar[rr]_{!} && 1_{\FF}
&& \FF \ar@<-1ex>@{=}[ru]_{1_{\FF}^*} &  \\
}$$
Therefore, modulo the canonical equivalences $\Gamma_{\eq}\OO_{\EE}\simeq\EE\!/1\cong\EE$ and $\Gamma_{\eq}\OO_{\FF}\simeq\FF$, $\Gamma_{\eq}\tau^\sharp$ is isomorphic to $\tau$.
\end{proof}

\begin{defn}\label{def_ax_space}
We adapt the following definitions from algebraic geometry (cf. Hartshorne \cite{hartshorne}):
\begin{itemize}
\item An \emph{axiomatized space} $\XX$ is a topological groupoid $X$ together with an sheaf of pretoposes $\OO_{\XX}\in\EqSh(X)$.
\item $\XX$ is \emph{locally axiomatized} if, for each $x\in X$, the stalk $(\OO_{\XX})_x$ is a local pretopos (i.e., satisfies the existence and disjunction properties cf. definition \ref{def_local}).
\item A \emph{morphism of axiomatized spaces} $\XX\to\YY$ is a pair $\<F_\flat,F^\sharp\>$ where $F_\flat:X\to Y$ is a continuous functor of topological groupoids and $F^\sharp:\OO_{\YY}\to F_{*}\OO_{\XX}$ in an internal pretopos functor in $\EqSh(Y)$.
\item A \emph{morphism of locally axiomatized spaces} is a morphism of axiomatized spaces such that the transposed map $F_\flat^*\OO_{\YY}\to\OO_{\XX}$ is conservative on every stalk.\footnote{Notice that the disjunction property implies that the subobjects of 1 in a local pretopos contain a unique maximal ideal. Conservativity on stalks amounts to the requirement that the maximal ideal in the stalk of $\OO_{\YY}$ at $F_\flat(x)$ maps into the maximal ideal of the stalk of $\OO_{\XX}$ at $x$.}
\item A morphism $F:\XX\stackrel{\sim}{\lto} \YY$ is an \emph{equivalence of axiomatized spaces} if $F_\flat$ and $F^\sharp$ are both full, faithful and essentially surjective.\footnote{The conditions on $F^\sharp$ should be interpreted internally. An internal functor $F:\CC\to\DD$ is (internally) full and faithful if the square below is a pullback, and essentially surjective if the displayed map on the right is a regular epimorphism:
$$\xymatrix@=0ex{\CC_1 \pbcorner \ar[dd] \ar[rr]^{F_1} &\left.\ \ \ \right.& \DD_1 \ar[dd] \\
&&&\left.\ \ \right.& \raisebox{-3ex}{$\CC_0\overtimes{D_0}\rm{Iso}(\DD_1)$} \ar@{->>}[rr] &\left.\ \ \ \right.&\DD_0\\
\CC_0\times\CC_0 \ar[rr]_{F_0\times F_0} && \DD_0\times\DD_0
}$$
}
\item A \emph{2-morphism of axiomatized spaces} $\xymatrix{\XX \rtwocell^F_G{\tau} & \YY}$ is a pair $\<\tau^*, \tau^\sharp\>$ where
\begin{itemize}
\item $\tau^*$ is an internal pretopos functor $G_{\flat*}\OO_{\FF}\to F_{\flat*}\OO_{\FF}$ (in the opposite direction) such that $\Gamma\tau^*\cong 1_{\FF}$.\footnote{Note that this condition is justified by the fact that $\Gamma_{\eq}^{\YY}\circ F_{*}\simeq \Gamma_{\eq}^{\XX}\simeq \Gamma_{\eq}^{\YY}\circ G_{*}$.}
\item $\tau^\sharp$ is a natural transformation
$$\xymatrix{
\OO_{\YY} \ar[rr]^{F^\sharp} \ar[dr]_{G_\sharp} \rrlowertwocell<\omit>{<3>\ \ \tau^\sharp} && F_{\flat*}\OO_{\XX}.\\
& G_{\flat*}\OO_{\XX} \ar[ru]_{{\tau^*}} & \\
}$$
\end{itemize}\end{itemize}\end{defn}

Theorem \ref{aff_schemes} states that the affine scheme $\Spec(\EE)=\<\MM_{\EE},\OO_{\EE}\>$ associated with a pretopos $\EE$ is a locally axiomatized space. Equivariant global sections provide a functor $\Gamma_{\eq}:\bf{AxSp}\to\Ptop$, and proposition \ref{embed_prop} guarantees that the composite $\Gamma\circ\Spec$ is equivalent to $1_{\Ptop}$. Now we will define a logical scheme to be a locally axiomatized space which is covered by affine pieces.

\begin{defn}\mbox{}\begin{itemize}
\item Suppose that $\XX$ is an axiomatized space and $U\subseteq X$ is a (non-full) subgroupoid (with the subspace topology). The associated \emph{axiomatized subspace} $\UU\subseteq\XX$ is defined by restricting $\OO_{\XX}$ to the object space $U_0$, equipped with the equivariant action inherited from $U_1\subseteq X_1$.
\item $\UU\subseteq\XX$ is an open subspace if $U_0\subset X_0$ and $U_1\subseteq X_1$ are open sets.
\item A family of open subgroupoids $\{U^i\subseteq X\}$ is an \emph{open cover} of $X$ if any $\alpha:x\to x'$ in $X_1$ has a factorization in $\bigcup_i U^i$:
$$\xymatrix{
x  =z_0 \ar@/^4ex/[rrr]^\alpha \ar[r]_-{\beta_1} & z_1 \ar[r]_-{\beta_2}&\ldots \ldots\ar[r]_-{\beta_n}&z_n = x',& \beta_j\in \bigcup_i U^i_1.\\
}$$
\end{itemize}\end{defn}

\begin{defn}
A \emph{logical scheme} is a locally axiomatized space $\XX=\<X,\OO_{\XX}\>$ for which there exists an open cover by affine subspaces $\UU^i\simeq\Spec(\EE^i)$.

This defines a full 2-subcategory $\LSch\subseteq\bf{AxSp}$.
\end{defn}

In order to lighten the notation when working with schemes we will write $\Gamma$ in place of $\Gamma_{\eq}$. We also drop the $(-)_\flat$ notation from definition \ref{def_ax_space} and use the same letter $F$ to denote a scheme morphism $\XX\to\YY$ and the underlying continuous groupoid homomorphism $X\to Y$.

Recall that the (sheaf) descent category for an open cover $J=\{U_i\subseteq X\}$ is defined as follows. Set $U=\coprod_i U_i$ and let $q$ denote the resulting map $U\to X$. An object of $\Desc(J)$ is a pair $\<E,\alpha\>$ where $E\in \EqSh(U)$ and $\alpha:p_1^*E\cong p_2^*E$ is an isomorphism over $U\overtimes{X}U$ (satisfying the usual cocycle conditions). We close with the following lemma, which allows us to represent structure sheaves using descent.

\begin{lemma}\label{eff_desc}
If $\XX$ is a scheme and $J=\{U_i\subseteq X\}$ is an open cover then the induced geometric morphism $j:\prod_i\EqSh(U_i)\to \EqSh(X)$ is an open surjection in the topos theoretic sense: the inverse image functor $q^*$ is faithful and its localic reflection is an open map.

In particular, $\EqSh(X)\simeq\Desc(J)$.
\end{lemma}
  
\begin{proof}
When an equivariant sheaf $\<E,\rho\>$ is described as an \'etale space, the inverse image $q^*\<E,\rho\>$ is defined by pulling $E$ back to $U_0$ and factoring $U_1$ through the pullback of $\rho$:
$$\xymatrix{
U_1\overtimes{U_0} q^*E \ar[r] \ar[rd]_{q^*\rho} & **[r] q^*(X_1\overtimes{X_0}E) \pbcorner \ar[rr] \ar[d] && X_1\overtimes{X_0}E \ar[d]_\rho \\
& q^*E \pbcorner \ar[rr] \ar[d] && E \ar[d]\\
& (U)_0 \ar[rr]_{q_0} && X_0
}$$

This is certainly faithful. If $g\not=h:E\to F$ are distinct maps in $\EqSh(X)$ there is a point $x$ such that the stalks disagree: $g_x\not=h_x$. Since $J$ is an open cover, $x$ belongs to one of the subgroupoids $U_i$; let $x_i$ denote the associated point in $U$. Since $q^*$ preserves stalks we have $(q^*g)_{x_i}\not=(q^*h)_{x_i}$ and therefore $q^*g\not=q^*h$.

The localic reflection of $\EqSh(X)$ is a topological space $S_X$ whose open sets are the $X$-invariant open sets $U\subseteq X_0$. For any set $P\subseteq S_X$ there is a smallest $X$-invariant set $\overline{P}$; we can define it as the image of the composite $P\overtimes{X_0}X_1 \to X_1 \to X_0$.

Because $\XX$ is a scheme, $X$ is an open topological groupoid, in the sense that the domain and codomain $X_1\rightrightarrows X_0$ are open maps. This is true for affine schemes (cf. definition \ref{M1points}) and it is a local condition. This means that if $V\subseteq X_0$ is open, then so is $\overline{V}$.

In order to show that the localic reflection is open it is enough to check that the restriction of $q^*$ to $\OO(S_X)$ has a left adjoint $q_!\dashv q^*$. We define $q_!$ by sending a family of $U_i$-equivariant sheaves $W=\<W_i\>$ to the union $\bigcup_i \overline{W_i}$. If $V$ is an $X$-equivariant open then $q^*V=\<V\cap (U_i)_0\>$, giving us the following sequence of equivalences:
$$\begin{array}{rcll}
W=\<W_i\> &\leq & \<V\cap (U_i)_0\>=q^*V& \\
\hline\hline
\forall\ i,\ W_i & \leq & V\\
\hline\hline
\forall\ i,\ \overline{W_i} & \leq & V \\
\hline\hline
q_!W=\bigcup_i\overline{W_i} & \leq & V\\
\end{array}$$

For the last statement of the lemma we appeal to a theorem of Joyal and Tierney \cite{JT}. They proved that open surjections are effective descent morphisms for sheaves, which means precisely that $\EqSh(X)\simeq \Desc(q)$.
\end{proof}

\section{Open subschemes and gluing}\label{sec_subsch}

In this section we will prove that schemes are local, in the sense that any open subspace of a scheme is again a scheme. Then we show that schemes which share an open subscheme $\XX\supseteq \UU\simeq \VV\subseteq \YY$ can be glued to give a pushout $\XX\overplus{\UU}\YY$. These will give us our first examples of non-affine schemes.

Recall that an affine scheme $\Spec(\EE)$ has a basis of open subsets $V_{\varphi(k)}$ where $\varphi(x_1,\ldots,x_n)\leq\Pi_i A_i$ is a formula in context and $k=(k_1,\ldots,k_n)\in\kappa^n$ is a sequence of parameters. This is the set of labelled $\EE$-models $\mu$ such that $\mu\models\varphi(k)$ (cf. definitions \ref{M0points} \& \ref{M0opens}). There is also an open set of morphisms $V_{k\mapsto k}\rightrightarrows V_{\varphi(k)}$, defining a subgroupoid $U_{\varphi(k)}$. We let $\VV_{\varphi(k)}$ denote the associated open subspace of $\Spec(\EE)$.

\begin{lemma}\label{slice_subsch}
\mbox{}
\begin{itemize}
\item If $\XX=\Spec(\EE)$ is affine then each open subspace $\VV_{\varphi(k)}$ is again affine, with $\VV_{\varphi(k)}\simeq\Spec(\EE\!/\varphi)$.
\item If $\XX$ is a scheme then any open subspace $\UU\subseteq\XX$ is again a scheme.
\end{itemize}\end{lemma}

\begin{proof}
For simplicity, suppose that $\varphi(x)$ is a unary formula; the same argument will apply to the general case. Thus we have a variable $x:A$ and a subobject $\varphi\leq A$. The open subset $V_{\varphi(k)}$ is the set of labelled $\EE$-models $\mu$ such that $\mu(k)\in A^\mu$ is defined and the underly model $M_\mu\models \varphi(\mu(k))$. On the other hand, a model of the slice category $\EE\!/\varphi$ is a pair $\<M,a\>$ where $M$ is an $\EE$-model, $a\in A^M$ and $M\models\varphi(a)$.

We define an equivalence $J:\VV_{\varphi(k)}\iso \Spec(\EE\!/\varphi)$ by sending a labelled $\EE$-model $\mu$ to the pair $\<\mu,\mu(k)\>$. Given any model $M$ and any $a\in A^M$ there is some labelling $\mu$ such that $\mu(k)=a$, so this map is essentially surjective. Similarly, an isomorphism of $\EE\!/\varphi$-models $\<M,a\>\iso \<M',a'\>$ is simply an isomorphism $\alpha:M\cong M'$ such that $\alpha(a)=a'$. Given labellings of these models with $\mu(k)=a$ and $\mu'(k)=a'$, this says precisely that $\alpha:k\mapsto k$. Thus $J$ is full and faithful.

Now note that $J$ is an open map. An open subset of $V_{\varphi(k)}$ has the form $V_{\psi(k,l)}$ where $\psi(x,y)\vdash \varphi(x)$ (cf. the inclusion lemma \ref{inclusion_lemma}). Now we associate the parameters $\<k,l\>$ (defined for $A$ and $B$ in $\EE$) with a new parameter $l_k$ defined for $\varphi\times B$ in $\EE\!/\varphi$. Since $\psi$ is a subobject of $\varphi\times B$, this defines an open subset $V_{\psi(l_k)}\subseteq\Spec(\EE\!/\varphi)$ and the equivalence $(\OO_{\EE})|_{V_{\varphi(k)}}\simeq\OO_{\EE\!/\varphi}$ then follows from
$$\OO_{\EE}(V_{\psi(k,l)})\simeq \EE\!/\psi\simeq (\EE\!/\varphi)/\psi\simeq \OO_{\EE\!/\varphi}(V_{\psi(k_l)}).$$

From this we can see that if $\XX$ is a scheme then so is any open subspace $\UU\subseteq\XX$. For any point $x\in\UU$ has an affine neighborhood $x\in\VV\subseteq\XX$ with $\VV\simeq\Spec(\FF)$. Since $\UU\cap\VV$ is again an open set, there is a basic open $V_{\varphi(k)}\subseteq\UU\cap\VV$ (where $\varphi$ is an $\FF$-formula). Then every point $x$ has an affine neighborhood  $x\in\VV_{\varphi(k)}\subseteq\UU$, so the objects of $\UU$ are covered by affine pieces.

Similarly, any isomorphism $\alpha:x\iso x'$ in $\UU$ can be factored as $\alpha=\beta_n\circ\ldots\circ\beta_1$ with $\beta_i:\mu_i\to\mu_{i+1}$ in $\VV_i\cong\Spec(\FF_i)$. Each intersection $\VV_i\cap\UU$ is again open and so contains a basic open neighborhood $\beta_i\in V_{k\mapsto l}$ (where the parameters $k$ and $l$ are associated with $\FF_i$).

Now, as in the proof of lemma \ref{struct_stalks}, we factor each $\beta_i$ into a trio of maps $\mu_i\to\nu_i\to\nu_i'\to\mu_{i+1}$ such that
$$k \stackrel{\beta_{i_1}}{\longmapsto} k\underset{\nu_i}{=} k' \stackrel{\beta_{i_2}}{\longmapsto} k\underset{\nu_i'}{=} l \stackrel{\beta_{i_3}}{\longmapsto} l.$$
Because $\VV_i\cap\UU$ is open, we may assume that this factorization lies inside $\UU$. Each $\beta_{i_j}$ preserves some parameter and therefore belongs to one of the affine subschemes $\VV_k$, $\VV_{k'}$ or $\VV_l$. Therefore the morphisms of $\UU$ are also covered by affine pieces and $\UU$ is a scheme.
\end{proof}

\begin{lemma}[Gluing Lemma]\label{gluing_lemma}
Logical schemes admit gluing along isomorphisms.
\end{lemma}
\begin{proof}
First we note that the category of logical schemes has (disjoint) coproducts. Given schemes $\XX$ and $\YY$ the spectrum for the coproduct is just the disjoint sum of groupoids $X+Y$. Described as the total space of an \'etale map, the structure sheaf is also a disjoint sum:
$$\OO_{\XX+\YY}\cong\OO_{\XX}+\OO_{\YY}\lto X+Y.$$
Given a pair of maps $\XX\stackrel{F}{\lto}\ZZ\stackrel{G}{\lfrom}\YY$, the following diagram indicates the induced map $\<F,G\>^\sharp:\<F,G\>^*\OO_{\ZZ}\to\OO_{\XX+\YY}$:
$$\xymatrix{
\TT_{\ZZ} \ar[d] & \ar[l] \pbrcorner F^*\TT_{\ZZ}+G^*\TT_{\ZZ} \ar[d] \ar[r]^-{F^\sharp+G^\sharp} & \TT_{\XX}+\TT_{\YY} \ar[d]\\
\MM_{\ZZ} & \ar[l]_-{\<F,G\>} \MM_{\XX}+\MM_{\YY} \ar@{=}[r]^-\sim&\MM_{\XX+\YY}.\\
}$$
Infinite sums are handled in the same fashion. These coproducts constitute a special case of gluing along the empty subscheme.

Next suppose that $\XX$ and $\YY$ have a common open subscheme: $\UU\subseteq\XX$ and $\UU\subseteq\YY$. We want to describe the glued space $\ZZ=\XX\overplus{\UU}\YY$. The first step is to construct the pushout of groupoids $Z=X\overplus{U}Y$. Because the forgetful functor $\Top\to\Sets$ lifts limits and colimits, the underlying category of $Z$ (sans topology) is the ordinary pushout of groupoids.

The space of objects is just the ordinary quotient $Z_0=(X_0+Y_0)/U_0$. The arrows of $Z$ are generated (under composition) by those of $X$ and $Y$. When it is defined, composition is computed as in $X$ or $Y$. Otherwise composition is formal, subject to the following equivalence relation:
$$\left.\right.[\beta]\underset{Z}{\circ}\left[\gamma\underset{X}{\circ}\alpha\right]\sim \left[\beta\underset{Y}{\circ}\gamma\right]\underset{Z}{\circ} [\alpha]$$
$$\xymatrix{x \ar[rr]^{\alpha\in X} && u \ar[rr]^{\gamma\in U} && u' \ar[rr]^{\beta\in Y} && y}.$$

This formal composition defines a function from the space of $U$-composable pairs $X\underset{U}{*}Y$ (or composable $n$-tuples) into $Z_1$, as below. The topology on $Z_1$ is defined by requiring that all of these are open maps:
$$\xymatrix@R=1ex{
X_1 \ar[d]_{\cod} & \bullet \pbrcorner \ar[l] \ar[d] & \raisebox{-3ex}{$X\underset{U}{*}Y$} \pbrcorner \ar[d] \ar[l]  \ar[r]^-{\circ}  & Z_1\\
X_0 & U_0 \ar[l] \ar[dd] & \bullet \pbrcorner \ar[l] \ar[dd]\\\\
& Y_0 & Y_1 \ar[l]^{\dom}\\
}$$
For example, any two open neighborhoods $V_X\subseteq X_1$ and $V_Y\subseteq Y_1$ define a neighborhood $V_Y\circ V_X=\circ\left(V_Y\underset{U}{*} V_X\right)\subseteq Z_1$. If $\alpha\in V_X$ and $\beta\in V_Y$ are composable then $\beta\circ\alpha\in V_Y\circ V_X$.

It is not difficult to see that this is a pushout of continuous groupoids. Suppose that $F:X\to Z'$ and $G:Y\to Z'$ are continuous groupoid homomorphisms which agree on $U$. $Z$ is a pushout in groupoids, so this induces a unique functor $H:Z\to Z'$. It is enough to check that this comparison map is continuous.

Because $F$ and $G$ agree on $U$, composition in $Z'$ induces a continuous functor $F\circ G:X\underset{U}{*}Y\to Z'$ (and similarly for longer composable strings).  Given an open set $W\subseteq Z'$, the inverse image $H^{-1}(W)\subseteq Z$ is a union of the inverse images $F^{-1}(W)$, $G^{-1}(W)$, $(F\circ G)^{-1}(W)$, etc. Since these are all continuous $H^{-1}(W)$ is a union of open sets, and $H$ is continuous.

Now that we have the underlying groupoid $Z$ we can define the structure sheaf $\OO_{\ZZ}$. Notice that the isomorphism $\alpha:\OO_{\XX}|_{U}\cong \OO_{\YY}|_U$ defines a object of descent data $\<\OO_{\XX},\OO_{\YY},\alpha\>$ for the binary cover $X+Y\to Z$. Since an open cover is effective descent (lemma \ref{eff_desc}), this defines the sheaf $\OO_{\ZZ}$ up to isomorphism.

The stalks of $\OO_{\ZZ}$ are inherited from $\OO_{\XX}$ and $\OO_{\YY}$ while the equivariant action is a composite $\rho_{\XX}$ and $\rho_{\YY}$:
$$\xymatrix{
\gamma:& x \ar[r]^\alpha & u \ar[r]^\beta& y\\
\rho_\gamma:& \OO_{\mathcal{Z},x}\cong\OO_{\XX,x} \ar[r]^-{\rho_{\XX,\alpha}}& \OO_{\XX,u}\cong \OO_{\UU,u}\cong\OO_{\YY,u} \ar[r]^-{\rho_{\YY,\beta}}&\OO_{\YY,y}\cong\OO_{\mathcal{Z},y}.\\
}$$
This defines a new scheme $\mathcal{Z}=\XX\overplus{\UU}\YY$ which contains open subschemes $\XX$, $\YY$ and $\UU\cong\XX\cap\YY$.

More generally, suppose that we have a family of schemes $\XX_i$, a family of open subschemes $\UU_{ij}\subseteq \XX_i$ and isomorphisms $h_{ij}:U_{ij}\cong U_{ji}$ which satisfy the cocycle conditions: $h_{ij}=h_{ji}^{-1}\ \ \rm{and}\ \ h_{ij}\circ h_{jk}=h_{ik}$ (when the latter is defined). We can glue these to form a new scheme $\mathcal{Z}=\displaystyle\bigoplus_{U_{ij}} \XX_i$.

The definition of $\mathcal{Z}$ is entirely analogous to the binary case. The underlying groupoid $Z$ is a pushout in the category of groupoids, consisting of strings of composable arrows and topologized by open maps from the spaces of such tuples. The structure sheaf on $\ZZ$ is defined locally by descent from the cover $\coprod_i X_i\to Z$. More formally, this shows that the usual descent diagram has a colimit in $\LSch$ whenever each $\UU_{ij}\mono\XX_i$ is an open subscheme:
$$\xymatrix{
\displaystyle\coprod\UU_{ijk} \ar@<1.5ex>[r] \ar[r] \ar@<-1.5ex>[r] &
\displaystyle\coprod\UU_{ij} \ar@<1.5ex>@{>->}[r] \ar@<-1.5ex>@{>->}[r]&
\displaystyle\coprod\XX_i \ar[l] \ar@{->>}[r]&
\raisebox{-.75cm}{$\displaystyle\ZZ=\bigoplus_{\UU_{ij}} \XX_i$}.\\
}$$
\end{proof}

The gluing lemma provides our first examples of non-affine schemes. In fact, the coproduct $\XX=\Spec(\EE)+\Spec(\FF)$ is already non-affine. An equivariant section on $\XX$ is just a section of $\Spec(\EE)$ together with a section of $\Spec(\FF)$, so that $\Gamma\XX\simeq \EE\times\FF$. In general (the diagram of) an $(\EE\times\FF)$-model is not equivalent to the diagram of either an $\EE$-model or an $\FF$-model, which means that $\Spec(\Gamma\XX)\not\simeq \XX$. We will see in the next section that this is a general characterization of non-affine schemes.

\section{Limits of schemes}

In this section we will prove that the category of logical schemes is closed under finite limits, and that these can be computed from the colimits of pretoposes. This is analogous to the situation in algebraic geometry, where the affine line $A^1$ is (dual to) the ring of polynomials $\bCC[x]$ and the plane $A^1\times A^1$ corresponds to the \emph{co}product $\bCC[x,y]\cong\bCC[x]\overplus{}\bCC[y]$.

\begin{thm}
There is an (op-)adjunction
$\Ptop(\EE,\Gamma\XX)\simeq\LSch(\XX,\Spec(\EE))$
presenting pretoposes as a reflective (op-)subcategory of logical schemes
$$\xymatrix{
\Ptop^{\op} \rrtwocell_{\Spec}^{\Gamma}{`\bot} && \LSch.\\
}$$
\end{thm}

\begin{proof}
We will describe the adjunction in terms of its unit and counit. Because this is an ``adjunction on the right'', this amounts to a pair of natural transformations
$$\begin{array}{c}
\eta=\eta_{\XX}: \XX\lto\Spec(\Gamma\XX)\\
\epsilon=\epsilon_{\EE}:\EE\lto\Gamma(\Spec\ \EE).\\
\end{array}$$
The counit $\epsilon_{\EE}$ is the equivalence of categories computed in theorem \ref{aff_schemes}.

In order to describe the unit $\eta_{\XX}:\XX\to\Spec(\Gamma\XX)$ we fix affine cover $J=\{\UU_i\simeq\Spec(\FF_i)\}$. Taking global sections induces a family of pretopos functors and, dually, a family of scheme morphisms
$$\xymatrix@R=3ex{
**[r] \Gamma(\XX) \ar[rrr]^{J_i} &&& **[l] \Gamma(\UU_i)\simeq \FF_i\\
**[r] \UU_i\simeq \Spec(\FF_i) \ar[rrr]^-{J_i^*} &&& **[l] \Spec(\Gamma\XX).\\
}$$
Now we define $\eta=\eta_{\XX}$ by setting $\eta(x)=J^*_i(x)$ for $x\in\UU_i$. Similarly, we can define the action of $\eta$ on isomorphisms by setting $\eta(\alpha)=J^*_i(\alpha)$ when $\alpha\in\UU_i$ and extend to all of $\XX$ by composition.

To see that this is well-defined, suppose that we have an isomorphism $\alpha\in\UU_i\cap\UU_j$. The intersection is an open subscheme, so there is an affine neighborhood $\UU_{ij}\simeq\Spec(\FF_{ij})$ such that $\alpha\in\UU_{ij}\subseteq \UU_i\cap\UU_j$. The resulting composites $\UU_{ij}\rightrightarrows \Spec(\Gamma\XX)$ will agree on $\alpha$ (modulo a canonical isomorphism in $\Spec(\Gamma\XX)$) because they arise from isomorphic pretopos functors:
$$\xymatrix@C=2ex{
& \UU_{ij} \ar[rd] \ar[ld] &&&&& \FF_{ij} & \\
\UU_i \ar[rd] \ar[rdd]_{J_i} && \UU_j \ar[ld] \ar[ldd]^{J_j} &&& \FF_i \ar[ru] && \FF_j \ar[lu]\\
&\XX \ar@{-->}[d] &&&&& \Gamma\XX \ar[ru] \ar[lu] &\\
&\Spec(\Gamma\XX) &&&&& \Gamma(\Spec\ \Gamma\XX). \ar[u]_{\textstyle\vertsim}& \\
}$$

By considering a triple intersection we can show that these isomorphisms satisfy cocycle conditions, allowing us to invoke the gluing lemma \ref{gluing_lemma}. This says that $\XX$ is a colimit $\XX\simeq\colim_i\UU_i$. Since the maps $J_i^*$ agree (up to isomorphism) on their overlap, they induce a map $\eta:\XX\to\Spec(\Gamma\XX)$.

We also need a structure map $\eta^\sharp:\eta^*\OO_{\Gamma\XX}\to\OO_{\XX}$, which we construct via descent. Recall that $\EqSh(\XX)$ is equivalent to the category of descent sheaves $\Desc(J)$; via this equivalence, $\OO_{\XX}$ and $\eta^*\OO_{\Gamma\XX}$ correspond to the families $\<\OO_{\UU_i}\>$ and $\<J_i^*\OO_{\Gamma\XX}\>$, respectively. Both of these carry natural descent data and the maps $J_i^\sharp:J_i^*\OO_{\Gamma\XX}\to \OO_{\UU_i}$ commute with these isomorphisms (again because they are defined locally by isomorphic functors). Therefore the family $\<J_i^\sharp\>$ defines a map in $\Desc(J)$ and we let $\eta^\sharp$ equal the corresponding map in $\EqSh(\XX)$.

In order to see that $\eta$ and $\epsilon$ define a (2-)adjunction we must verify the triangle isomorphisms
$$\epsilon_\Gamma\circ\Gamma(\eta)\cong1_\Gamma\ \ \ \ \ \ \Spec(\epsilon)\circ\eta_{\Spec}\cong 1_{\Spec}.$$
Since $\epsilon$ is (pointwise) an equivalence of categories, it is enough to check that $\Gamma(\eta)$ and $\eta_{\Spec(\EE)}$ are also equivalences; this will guarantee the existence of the natural isomorphisms asserted above.

First we show that the the global sections $\Gamma(\eta)=\Gamma(\eta^\sharp)$ are pseudo-inverse to $\epsilon_\Gamma$. Given a global object $A\in\Gamma\XX$, let $B=\epsilon_\Gamma(A)$. This is a section over $\Spec(\Gamma\XX)$ sending a labelled $\Gamma\XX$-model $\mu$ to $L_\mu A$, the germ of $A$ at $\mu$.\footnote{$L_\mu A$ was denoted $\tilde{\top}(A)$ in section \ref{sec_diagrams}.}

For each $i$, this gives us a pair of local sections over $\UU_i$: $A_i=A|_{\UU_i}$ and $B_i=J_i(B)$. These families lift to a pair of objects in $\Desc(J)$ and we will show that these are isomorphic. Fixing a point $\nu\in\UU_i\subseteq\XX$, we calculate
$$\begin{array}{rcl}
(J_i^\sharp B)(\nu) & \cong & J_i^\sharp\big[(\epsilon_\Gamma A)(J_i^*\nu)\big]\\
&\cong & J_i^\sharp(L_{J_i^*\nu} A)\\
&\cong & L_\nu(J_i A)\\
\end{array}$$
Because $J_i$ is defined by global sections inclusion $\UU_i\subseteq\XX$ we have $J_i(A)\cong A_{\UU_i}$ and therefore $A_i\cong B_i$.

Appealing to the compatibility of the maps $J_i$ one easily check that these isomorphism commutes with the descent data yielding an isomorphism $\eta^\sharp(B)\cong A$. Therefore $\Gamma\eta$ is essentially surjective. A similar argument shows that $(\Gamma\eta)(\epsilon_\Gamma f)= f$ for any arrow $f\in\Gamma\XX$, so that $\Gamma\eta$ is full.

To see that $\Gamma\eta$ is faithful, first notice that the maps $\Gamma\XX\to\Gamma\UU_i$ are jointly faithful (because the maps $\UU_i\to\XX$ are jointly surjective). Therefore the induced map $\Spec(\prod_i\Gamma\UU_i)\to\Spec(\Gamma\XX)$ is again essentially surjective and gives rise to another faithful functor
$$\xymatrix{
\Gamma\Spec(\Gamma\XX) \ar@{>->}[rd]_{\Gamma\eta} \ar@{>->}[rr] && **[l] \Gamma\Spec(\prod_i \Gamma\UU_i)\simeq\prod_i\Gamma\UU_i.\\
&\Gamma\XX \ar[ur] &\\
}$$
Given the factorization above, $\Gamma\eta$ must also be faithful and, therefore, an equivalence of categories.

Finally we consider the map $\eta_{\Spec(\EE)}:\Spec(\EE)\to\Spec\ \Gamma(\Spec\ \EE)$. Given an $\EE$-model $\mu$, the underlying model of $\eta_{\Spec(\EE)}(\mu)$ is given by the composite of the localization $L_\mu:\Gamma(\Spec\ \EE)\to\Diag(\mu)$ with the canonical model $\Hom(1,-):\Diag(\mu) \lto \Sets$. Now we compose with the equivalence $\EE\simeq\Gamma(\Spec\ \EE)$ and note that the resulting functor $\EE\to \Sets$ is canonically isomorphic to the underlying model $M_\mu$. This shows that (the underlying groupoid homomorphism) $\eta_{\Spec(\EE)}$ is essentially surjective.

Similarly, any isomorphism $\alpha:\mu\iso \mu'$ induces an isomorphism of diagrams $\Diag(\alpha):\Diag(\mu)\cong\Diag(\mu')$ and this yields a factorization of $\alpha$ through $\Gamma(\Spec\ \EE)$:
$$\xymatrix@R=2ex{
&& \Diag(\mu) \ar[dd]|{\Diag(\alpha)} \ar[dr] & \\
\EE \ar[r]^-\sim \ar@/^10ex/[rrr]^{M_\mu} \ar@/_10ex/[rrr]_{M_{\mu'}} \ar@{}[rr]^{\raisebox{3ex}{$\vertiso$}} \ar@{}[rr]_{\raisebox{-5ex}{$\vertiso$}} \ar@{}[rr]|(.8){\mbox{$\vertiso$}} \ar@{}[rrr]|(.8){\mbox{$\vertiso$}}& \Gamma(\Spec\ \EE) \ar[ur] \ar[dr] && \Sets.\\
&& \Diag(\mu') \ar[ur] &\\
}$$

On the other hand,  for any $\mu\in\Spec(\EE)$ we have $\Diag(\mu)\cong\Diag(\eta(\mu))$, so an isomorphism $\eta(\mu)\cong\eta(\mu')$ factors through $\EE$ in the same way. Thus $\eta_{\Spec(\EE)}$ is fully faithful, and hence an equivalence of continuous groupoids.

At the level of structure sheaves we have a map $\eta^\sharp:\OO_{\Gamma(\Spec\ \EE)}\to \eta_*\OO_{\EE}$. By the previous argument, its global sections $I=\Gamma(\eta^\sharp)$ is an equivalence is an equivalence of categories. On any other basic open set $V_{\varphi(k)}\in\Spec\ \Gamma(\Spec\ \EE)$, the partial sections of $\eta^\sharp$ factor as sequence of equivalences
$$\OO_{\Gamma(\Spec\ \EE)}(V_{\varphi(k)})\simeq \Gamma(\Spec\ \EE)/\varphi\simeq \EE\!/I\varphi\simeq \OO_{\EE}(V_{I\varphi(k)})=\eta_*\OO_{\EE}(V_{\varphi(k)}).$$
Therefore $\eta_{\Spec(\EE)}$ is an equivalence of schemes, establishing the second triangle isomorphism and the adjunction
$$\Ptop(\EE,\Gamma\XX)\simeq\LSch(\XX,\Spec(\EE)).$$

\end{proof}

\begin{cor}
A logical scheme $\XX$ is affine if and only if the unit $\eta:\XX\to\Spec(\Gamma\XX)$ is an equivalence of schemes.
\end{cor}
\begin{proof}
We have just verified a triangle isomorphism $\Spec(\epsilon_{\EE})\circ\eta_{\Spec(\EE)}\cong 1_{\Spec(\EE)}$. If $\XX\simeq\Spec(\EE)$ is affine then $\eta_{\Spec(\EE)}$ is the asserted equivalence. If $\XX$ is not affine then for any $\XX\not\simeq\Spec(\EE)$ for any $\EE$, so $\eta:\XX\to\Spec(\Gamma\XX)$ cannot be an equivalence.
\end{proof}

Now we will show that the category of logical schemes is closed under finite 2-limits (or weighted limits). We will not recall the general definition, instead referring the reader to \cite{lack}. Instead, we introduce a few examples which will suffice for our purposes.

For any ordinary limit (e.g., products, equalizers, etc.) defined in ordinary categories there is an associated pseudo-limit defined for 2-categories. For example, if $\KK$ is a 2-category then a (pseudo-)terminal object in $\KK$ is an object $1_{\KK}$ such that for any other $K\in\KK$ there is an essentially unique morphism $K\to 1_{\KK}$. More precisely, such a map always exists and any two such maps are associated by a unique 2-cell $\xymatrix{K \rtwocell{\omit \vertiso} & 1_{\KK}}$. Such a limit is unique up to equivalence in $\KK$.

Similarly, the (pseudo-)pullback
$$\xymatrix{
\raisebox{-3ex}{$K_1\overtimes{L} K_2$} \ar[r]^-{p_1} \ar[d]_{p_2} \drtwocell<\omit>{\omit \textstyle^\gamma\rotatebox{45}{$\cong$}} & K_1 \ar[d]^{k_1}\\
K_2 \ar[r]_{k_2} & L\\
}$$
is defined by the following universal property: for any maps $z_i: Z\to K_i$ with a natural isomorphism (2-cell) $\theta:k_1\circ z_1\cong k_2\circ z_2$ there is an essentially unique map $\overline{\theta}:Z\to K_1\overtimes{L} K_2$ such that $z_i\cong p_1\circ\overline{\theta}$ and $\theta=\gamma\cdot\overline{\theta}$. Henceforth we will drop the ``pseudo-'' prefix, and simply refer to such objects as limits in $\KK$.

Additionally, there are 2-limits which have no analogue in ordinary categories. Remember that $I$ is the ``walking arrow'' category with two objects and one non-identity arrow. A \emph{power} of an object $K$ by $I$ is an object $K^{\2}$ with a universal (oriented) 2-cell $\xymatrix{ K^{\2} \rtwocell & K}$. In $\Cat$ (or $\Ptop$) the power of $\EE$ by $I$ is the arrow category $\EE^{\2}$; this has two projections ($\dom$ and $\cod$) and any natural transformation $\xymatrix{\FF \rtwocell{\tau} & \EE}$ induces a map $\FF\to \EE^{\2}$ sending $F\mapsto \tau_F\in\Ar(\EE)$.

In ordinary categories, terminal objects and pullbacks suffice to construct all finite limits. Similarly, a theorem of Street \cite{street_limits} guarantees that these three constructions (terminal objects, pullbacks and powers of $I$) suffice to construct all finite limits. All of the above can be reversed to give definitions of 2-colimits, and the same considerations apply.

\begin{lemma}\label{2colimits}
The 2-category of pretoposes is closed under finite 2-colimits.
\end{lemma}
\begin{proof}
In fact, $\Ptop$ is closed under \emph{all} small 2-colimits (essentially because $\Ptop$ is monadic over $\Cat$), but we will not use that fact here. By the above, it is enough to check that $\Ptop$ has an initial object, pushouts and copowers by $I$.

The category of finite sets $\Sets_f$ is the initial object in $\Ptop$. Since a pretopos functor must preserve the terminal object and finite coproducts, any map $\Sets_f\to\EE$ is canonically determined (up to unique isomorphism) by $n\mapsto \overbrace{1_{\EE}+\ldots+1_{\EE}}^{n\rm{ times}}$.

Next consider the pushout of two pretopos functors $I:\EE\to\FF$ and $J:\EE\to\GG$. According to its universal property, a model of the pushout (i.e. a functor $\FF\overplus{\EE}\GG\to\SS$) is a pair of models $M\models\FF$ and $N\models\GG$ together with an isomorphism between their reducts $I^*M\cong J^*N$. This class is axiomatized by the following theory:
$$\begin{array}{rcl}
\LL_+&=&\LL_{\FF}+\LL_{\GG}+\{i_E: IE\to JE,\ j_E:JE\to IE\ |\ E\in\EE\}\\
\bTT_+&=& \bTT_{\FF}+\bTT_{\GG}+\{J(f)\circ i_D=i_{E}\circ I(f)\ |\ f:D\to E)\},\\
&&\hspace{1.4cm}+\{I(f)\circ j_D=j_{E}\circ J(f)\ |\ f:D\to E\},\\
&&\hspace{1.4cm}+\{ j_E\circ i_E=1_{IE},\ i_E\circ j_E=1_{JE}\ |\ E\in\EE\}\\
\end{array}$$
Therefore the classifying pretopos of $\bTT_+$ is a pushout in $\Ptop$.

The copower of $\EE$ by $I$, which we denote $\HH=\EE\To\EE$, is similarly axiomatizable. A model of the copower consists of a pair of $\EE$-models together with a homomorphism between them. To axiomatize this class, first duplicate each object (or arrow or equation between arrows) of $\EE$; every symbol $E\in\EE$ has two copies $E_0$ and $E_1$ in the copower. We also add in components of a transformation $h_E:E_0\to E_1$ along with equational axioms expressing its naturality. This leaves us with the following theory:
$$\begin{array}{rcl}
\LL_\To&=&(\LL_{\EE})_0+(\LL_{\EE})_1+\{h_E: E_0\to E_1\ |\ E\in\EE\}\\
\bTT_\To&=& (\bTT_{\EE})_0+(\bTT_{\EE})_1+\{h_E\circ f_0=f_1\circ h_D\ |\ f:D\to E\}.\\
\end{array}$$
According to its universal property, the classifying pretopos of $\bTT_{\To}$ must be equivalent to the copower $\EE\To\EE$.

More categorically, we can regard $\bTT_\To$ as a presentation of the product category $\bHH=\EE\times I$. This has finite limits, computed pointwise, and it comes equipped with two ``vertical'' functors $m_0,m_1:\EE\to\bHH$. These send $E$ to $\<E,0\>$ or $\<E,1\>$ and they induce a coherent topology on $\bHH$: the basic covers are vertical families which are covers in $\EE$.

This makes sense: the ``horizontal'' maps $h_E=\<1_E,I\>$ are not covers because homomorphisms are not generally surjective. However, the composition in $\bHH$ is computed pointwise, and this ensures that the horizontal maps satisfy a naturality condition. For each $\tau:D\to E$
$$\begin{array}{rcccl}
h_E\circ m_0(\tau)&=&\<1_E,I\>\circ\<\tau,\rm{id}_0\>\\
&=&\<\tau,I\>\\
&=&\<\tau,\id_1\>\circ\<1_D,I\> &=& \tau_1\circ h_D\\
\end{array}$$

This shows that $\bHH$ is a copower of $\EE$ in $\Coh$, the class of coherent categories. Because pretopos completion is a left adjoint, it follows immediately that the pretopos completion of $\bHH$ is the copower of $\EE$ by $I$ in $\Ptop$
$$\begin{array}{c}
\xymatrix{\EE \rtwocell^{M_1}_{M_2}{h} & \SS}\\\hline
\xymatrix{\bHH \ar[r]^{N} & \SS}\\\hline
\xymatrix{\HH=\Ptop(\bHH) \ar[r] & \SS}.\\
\end{array}$$

\end{proof}

\begin{thm}\label{scheme_limits}
The 2-category of logical schemes is closed under finite 2-limits, which are computed from colimits in the 2-category of pretoposes.
\end{thm}
\begin{proof}
As in the algebraic case, we exploit the fact that an adjunction on the right must send colimits to limits, so that $\lim_i \Spec(\EE_i)\simeq\Spec(\colim_i\EE_i)$:
$$\begin{array}{rcll}
\LSch\big(\ZZ,\lim_i(\Spec\ \EE_i)\big)&\cong& \lim_i \LSch(\ZZ,\Spec\ \EE_i) & \rm{(limit in }\Cat\rm{)}\\
&\cong & \lim_i \Ptop(\EE_i,\Gamma\ZZ)\\
&\cong & \Ptop(\colim_i\EE_i,\Gamma\ZZ)\\
&\cong & \LSch\big(\ZZ,\Spec(\colim_i\EE_i)\big).\\
\end{array}$$
This shows that finite 2-limits of affine schemes exist, and they are the affine schemes associated with colimits in $\Ptop$. This already shows that schemes have a terminal object: $\Spec(\Sets_f)$.

Now we construct the pullback of two scheme morphisms $\YY\to \XX$ and $\ZZ\to\YY$; the argument is essentially the same as that for the algebraic case. First, note that if $\YY\overtimes{\XX}\ZZ$ exists and $\UU\subseteq \YY$ then the pullback $\UU\overtimes{\XX}\ZZ$ is given by the inverse image $p_{\YY}^{-1}(\UU)$. In the diagram below, $\WW$ is supported over $\UU$, so the induced map $\WW\to\YY\overtimes{\XX}\ZZ$ must factor through $p_{\YY}^{-1}(\UU)$:
$$\xymatrix@C=3ex{
\WW \ar[rrrrd] \ar[rdd] \ar@{-->}[rrd] \ar@{-->}[rd]\\
& p_{\YY}^{-1}(\UU) \ar@{}[r]|{\textstyle\subseteq} \ar[d] & \raisebox{-3ex}{$\YY\overtimes{\XX}\ZZ$} \pbcorner \ar[rr] \ar[d] && \ZZ \ar[d]\\
& \UU \ar@{}[r]|{\textstyle\subseteq} & \YY \ar[rr] && \XX\\
}$$

Now suppose that $\YY\simeq\bigoplus_i \YY_i$ is an open cover and that each pullback $\YY_i\overtimes{\XX}\ZZ$ exists. Let $\VV_{ij}=p_{\YY_i}^{-1}(\YY_i\cap\YY_j)$. By the previous observation, $\VV_{ij}$ is a pullback and therefore uniqueness implies that there is a canonical equivalence $\VV_{ij}\simeq\VV_{ji}$. This means that we can glue the schemes $\YY_i\overtimes{\XX}\ZZ$ along the $\VV_{ij}$ in order to define a scheme $\bigoplus_i \left(\YY_i\overtimes{\XX}\ZZ\right)$.

Given any maps $F:\WW\to\YY$ and $G:\WW\to\ZZ$ which commute with the projections to $\XX$, we set $\WW_i=F^{-1}(\YY_i)$. This defines a family of maps $\WW_i\to \YY_i\overtimes{\XX}\ZZ$ and one easily checks that these agree on their overlaps. Gluing gives the existence of a map $\WW\to\bigoplus_i \left(\UU_i\overtimes{\XX}\ZZ\right)$ and the essential uniqueness of this map can be checked locally. This shows that $\YY\overtimes{\XX}\ZZ=\bigoplus_i \left(\UU_i\overtimes{\XX}\ZZ\right)$ is a pullback in schemes.

By the argument at the beginning of the theorem we know that affine schemes have pullbacks:
$$\Spec(\FF)\overtimes{\Spec(\EE)}\Spec(\GG)\simeq\Spec(\FF\overplus{\EE}\GG).$$
Now suppose that $\XX=\Spec(\EE)$ is affine and choose open covers $\YY\simeq\bigoplus_i\YY_i$ and $\ZZ\simeq\bigoplus_j\ZZ_j$. By the previous argument, applied twice, we have
$$\YY\overtimes{\XX}\ZZ\simeq \bigoplus_i \left(\YY_i\overtimes{\XX}\ZZ\right) \simeq \bigoplus_{i,j} \left(\YY_i\overtimes{\XX}\ZZ_j\right).$$

Finally, suppose that we have an open cover $\XX=\bigoplus_i\XX_i$; let $\YY_i$ and $\ZZ_i$ denote the inverse images in $\XX_i$ in $\YY$ and $\ZZ$. We have just shown that the pullback $\YY_i\overtimes{\XX_i}\ZZ_i$ exists; now we observe that $\YY_i\overtimes{\XX_i}\ZZ_i\simeq\YY_i\overtimes{\XX}\ZZ$:
$$\xymatrix{
\raisebox{-3ex}{$\YY_i\overtimes{\XX}\ZZ$} \pbcorner \ar[r] \ar[d] & \ZZ_i \pbcorner \ar[d] \ar[r] & \ZZ \ar[d] \\
\YY_i \ar[r] & \XX_i \ar[r] & \XX\\
}$$
By the same argument as above, it follows that $\LSch$ is closed under pullbacks:
$$\YY\overtimes{\XX}\ZZ\simeq\bigoplus_i \left(\YY_i \overtimes{\XX} \ZZ\right) \simeq\bigoplus_i \left(\YY_i\overtimes{\XX_i}\ZZ_i\right).$$

A similar argument shows that $\LSch$ has powers of $I$, which we denote $\XX^{\2}$. In $\Ptop$ we can push out a natural transformation along a slice $\EE\to\EE\!/A$. Locally, this defines restriction of a 2-cell $\xymatrix{\ZZ \rtwocell^F_G{\tau} & \XX}$ to an open subscheme $\UU\subseteq\XX$:
$$\xymatrix{
\EE \rrtwocell^{\2}_J{\theta} \ar[d] && \FF \ar[d] & \XX \rrtwocell<\omit>&& \ZZ \lltwocell^G_F{\omit\ \ \ \ \ \tau} \\
\EE\!/A \rrlowertwocell<\omit>{<3>\ \ \ \theta/A} \ar[rr] \ar[rd] && \FF/IA & \UU \ar[u] \rrlowertwocell<\omit>{<3>\ \ \ \tau|_{\UU}} && F^*\UU \ar[ll] \ar[dl] \ar[u]\\
& \FF/JA \ar[ur]_{\theta_A^*} &&& G^*\UU \ar[ul]\\
}$$

This is closely related to the fact that the copower of slices is a slice of the copower: $\big(\EE\!/A\To\EE\!/A\big)\simeq (\EE\To\EE)/A_1$. This is because an $\EE\!/A$-morphism $h:\<M_1,a_1\>\to\<M_2,a_2\>$ is, by definition, an $\EE$-morphism $M_1\to M_2$ such that $h(a_1)=a_2$.

Now suppose that $\XX^{\2}$ exists and that $\UU\subseteq\XX$ is an open subscheme. Then $\UU^{\2}$ exists, and it is given by the inverse image along the domain map $\xymatrix{\XX^{\2} \rtwocell^\dom_\cod & \XX}$. To see that this gives the correct universal property, suppose that we have a map $T:\ZZ\to \UU^{\2}\subseteq\XX^{\2}$. By the universal property of $\XX^{\2}$ this is equivalent to a 2-cell $\xymatrix{\ZZ \rtwocell^F_G{\tau} & \XX}$. Because $T$ factors through $\UU^{\2}=\dom^{-1}(\UU)$, we must have $F^*\UU=\dom(T)^*\UU \cong \ZZ$. It follows that the restriction of $\tau$ to $\UU$ is a 2-cell in $\LSch(\ZZ,\UU)$:
$$\begin{array}{c}
\xymatrix{\ZZ \ar[rr] && \dom^{-1}(\UU)=\UU^{\2}}\\\hline
\xymatrix{**[r] \ZZ\cong F^*\UU \rrlowertwocell<\omit>{<3>\ \ \ \tau|_{\UU}} \ar[rr] \ar[rd] && \UU.\\
& G^*\UU \ar[ur]\\}\\
\end{array}$$

On the other hand, if $\XX=\bigoplus_i \XX_i$ and $\XX_i^{\2}$ exists then so does $\XX^{\2}$, and it is given by $\bigoplus_i \XX_i^{\2}$. Here $\XX_i^{\2}$ is a matching family because, by the previous argument, $(\XX_i\cap\XX_j)^{\2}$ exists and is a subscheme of $\XX_i^{\2}$. One can check that if a family of scheme morphisms $\YY_i\to\XX_i^{\2}$ matches on $\YY=\bigoplus_i\YY_1$, then so do the induced 2-cells $\xymatrix{\YY_i\rtwocell & \XX_i}$. Thus any map $\YY\to\XX^{\2}$ corresponds to a unique 2-cell $\xymatrix{\YY \rtwocell & \XX}$, so $\XX^{\2}$ is a power in $\LSch$.

If $\XX=\Spec(\EE)$ is affine, then its power $\XX^{\2}$ is also affine, given by $\Spec(\EE\To\EE)$. If $\VV=\Spec(\FF)$ is affine then we have the sequence of equivalences below; the general case follows by gluing.
$$\begin{array}{c}
\xymatrix@=5ex{**[r] \VV \rrtwocell{\ \ \tau'} && **[l]\UU}\\\hline
\xymatrix@=5ex{**[r]\EE \rrtwocell{\ \ \ \Gamma\tau'} && **[l]\FF}\\\hline
\widetilde{\Gamma\tau'}:\xymatrix@=5ex{**[r]\EE\To\EE \ar[rr] && **[l]\FF}\\\hline
\widetilde{\tau'}:\xymatrix@=5ex{**[r]\VV \ar[rr] && **[l]\Spec(\EE\To\EE)}\\
\end{array}$$
Since any scheme has an affine cover $\XX\simeq\bigoplus_i \Spec(\EE_i)$, it also has a power $\XX^{\2}\simeq\bigoplus_i\Spec(\EE_i\To\EE_i)$.
\end{proof}

%% file: Ch4.tex
\renewcommand{\Mod}{\mathbf{Mod}}

\chapter{Applications}

The affine scheme $\Spec(\bTT)$ associated with a logical theory incorporates both semantic and syntactic components of the theory. As such, it is a nexus to study the connections between different branches of logic and other areas of mathematics. In this chapter we will describe a few of these connections.

The first section discusses a connection between schemes and topos theory. The structure sheaf $\OO_{\bTT}$ is an internal pretopos on the category of equivariant sheaves $\EqSh(\MM_{\bTT})$. The machinery of sheaves and sites can be relativized to this context, so that $\OO_{\bTT}$ is a site for the (internal) coherent topology. We show that this topos classifies $\bTT$-model homomorphisms. It follows that this can also be regarded as the (topos) exponential of $\Sh(\MM_{\bTT})$ by the Sierpinski topos $\Sets^{\2}$.

The next section describes another view of the structure sheaf, as a type-theoretic universe. From the results of chapters 1 and 3 we know that an equivariant morphism of sheaves $e:\ext{\varphi}\to\OO_{\bTT}$ can be classified (up to isomorphism) by an object $E\in\EE\!/\varphi$. We will define an auxilliary sheaf $\El(\OO_{\bTT})\to\OO_{\bTT}$ which allows us to recover $E$ from $e$ via pullback:
$$\xymatrix{
\ext{E} \pbcorner \ar[d] \ar[rr] && \El(\OO_{\bTT}) \ar[d]\\
\ext{\varphi} \ar[rr] && \OO_{\bTT}.\\
}$$
This allows us to think of $\OO_{\bTT}$ as a universe of \emph{definably} or \emph{representably} small sets. Formally, we show that $\OO_{\bTT}$ is a coherent universe, a pretopos relativization of Streicher's notion of a universe in a topos \cite{streicher}.

In the third section we demonstrate a tight connection between our logical schemes and a recently defined ``isotropy group'' \cite{isotropy} which is present in any topos. This allows us to interpret the isotropy group as a logical construction. Using this, we compute the stalk of the isotropy group at a model $M$ and show that its elements can be regarded as parameter-definable automorphisms of $M$.

In the last section we discuss Makkai \& Reyes' conceptual completeness theorem \cite{MakkaiReyes} and reframe it as a theorem about schemes. The original theorem says that if an interpretation $I:\bTT\to\bTT'$ induces an equivalence $I^*:\bf{Mod}(\bTT')\iso\bf{Mod}(\bTT)$ under reducts, then $I$ itself was already an equivalence (at the level of syntactic pretoposes). The corresponding statement for schemes is trivial: if $\<I^\flat,I_\sharp\>:\Spec(\bTT')\to\Spec(\bTT)$ is an equivalence of schemes, then the global sections $\Gamma I_\sharp$ defines an equivalence $\bTT\simeq\bTT'$. However, we can unwind the Makkai \& Reyes proof to provide insight into the spectral groupoid $\MM_{\bTT}$. The resulting ``Galois theory'' relates logical properties of $I$ to a mixture of topological and algebraic (i.e., groupoid-theoretic) properties of $I^\flat$.

\section[$\OO_{\bTT}$ as a site]{Structure sheaf as site}

The formal definitions of sites and sheaves can be described in geometric logic, and can therefore be interpreted internally in any topos (see, e.g., \cite{elephant}, C2.4). We have already noted that the structure sheaf $\OO_{\EE}$ is a pretopos (and hence a site) internally in $\Sh(\EE)\simeq\EqSh(\MM_{\EE})$. In this section we will discuss the topos of internal sheaves $\Sh_{\EE}(\OO_{\EE})$. In particular, we will show that this category of internal sheaves is equivalent to the topos exponential of $\Sh(\EE)$ by the Sierpinski topos $\Sets^{\2}$.

Before proceeding we introduce the necessary terminology and notation. In this section we let $\SS=\Sets$, although the same arguments can be relativized to any base topos. We let $I$ denote the poset $\{0\leq 1\}$; we will variously regard $I$ as a category (with one non-identity morphism), a finite-limit category (where $1$ is terminal and $0$ is a subobject of $1$) and a regular category (where the non-identity morphism is not a cover. We will use standard notation $\bf{C}\times \bf{D}$ or $\bf{D}^\bf{C}$ for products and exponentials (i.e., functor categories) in $\Cat$.

In particular, the \emph{Sierpinski topos} is the functor category $\SS^{\2}$, also known as the arrow category of $\SS$. Its objects are functions in $\SS$ and its morphisms are commutative squares. The Sierpinski topos classifies subobjects $p\leq 1$. Indeed, a geometric morphism $f:\ZZ\to\SS^{\2}$ is equivalent to a left-exact functor $f_0:\2\to\ZZ$. Since $f_0$ must preserve the terminal object 1, it is completely determined by the image $f_0(0)$. Similarly, $f_0$ must preserve subobjects, so the image $f_0(0)=p$ is subterminal in $\ZZ$, and $\2$ has no other finite limits.

Let $\Top$ denote the category of (Grothendieck) toposes and geometric morphisms. We will be concerned with products and exponentials in $\Top$ (which will exist for those cases we are concerned with, cf. \cite{elephant} C4); in order to distinguish these from product and functor categories we use the following notation:
$$\begin{array}{c}
\XX\otimes\YY\to \ZZ\\\hline
\XX\to\exp(\YY,\ZZ).
\end{array}$$

As a final piece of notation, recall section 3.5 that $\EE\To\EE$ denotes the copower of $\EE$ in $\Ptop$, which is to say that a model of $\EE\To\EE$ is the same as a homomorphism of $\EE$-models:
$$\begin{array}{c}
(\EE\To\EE)\lto\SS\\\hline
\xymatrix{\EE \rtwocell & \SS}.
\end{array}$$
From lemma \ref{2colimits} we have a syntactic description of $\HH=(\EE\To\EE)$ as the pretopos completion of the product category $\EE\times \2$. In particular, $\HH$ contains two objects $A_0$ and $A_1$ for each object $A\in\EE$, and these are connected by a distinguished morphism $k_A:A_0\to A_1$

\begin{thm}
The following toposes are equivalent:
\begin{enumerate}
\item The category of internal (coherent) sheaves on $\OO_{\EE}$.
\item The category of (coherent) sheaves on $\EE\To\EE$.
\item The exponential (in $\Top$) of $\Sh(\EE)$ by the Sierpinski topos.
\end{enumerate}
$$\Sh_{\EE}(\OO_{\EE})\simeq\Sh(\EE\To\EE)\simeq\exp(\SS^{\2},\Sh(\EE))$$
\end{thm}

\begin{lemma}
The product in $\Top$ of a topos $\FF$ with the Sierpinski topos is given by the functor category $\FF^{\2}$:
$$\FF\otimes\SS^{\2}\simeq\FF^{\2}.$$
\end{lemma}

\begin{proof}[Sketch]
A more general proof, following from Diaconescu's theorem, can be found in Johnstone \cite{elephant}, Cor. 3.2.12.

We will show that both toposes have the same classifying property. First of all, notice that $\FF^{\2}$ has projections to both $\FF$ and $\SS^{\2}$:
$$\begin{array}{c}
\gamma_{\FF}:\xymatrix{\FF^{\2} \rtwocell^{\id}_{\dom}{`\bot} & \FF},\\
\Gamma^{\2}:\xymatrix{\FF^{\2} \rtwocell^{\Delta^{\2}}_{\Gamma^{\2}}{`\bot} & \SS^{\2}},\\
\end{array}$$
Given $P:\ZZ\to\FF^{\2}$ we can recover $f:\ZZ\to\FF$ and $p\leq 1_{\ZZ}$ by composing with these two projections (and applying the classifying property of $\SS^{\2}$).

On the other hand, given $p$ and $f$ we wish to construct an extension $P_f$ as below
$$\xymatrix{
\SS^{\2} \ar[r]^{\Delta^{\2}} \ar@{-->}[rd]|{P} & \FF^{\2} \ar@{-->}[d]^{P_f}\\
\2 \ar[r]_{p\leq 1} \ar[u]^y & \ZZ.\\
}
$$
In order to construct $P_f$ it helps to first understand $P$.

The representable functors $y0$ and $y1$ correspond to the unique functions $\emptyset\to 1$ and $1\to 1$, respectively; $P$ sends these to $p$ and to $1_{\ZZ}$. The unique function $!_2:2\to 1$ is the coequalizer of two projections $y0\rightrightarrows y1$; it follows that $P(!_2)$ is given by two copies of $1_{\ZZ}$ glued together along $p$. Similarly, $P(!_n)$ consists of $n$ copies of $1_{\ZZ}$ glued at $p$. $P$ sends any other function $\pi:F\to A$ to $A$-many disjoint pieces, each consisting of $F(a)$-many copies of $1_{\ZZ}$ glued along $p$.

The construction of $P_f$ is formally similar. The inverse image of a map $\pi:F\to A$ is covered by $p\times f^*A$ and $f^*F$, which we think of as $A$-many copies of $p$ together with $F$-many copies of $1_{\ZZ}$. The image of $\pi$ is then given as a pushout
$$\xymatrix{
p\times f^*F \ar[r]^{p\times f^*\pi} \ar[d]_{p_2} & p\times f^*A \ar[d] \\
f^*F \ar[r] & P_f^*(\pi).
}$$
When $F$ and $A$ are discrete (i.e., in the image of $\Delta$) this is the same as our description of $P$, so the extension diagram above commutes. One should also check that $P_f^*$ does, in fact, define the inverse image of a geometric morphism. This can be established by constructing a hom-tensor adjunction as in \cite{SGL}, VII.7-9. 
\end{proof}

\begin{lemma}\label{left_equiv}
The exponential topos $\exp(\SS^{\2},\Sh(\EE))$ classifies homomorphisms of $\EE$-models and is therefore equivalent to the topos of sheaves on the copower $\EE\To\EE$.
\end{lemma}
\begin{proof}
According to the defining adjunction of the exponential together with the last lemma we have equivalences, natural in $\ZZ$,
$$\begin{array}{c}
\ZZ\lto\exp(\SS^{\2},\Sh(\EE))\\\hline
\ZZ\otimes\SS^{\2}\lto \Sh(\EE)\\\hline
\ZZ^{\2}\lto\Sh(\EE).\\
\end{array}$$

The latter geometric morphism is completely determined by a pretopos functor $\EE\to\ZZ^{\2}$. But $\ZZ^{\2}$ is a power object in $\Ptop$ and, by the universal properties of powers and copowers we have another sequence of equivalences
$$\begin{array}{c}
\EE\lto \ZZ^{\2}\\\hline
\xymatrix{\EE \rtwocell & \ZZ}\\\hline
(\EE\To\EE)\lto \ZZ.
\end{array}$$
The last pretopos functor induces a unique geometric morphism ${\ZZ\to\Sh(\EE\To\EE)}$. By the uniqueness of classifying toposes, we conclude that 
$$\exp(\SS^{\2},\Sh(\EE))\simeq\Sh(\EE\To\EE).$$
\end{proof}

\begin{prop}\label{right_equiv}
The topos of internal sheaves on $\OO_{\EE}$ is equivalent to the category of sheaves on $\EE\To\EE$:
$$\Sh_{\MM_{\EE}}(\OO_{\EE})\simeq\Sh(\EE\To\EE).$$
\end{prop}

\begin{proof}

Let $\GG=\Sh_{\MM_{\EE}}(\OO_{\EE})$. Recall the construction of $\OO_{\EE}$: we first sheafified the codomain fibration $\EE^{\2}\to\EE$ and then transported it across the equivalence $\Sh(\EE)\simeq\EqSh(\MM_{\EE})$. Because equivalent sites produce equivalent categories of internal sheaves, the topos $\GG$ is insensitive to both operations; therefore we need not distinguish between $\OO_{\EE}$ and $\EE^{\2}$.

Since $\Sh(\EE\To\EE)$ is a coherent topos, Delingne's theorem ensures that it has enough points. Moreover, an internal version of Deligne's theorem ensures that $\GG$ also has enough points. Since $\OO_{\EE}$ is a coherent site internally it has enough ``internal points'' $\Sh(\EE)\to\HH$. $\Sh(\EE)$ is (externally) coherent, so it has enough ordinary points, and we may compose these facts.

Given distinct, parallel morphisms $f\not=g$ in $\HH$ first find $H:\Sh(\EE)\to\HH$ such that $H^*f\not=H^*g$. Then find $M:\SS\to\Sh(\EE)$ such that $M^*(H^*f)\not= M^*(H^*g)$. The resulting composite $H\circ M$ is a point $\SS\to\HH$ which separates $f$ and $g$. Therefore $\GG$ also has enough points and we can verify that $\GG\simeq\Sh(\EE\To\EE)$by checking an equivalence between their categories of points:
$$\Top(\SS,\Sh(\EE\To\EE))\simeq\Top(\SS,\GG).$$

We begin by calculating the points $g:\SS\to\GG$. Such a point can be represented as a pair $g=\<M,f\>$ where $M$ is a point of $\Sh(\EE)$ and $f$ is a point in the (external) sheaf topos $\Sh(M^*\OO_{\EE})$.

Given $g$, we can recover $M$ by composing with the internal global sections morphism $\GG\to\Sh(\EE)$. From the results of section 2.4 we know that the stalk of $\OO_{\EE}$ at $M$ is the diagram of $M$: $M^*\OO_{\EE}\simeq \Diag(M)$. Since a model of the diagram of $M$ consists of another model $N$ together with a homomorphism $f:M\to N$, the points of $\GG$ are the same as those of $\Sh(\EE\To\EE)$.

We must also check the morphisms between points of $\GG$ and show that these agree with the morphisms of points in $\Sh(\EE\To\EE)$. First suppose that we have a morphism of points in $\Sh(\EE\To\EE)$; this amounts to a natural transformation $\xymatrix{**[l](\EE\To\EE) \rtwocell^h_{h'}{\eta} & \SS}$; here $h$ and $h'$ correspond to $\EE$-model homomorphisms $M\to N$ and $M'\to N'$, respectively.

If we consider the components of $\eta$ at $A_0$ (resp. $A_1$) for each object $A\in\EE$, these define a model homomorphism $\eta_0:M\to M'$ (resp. $\eta_1:N\to N'$). Naturality of $\eta$ along the distinguished morphisms $k_A:A_0\to A_1$ shows that these homomorphisms commute with $h$ and $h'$, so that a morphism of points is $\Sh(\EE\To\EE)$ is equivalent to a commutative square in $\bf{Mod}(\EE)$:

$$\xymatrix{
h(A_0) \ar[r]^{\eta_{0A}} \ar[d]_{h(k_A)} & h'(A_0) \ar[d]^{h'(k_A)} && M \ar[r]^{\eta_0} \ar[d]_{h} & M' \ar[d]^{h'}\\
h(A_1) \ar[r]_{\eta_{1A}} & h'(A_1) && N \ar[r]_{\eta_1} & N'.
}$$

Now consider a transformation $\xymatrix{\SS \rtwocell^g_{g'}{\gamma} & \GG}$. As we did for the points, we can characterize such a transformation in terms a component at $\Sh(\EE)$ together with a component at (the inverse image of) the structure sheaf.

As above, the first of these is defined by whiskering with the internal global sections geometric morphism. This yields a homomorphism $\gamma_0$ between the first components of $g=\<M,f\>$ and $g'=\<M',f'\>$:
$$\xymatrix{\SS \rrtwocell^g_{g'}{\gamma} \ar@/^1cm/[rrr]^M \ar@/_1cm/[rrr]_{M'} && \GG \ar[r] & \Sh(\EE)}$$

This homomorphism induces a pretopos functor between the diagram of $M$ and the diagram of $M'$ and which, by abuse of notation, we also call $\gamma_0$. This functor mediates a natural transformation $\overline{\gamma}$ between the second components of $g$ and $g'$:
$$\xymatrix{
\Diag(M) \ar[rr]^{\Diag(\gamma_0)} \ar[rd]_f & \ar@{}[d]|{\To} \ar@{}[d]|(.35){\overline{\gamma}} & \Diag(M') \ar[ld]^{f'}\\
& \SS &.\\
}$$
The original transformation $\gamma$ is completely determined by the pair $\<\gamma_0,\overline{\gamma}\>$.
 
Recall that an object of $\Diag(M)$ has the form $\varphi(x,b)$, where $\varphi\mono A\times B$ is a formula in $\EE$ and $b\in B^M$ is a parameter in $M$. Since $\gamma_0$ (regarded as a pretopos functor) acts by sending $\varphi(x,b)\mapsto\varphi(x,\gamma_0(b))$, it has no effect on unparameterized formulas. Therefore it commutes with the canonical inclusions $\EE\to\Diag(M)$ and $\EE\to\Diag(M')$. If we let $N$ and $N'$ denote the (codomain) models associated with $f$ and $f'$, then composing with $\overline{\gamma}$ yields a second model homomorphism $\gamma_1:N\to N'$
$$\xymatrix@R=3ex{
& \Diag(M) \ar[dd] \ar[rd]^f &\\
\EE \ar[rd] \ar[ru] \ar@/^1.5cm/[rr]^N \ar@/_1.5cm/[rr]_{N'} & \ar@{}[r]|{\overline{\gamma}} \ar@{}[r]|(.35){\Downarrow}& \SS\\
& \Diag(M') \ar[ru]_{f'} &\\
}$$

The naturality of $\overline{\gamma}$ ensures that this homomorphism defines a commutative square. To see this, consider the component of $\overline{\gamma}$ at the singleton formula $x=a$ (for some element $a\in A^M$). This has an obvious inclusion $(x=a)\mono A$ and, by naturality, this must commute with $\overline{\gamma}_A=\gamma_{1A}$:
$$\xymatrix{
\Diag(M): \ar[d]_{\gamma_0} & x=a \ar@{|->}[d] \ar@{|->}[r] && (x=f(a))^N \ar[d]_{\overline{\gamma}_{x=a}} \ar@{>->}[r] & A^N \ar[d]^{\gamma_{1A}=\overline{\gamma}_A} \\
\Diag(M'): & x=\gamma_0(a) \ar@{|->}[r] && (x=f'(\gamma_0(a)))^{N'} \ar@{>->}[r] & A^{N'}\\
}$$
This shows that the component of $\gamma_1$ at $A$ applied to the unique element ${x=f(a)}$ in $N$ is equal to the unique element $x=f'(\gamma_0(a))$ in $N'$; in other words, $\gamma_1\circ f=f'\circ \gamma_0$.

Therefore $\Top(\SS,\GG)\simeq\Top(\SS,\Sh(\EE\To\EE)$, completing the proof.
\end{proof}

Combining lemmas \ref{left_equiv} and \ref{right_equiv} we have proved the following theorem:
\begin{thm*}
The following toposes are equivalent:
\begin{enumerate}
\item The category of internal sheaves on $\OO_{\EE}$.
\item The category of sheaves on $\EE\To\EE$.
\item The exponential (in $\Top$) of $\Sh(\EE)$ by the Sierpinski topos.
\end{enumerate}
$$\Sh_{\EE}(\OO_{\EE})\simeq\Sh(\EE\To\EE)\simeq\exp(\SS^{\2},\Sh(\EE))$$
\end{thm*}

\section{Structure sheaf as universe}

In this section we investigate a connection between logical schemes and type theory, specifically the notion of a type-theoretic universe. We will show that the structure sheaf $\OO_{\EE}$ can be regarded as a universe of ``representably small'' morphisms in the topos of equivariant sheaves $\EqSh(\MM_{\EE})$. In particular, this will show that we can conservatively extend any coherent theory to include a universe of small sets.

The category theorist tends to think of arrows, rather than objects, as small; a universe in $\EE$ is a subclass of arrows $\UU\subseteq \rm{Ar}(\EE)$ satisfying some natural closure and generation principles. Intuitively, a map is small when each of its fibers is; consequently, an object is small just in case the terminal projection $E\to 1$ belongs to $\UU$. From this characterization we can immediately recognize one of the defining properties of a universe: it should be closed under pullback. After all, any fiber of a pulled back map is a fiber of the original map, so if a given map has small fibers so will all of its pullbacks. Notice that this automatically closes small maps under isomorphism.

Another important requirement is that small maps should be closed under composition. In $\Sets$, this corresponds to the assumption that a small coproduct of small sets is again small. This is essentially the axiom of replacement: when $A$ is a set and $\{F_a\ |\ a\in A\}$ is a (disjoint) family of sets, then the union $F=\bigcup_{a\in A} F_a$ is again a set. More generally, this translates to closure under the type-theoretic operation of dependent sums.

Finite sets provide a motivating example for these issues. Consider the following map in $\Sets$:
$$K=\{(k,n)\in \NN\times\NN\ |\ k<n\} \stackrel{p_2}{\lto} \NN.$$
This projection has the nice property that, for every $m\in\NN$, the fiber of $K$ over $m$ contains exactly $m$ elements. More generally, any function $F\to A$ whose fibers are finite is realized as a pullback of $K$:
$$\xymatrix{
F \pbcorner \ar[rr] \ar[d] && K \ar[d]\\
A \ar[rr]_{a\mapsto \|F(a)\|} && \NN.
}$$
We say that $K\to \NN$ is \emph{generic} for maps with finite fibers.

The third axiom of small maps says that any categorical universe should contain such a generic display map. Any class of small maps should contain a map $\rm{El}(U)\to U$ such that every small $s:E'\to E$ arises as the pullback against some $f:E\to U$:
$$\xymatrix{
E' \pbcorner \ar[r] \ar[d]_{\forall\ \rm{small}\ s} & \rm{El}(U) \ar[d]\\
E \ar[r]_{\exists\ f} & U.\\
}$$
Note that in general neither $U$ nor $\El(U)$ will themselves by small objects.

\begin{defn}
A \emph{coherent universe} in a category $\EE$ is a class of maps $\UU\subseteq \rm{Ar}(\EE)$ such that:
\begin{itemize}
\item $\UU$ is closed under pullbacks.
\item $\UU$ is closed under composition/dependent sums.
\item There is a an object $U$ and a morphism $\pi_{\UU}:\El(U)\to U$ in $\UU$ such that every small map is a pullback of $\pi_{\UU}$.
\end{itemize}\end{defn}

This is a generalization of Streicher's definition of a universe in a topos (cf. \cite{streicher}), from which we have removed two conditions which are inappropriate to the context of pretoposes.

The first says that all monos should be small maps. This is intuitively reasonable, since the fibers of a mono are either empty or singletons. However, this is not as straightforward as it seems. We can reformulate the example of finite sets in any pretopos containing a (parameterized) natural numbers object. This amounts to a weak intuitionistic set theory, and in these cases the finite sets may not be closed under subobjects. This means that we can have a small object which contains non-small subobjects, in the sense that they do not arise as pullbacks of $K$.

Something similar occurs in our context, where smallness will correspond to definability. As already observed in Chapter 2, not every equivariant subobject of a definable sheaf is definable; a further compactness condition is required. Our notion of smallness incorporates, in some sense, smallness of definition in addition to smallness of fibers and so closure under monomorphisms is an unreasonable requirement in this setting.

Additionally, the definition of a universe in a topos usually requires small maps to be closed under dependent products as well as sums. As we are working with pretoposes, which do not model dependent products, it is reasonable to omit this requirement as well.

\begin{defn}
A map of equivariant sheaves $f:E'\to E\in\EqSh(\MM)$ is \emph{definably small} if the pullback along any map from a definable sheaf $\ext{A}\to \FF$ (cf. page \pageref{def_sheaf}) is a definable map:
$$\xymatrix{
[\![\ F\extr  \pbcorner \ar[d]_{\ext{\tau}} \ar[r] & E' \ar[d]^{f}\\ 
\ext{A} \ar[r] & E.
}$$
\end{defn}
These maps are also called \emph{representably small}, because the equivalence $\EqSh(\MM)\simeq\Sh(\EE)$ carries each definable sheaf $\ext{E}$ to the corresponding representable functor $yE$.

By stacking pullback squares either horizontally or vertically, the two pullbacks lemma immediately verifies both the first and second requirements for a universe:
$$\begin{array}{p{7cm}p{5cm}}
\raisebox{-1cm}{$\xymatrix{
 q^*(p^*E')\cong (pq)^*E' \pbcorner \ar@/^3ex/[rr] \ar[r] \ar[d]_{\rm{definable}} & p^*E' \pbcorner \ar[d] \ar[r] & E' \ar[d]^{\rm{small}}\\
\ext{A} \ar[r]_q & D \ar[r]_p & E,\\
}$} &
\xymatrix{
\ext{F} \pbcorner \ar[d] \ar@/_3ex/[dd]_{\rm{definable}} \ar[r] & E'' \ar[d]^{\rm{small}}\\
\ext{B} \pbcorner \ar[d] \ar[r] & E' \ar[d]^{\rm{small}}\\
\ext{A} \ar[r] & E.\\
}
\end{array}$$
The most interesting aspect of this universe is the generic map.

Let $L$ denote the ``walking section'', the category with three non-identity arrows
$$\xymatrix{
0 \ar@<1ex>[r]^p \ar@(dl,ul)^r & \ar@<1ex>[l]^s 1, & r=s\circ p,& p\circ s=\id_1.
}$$
The inclusion of $\2$ into $L$ (as $p$) induces a functor $\EE^L\to\EE^{\2}$. This makes $\EE^L$ into a stack over $\EE$, so we sheafify it and transport it across the equivalence $\Sh(\EE)\simeq\EqSh(\MM_{\EE})$ just as we did with for the structure sheaf. This gives our desired generic map.

\begin{thm}
Let $\pi_{\EE}:\El(\OO_{\EE})\to\OO_{\EE}$ denote the map of equivariant $\MM_{\EE}$-sheaves which arises as the sheafification and transport of the canonical functor $\EE^L\to\EE^{\2}$. This map is generic for definably small maps in $\EqSh(\MM_{\EE})$.
\end{thm}
\begin{proof}
First we show that every definable map $\ext{\tau}:\ext{B}\to\ext{A}$ arises as a pullback of $\El(\OO_{\EE})$. From proposition \ref{str_sheaf_secs} we know that there is an equivalence between partial sections of the structure sheaf and objects of the slice category
$$\OO_{\EE}(V_{A(k)})\simeq\EE\!/A.$$

Using this we can associate $\tau$ with a partial equivariant section $t:V_{A(k)}\to\OO_{\EE}$. According to lemma \ref{equiv_ext}, this has a unique extension to an equivariant map ${\overline{t}:\ext{A}\to\OO_{\EE}}$. Now consider the pullback
$$\xymatrix{
P \pbcorner \ar[r] \ar[d] & \El(\OO_{\EE}) \ar[d]^{\pi_{\EE}}\\
\ext{A} \ar[r]_{\overline{t}} & \OO_{\EE}.\\
}$$
We claim that the left-hand map is isomorphic to $\ext{\tau}$.

Both strictification and transport preserve pullbacks, so it will be enough to show that the following diagram of functors is a (strict) pullback in stacks over $\EE$:
$$\xymatrix{
\EE\!/B \pbcorner \ar[d]_{\EE\!/\tau} \ar[r] & \EE^L \ar[d]\\
\EE\!/A \ar[r]_{\overline{\tau}} & \EE^{\2}\\
}$$
Here $\EE\!/A$, e.g. is the (representable) stack fibered over its domain as in proposition \ref{2yoneda}.

Now suppose that we have an object $\epsilon\in\EE\!/A$ and a section $\<p,t\>\in\EE^\sigma$. To say that these agree over $\EE^{\2}$ means that $\overline{\tau}(\epsilon)\cong\epsilon^*(\tau)\cong p$, and $t$ is a section of this pullback:
$$\xymatrix{
\epsilon^*B \pbcorner \ar[d]^{\overline{\tau}(\epsilon)} \ar[rr] && B \ar[d]^{\tau}\\
E \ar[rr]_{\epsilon} \ar@/^2ex/[u]^t && A\\
}$$
But this data is precisely equivalent to a map $\beta:E\to B$ such that $\beta\circ\tau=\epsilon$. Since $\EE\!/-$ acts by precomposition, this means that $\EE\!/B$ is the pullback of $\EE^\sigma$ along $\hat{\tau}$.

Now suppose that some equivariant sheaf map $f:D\to E$ is definably small. We must define a map $\name{f}:E\to\OO_{\EE}$ such that we have an isomorphism $\name{f}^*(\EE^L)\cong D$ over $E$. We define $\name{f}$ locally by giving the image of any basic open section in $E$. For such a partial section $e:V_{A(k)}\to E$ there is a unique equivariant extension $\ol{e}:\ext{A}\to E$ and, because $f$ is definably small, its pullback is a definable map $\ext{\beta}$. This definable map induces another a partial section $\name{\beta}$ in the structure sheaf and this section will be the image of $e$ under $\name{f}$.

We can see that $\name{f}$ is well-defined by checking that its definition agrees whenever two sections $b$ over $V_{B(j)}$ and $c$ over $V_{C(k)}$ overlap, as in the diagram below. That overlap will contain a subsection over a smaller basic open neighborhood $V_{A(i,j,k)}$, and these lift to equivariant maps $\ext{A}\to\ext{B}$ and $\ext{A}\to\ext{C}$. Because the sections $b$ and $c$ overlap, these canonical lifts will commute:
$$\xymatrix{
\ext{A} \ar@{-->}[rr] \ar@{-->}[rd] && \ext{C} \ar@{->}[rd]^{\overline{c}} &\\
& \ext{B} \ar@{->}[rr]_(.3){\overline{b}} && E \\
V_{A(i,j,k)} \ar[uu] \ar@{>->}[rr]|(.52)\hole \ar@{>->}[rd] && V_{C(j)} \ar[uu]|\hole \ar[ur]^{c} \\
& V_{B(k)} \ar@/_5ex/[uurr]_(.3){b} \ar[uu] \\
}$$

Now suppose that $d$ is a lift of $e$ along the map $f$ (i.e., $f\circ d=e$) as in the diagram below. This induces a equivariant lift $\overline{d}:\ext{A}\to D$ such that $f\circ\overline{d}=\overline{e}$, and hence a section of the pullback $\ext{\beta}$. $$\xymatrix{
\ext{B} \ar[dd]^{\ext{\beta}} \ar[rr] && D \ar[dd]_f \ar@{-->}[rr] && \El(\OO_{\EE}) \ar[dd]\\
&&& V_{A(k)} \ar[ld]_e \ar[ul]^{d} \ar[dr]_{\name{\beta}} \ar[ur]^{\name{\<\beta,\delta\>}}\\
\ext{A} \ar[rr]_{\ol{e}} \ar@/^3ex/@{-->}[uu]^{\ext{\delta}} \ar[uurr]_{\ol{d}}&& E \ar[rr] &&\OO_{\EE}\\
}$$
The pair $\<\beta,\delta\>$ defines a partial section of $\El(\OO_{\EE})$ over $V_{A(k)}$ and an argument like the one above shows that these agree on their overlaps. Therefore they patch together to give a map $D\to\EE^L$. The projection $\EE^L\to\EE^{\2}$ sends $\name{\<\beta,\delta\>}$ to $\name{\beta}$, so the map commutes over $\EE^{\2}$. Moreover, the fact that $\delta$ is uniquely determined by the canonical lift $\ol{d}$ implies that the right-hand square is a pullback.

This demonstrates that the projection $\El(\OO_{\EE})\to\OO_{\EE}$ is generic for definably small maps in $\EqSh(\MM)$ and completes the proof.
\end{proof}

The results of this section show that we can regard the map $\El(\OO_{\EE})\to\OO_{\EE}$ as a universe for an interpretation of (a weak form of) dependent type theory which involves dependent sums but not dependent products. More specifically, when $\EE$ is locally Cartesian closed, this can act as a universe for all of dependent type theory \cite{repre_models}. Following the results of \cite{AST} we can also use this as the basis for a model of algebraic set theory.

\section{Isotropy}

In this section we discuss a connection between our logical schemes and recent developments in topos theory. Motivated by constructions in semigroups, Funk, Hofstra and Steinberg have recently discovered a canonical group object internal to any topos \cite{isotropy}. Here we give a logical interpretation of this isotropy group, showing that it is closely related to the structure sheaf of our logical schemes. This provides a new perspective on this construction, and also provides an easy calculation of the stalks of the group. As a corollary of our earlier results, we also give an external description of the isotropy group relating it to the Sierpinski topos.

If $\FF$ is a topos, then so is each slice category $\FF\!/A$, and any morphism $f:B\to A$ in $\FF$ induces a geometric morphism $\overline{f}:\FF\!/B\to\FF\!/A$ whose inverse image is given by pullback along $f$. In particular, every slice topos has a canonical geometric morphism $\pi_A:\FF\!/A\to\FF$. 

Funk, Hofstra and Steinberg define an \emph{isotropy functor} ${\ZZ:\FF^{\op}\to\bf{Grp}}$ which sends each object $A\in\FF$ to the group of natural automorphisms of $\pi_A$ (i.e., of the inverse image $\pi_A^*$):
$$\ZZ(A)\cong\left\{\alpha\ \bigg|\  \xymatrix{**[l]\FF\!/A \rtwocell{\omit \vertiso\ \alpha} & \FF}\right\}$$
We sometimes denote such an automorphism by writing $\alpha:\FF\!/A\auto\FF$. The action of a morphism $f:B\to A$ on $\ZZ$ is given by composition with $\overline{f}$.

As they demonstrate, $\ZZ$ preserves colimits (i.e., sends colimits in $\FF$ to limits in $\bf{Grp}$) from which it follows that $\ZZ$ is representable: there is an internal group $Z\in\bf{Grp}(\FF)$ such that $\ZZ\cong\Hom_{\FF}(-,Z)$. To define $Z$ from $\ZZ$, represent $\FF$ as the topos of sheaves on a site $(\bf{C},\JJ)$. The composite $\bf{C}^{\op}\stackrel{y}{\lto}\FF^{\op}\stackrel{\ZZ}{\lto}\bf{Grp}$ is a sheaf of groups and (naturally, for any object $\displaystyle F\cong\colim_{i\in\int F} yC_i$)
$$\begin{array}{rclcl}
\ZZ(F)&\cong&\ZZ(\colim_i yC_i)\\
&\cong& \lim_i \ZZ(yC_i)\\
&\cong& \lim_i Z(C_i)\\
&\cong& \lim_i \Hom_{\FF}(yC_i,Z)\\
&\cong& \Hom_{\FF}(\colim_i yC_i,Z) & \cong & \Hom_{\FF}(F,Z).\\
\end{array}$$

\begin{defn}
The \emph{isotropy group} of a topos $\FF$ is the internal group $Z\in\bf{Grp}(\FF)$ which represents the isotropy functor (i.e., $\ZZ\cong\Hom_{\FF}(-,Z)$).
\end{defn}

\begin{lemma}
When $(\bf{C},\JJ)$ is a subcanonical site closed under products then for any $A\in\bf{C}$ the Yoneda embedding induces a canonical isomorphism
$$\ZZ(yA)=\Aut(\pi_{yA}^*)\cong\Aut(A^\times).$$
\end{lemma}

\begin{proof}
Here $A^\times:\bf{C}\to\bf{C}/A$ is the functor sending each object $C\in\bf{C}$ to the second projection $C\times A\to A$. The slice category $\bf{C}/A$ inherits a topology from $\bf{C}$ and (because $J$ is subcanonical) there is an equivalence $\Sh(\bf{C}/A)\simeq\Sh(\bf{C})/yA$ (cf. \cite{elephant}, C2.2.17).

Let $y_A$ denote the sliced Yoneda embedding $\bf{C}/A\to\Sh(\bf{C}/A)\simeq\Sh(\bf{C})/yA$. Since both $A^\times$ and $\pi_{yA}^*$ act by taking products (with $A$ and $yA$, respectively) there is a factorization $\pi_{yA}^*\circ y\cong y_A\circ A^\times$:
$$\pi_{yA}^*(yC)= \raisebox{.7cm}{\xymatrix{yC\times yA\ar[d]\\yA\\}} \cong \raisebox{.7cm}{\xymatrix{y(C\times A)\ar[d]\\yA\\}}= \ y_A\left(\raisebox{.7cm}{\xymatrix{C\times A\ar[d]\\A\\}} \right)=y_A(A^\times(C)).$$

Now suppose $\alpha$ is a natural automorphism of $\pi_{yA}^*$. Since $\pi_{yA}^*$ factors through $A^\times$ while $y_A$ is full and faithful, the components of $\alpha$ descend uniquely to an automorphism of $A^\times$:
$$\xymatrix{
\Sh(\EE) \ar[rr] \ar[rr]|{\rotatebox{270}{$\curvearrowleft$}}^(.45){\alpha}_(.7){\pi_{yA}^*} && **[r] \Sh(\EE\!/A)\simeq\Sh(\EE)/yA \\\\
\EE \ar[uu]^y \ar[rr] \ar[rr]|{\rotatebox{270}{$\curvearrowleft$}}^(.43){\alpha_0}_(.7){A^\times} && **[r] \EE\!/A \ar[uu]_y\\
}$$
Thus restriction along $y$ defines a map $\Aut(\pi_{yA}^*)\to\Aut(A^\times)$; since the group action in either case is defined by composition, this is obviously a group homomorphism.

Now we must show that the map $\alpha\mapsto\alpha_0$ is invertible. To see this, fix a sheaf $E\in\Sh(\bf{C})$ and represent it as a colimit $E\cong\colim_j yB_j$. If $\beta$ is a natural automorphism of $A^\times$ this gives us a family of automorphisms $\beta_j:B_j\times A\cong B_j\times A$. By naturality these isos commute with the colimit presentation $E$, inducing an automorphism $\overline{\beta_E}:E\times yA\cong E\times yA$:
$$\xymatrix{
y(B_j\times A) \ar[r] \ar[d]_{y\beta_j} & y(B_{j'}\times A) \ar[d]_{y\beta_{j'}} \ar[r] & E\times yA \ar@{-->}[d]^{\overline{\beta_E}}\\
y(B_j\times A) \ar[r] & y(B_{j'}\times A) \ar[r] & E\times yA.\\
}$$
Using the fact that $\overline{\beta_E}$ is uniquely determined from $\beta$, it is easy to show that these two maps are mutually inverse.
\end{proof}

\begin{cor}\label{isot_desc}
If $\EE$ is a pretopos, the isotropy group of $\Sh(\EE)$ can be defined directly from $\EE$:
$$Z(A)\cong \Aut(A^\times)=\left\{\alpha\ \bigg|\ \xymatrix{\EE \rtwocell^{A^\times}_{A^\times}{\omit \vertiso\ \alpha} & **[r] \EE\!/A}\right\}$$
\end{cor}
\begin{proof}
This follows immediately from the lemma, given that the coherent topology is subcanonical and pretoposes are closed under products.

As with any sheaf, we may represent $Z$ as a fibration over $\EE$ and its description in this context is particularly nice. On one hand, we have the codomain fibration $\EE^{\2}$, whose fiber over $A$ is exactly the slice category $\EE/A$. On the other, the constant fibration $\Delta\EE\cong \EE\times\EE$ has $\EE$ for each fiber. There is also a canonical fibered functor $T:\Delta\EE\to\EE^{\2}$, the transpose of the identity functor $\EE\to\Gamma\EE^{\2}\cong\EE$. The component of $T$ at $A$ is exactly $A^\times$, so the corollary says precisely that (the Grothendieck construction applied to) $Z$ is the group of (fibered) natural automorphisms of $T$:
$$Z\cong\Aut_{\EE}(T:\Delta \EE\auto \EE^{\2}).$$
\end{proof}

\begin{defn}
Fix an $\EE$-model $M$. We say that an automorphism ${\alpha:M\cong M}$ is \emph{(parameter-)definable} if for every (basic) sort $B$ there is an object $A_B$, an element $a\in A_B^M$ and a formula $\sigma(y,y',x)$ (where $x:A$ and $y,y':B$) such that
$$\alpha(b)=b'\Iff M\models\sigma(b,b',a).$$
\end{defn}

\begin{defn}\label{Mdef}
We say that a family of formulas $\{\sigma_B(y,y',x)\}$ (where $x:A$ is fixed and $y,y':B$, range over all (basic) sorts $B$) is an \emph{$A$-definable automorphism} if for every model $M$ and every element $a\in A^M$ the parameterized formulas $\sigma_B(y,y',a)$ yield a parameter-definable automorphism.

Given a model $M$, we say that a family of parameterized formulas $\{\sigma_B(y,y',a_B)\}$ in $\Diag(M)$ is an \emph{$M$-definable automorphism} if for every homomorphism $h:M\to N$ parameterized formulas $\sigma_B(y,y',h(a_B))$ yield a parameter-definable automorphism of $N$.
\end{defn}

\begin{lemma}
The family $\{\sigma_B(y,y',x)\}$ is an $A$-definable automorphism just in case $\EE$ proves the following sequents:
$$\begin{array}{ccl}
\underset{x,y}{\vdash} \exists y'.\sigma_B(y,y',x) & \sigma_B(y,y',x)\wedge\sigma_B(y,y'',x)\underset{x,y,y',y''}{\vdash} y'=y''\\[3ex]
\underset{x,y'}{\vdash}\exists y.\sigma_B(y,y',x) & \sigma_B(y,y'',x)\wedge\sigma_B(y',y'',x)\underset{x,y,y',y''}{\vdash} y=y'\\[3ex]
\sigma_B(y,y',x)\wedge R(y)\underset{x,y,y'}{\vdash} R(y') & \sigma_B(y,y',x) \underset{x,y,y'}{\vdash}\sigma_C(f(y),f(y'),x)\\ 
\end{array}$$
\end{lemma}
\begin{proof}
This is immediate from completeness. The sequents on the first line specify functionality $y\mapsto y'$; the second line specifies invertibility. The third line (where I have simplified by assuming that $R$ and $f$ are unary) says that the family of maps defined by $\sigma_B$ respects the basic functions and relations, ensuring that the bijection is a model homomorphism and hence an automorphism.
\end{proof}

A good example to keep in mind is conjugation of groups. The classifying pretopos $\EE_{\Grp}$ contains an object $U$ which represents the underlying set; for any group $G:\EE_{\Grp}\to\Sets$ we have $U^G=|G|$. Because the theory of groups is single-sorted, $U$ is the only basic sort. Conjugation is a $U$-definable automorphism with a defining formula
$$\sigma_U(y,y',x)\Iff y'=xyx^{-1}.$$
The next lemma shows that this exactly corresponds to a natural automorphism $\EE_{\Grp}\auto\EE_{\Grp}/U$.

\begin{lemma}\label{def_aut}
For each $A\in\EE$, the isotropy group $Z(A)$ is isomorphic to the family of $A$-definable automorphisms in $\EE$.
\end{lemma}

\begin{proof}
Recall that $\EE\!/A$, regarded as a logical theory, represents the original theory $\EE$ extended by a single constant $\bf{c}_a:A$. Given a model $\<M,a\>$ of the extended theory and a natural automorphism $\alpha:\EE\auto\EE/\!A$, composition induces a model automorphism $M\underset{a}{\cdot}\alpha:M\cong M$
$$\xymatrix{
\EE\!/A \ar[rrd]^{\<M,a\>} \\
\EE \ar[u] \ar[u]|{\curvearrowleft}^(.3){\alpha} \ar[rr] \ar@{}[rr]|(.45){\rotatebox{270}{\footnotesize$\curvearrowleft$}}_(.6){M\underset{a}{\cdot}\alpha} && \Sets\\
}$$

The fact that each component $(M\underset{a}{\cdot}\alpha)_B:B^M\cong B^M$ is the image of a map in $\EE\!/A$ means that there is a formula $\sigma_B(y,y',x)$ (where $x:A$ and $y,y':B$) such that
$$(M\underset{a}{\cdot}\alpha)_B(b)=b' \Iff M\models\sigma_B(b,b',a).$$
This is precisely an $\EE$-definable automorphism depending on a parameter $x:A$.

On the other hand, suppose that we have a family of formulas $\{\sigma_B(y,y',x)\}$ ranging over the (basic) sorts of $\EE$ and which define an $\EE$-definable automorphism (parameterized by variable $x:A$). By naturality, we are forced to set
$$\sigma_{B\times C}(\<y,z\>,\<y',z'\>,x)=\sigma_B(y,y',x)\wedge\sigma_{C}(z,z',x),$$
$$\xymatrix{
B\times A \ar[d]_{\sigma_B} & B\times C\times A \ar[r] \ar[l] \ar[d]|{\exists ! \sigma_{B\times C}=\sigma_B\wedge\sigma_C} & C\times A \ar[d]^{\sigma_C}\\
B & B\times C \ar[l] \ar[r] & C\\
}$$
In much the same way, there is a unique extension of $\sigma$ to any pullback, coproduct or quotient.

Because these maps are assumed to define an $\EE$-model automorphism they respect basic relations and this allows us to define a restriction $\sigma_R$ for any basic relation $R \leq B$ (which, for simplicty, we take to be unary) 
$$\sigma_R(y,y',x)=\sigma_B(y,y',x)\wedge R(y)\wedge R(y').$$
Similarly, respect for basic functions $f(y)=z$ allows us to preserve naturality:
$$\xymatrix{
R\times A \ar@{>->}[r] \ar@{-->}[d]_{\sigma_R} & B\times A \ar[d]^{\sigma_{B}} & & B\times A \ar[r]^-{f\times A} \ar[d]_{\sigma_{B}} & C\times A \ar[d]^{\sigma_C}\\
R \ar@{>->}[r] & B & & B \ar[r]_f & C.\\
}$$
This shows that any $A$-definable automorphism can be extended to a natural automorphism $\EE\auto\EE\!/A$. Using the uniqueness of the extension of $\sigma$ from basic sorts to products, coproducts, etc., one easily shows that these constructions are mutually inverse.
\end{proof}

In order to connect the isotropy group with our structure sheaves, we use the description of $Z$ in terms of fibrations given in corollary \ref{isot_desc}. We can strictify this map and transport it across the equivalence $\Sh(\EE)\simeq\EqSh(\MM_{\EE})$, just as we did for the structure sheaf in chapter 3, section 1. However, the constant fibration $\Delta\EE=\EE\times\EE$ is already a sheaf (over $\EE$), so it is uneffected by strictification. Transporting it across the equivalence sends it to another constant sheaf over $\MM$ (also denoted $\Delta\EE$). Alternatively, we can think of $\tau$ as the transpose of the identity functor $\EE\to\EE\simeq\Gamma\OO_{\EE}$ under the adjunction $\Delta\dashv\Gamma$.
$$\xymatrix{
&& \ol{\EE^{\2}} \ar[d]\\
\Delta\EE \ar[rr]^{T} \ar[rru]^{\ol{T}} \ar[rd]_{p_2} && \EE^{\2} \ar[ld]^{\cod} &\ar@{|->}[r] \ar@{}@<3ex>[r]^{\EqSh(\MM)\to\Sh(\EE)}&& \Delta\EE \ar[rr]^\tau \ar[rd] && \OO_{\EE} \ar[ld]\\
&\EE&&&&&\MM_{\EE}&\\
}$$

\begin{prop}
The equivalence $\Sh(\EE)\simeq\EqSh(\MM_{\EE})$ identifies the isotropy group $Z$ with the group of automorphisms of $\tau$.
\end{prop}
\begin{proof}
We already know that $Z$ is equivalent to the (internal) group of fibered automorphisms of $T$. Strictification is a pointwise equivalence of categories, so any fibered automorphism of $T$ lifts uniquely to $\ol{T}$. Similarly, the equivalence $\Sh(\EE)\simeq\EqSh(\MM_{\EE})$ preserves natural automorphisms of (strict) internal categories, leaving us with the following isomorphisms (modulo the Grothendieck construction, applied to $Z$)
$$Z\cong\Aut(T)\cong\Aut\left(\ol{T}\right)\cong\Aut(\tau).$$
\end{proof}

This shows that the isotropy group can be defined directly from our structure sheaf. This also provides an immediate calculation of the stalks of $Z$.

\begin{cor}
Given an $\EE$-model $M$, the stalk $Z_M$ is the group of $M$-definable automorphisms from definition \ref{Mdef}.
\end{cor}
\begin{proof}
The argument follows the same lines as lemma \ref{def_aut}. One can define the natural automorphism group of an internal functor using only finite limits, so it is preserved when passing to the stalks. This means that we can use the stalks of $\tau:\Delta\EE\to\OO_{\EE}$ to compute $Z_M$:
$$\begin{array}{rcl}
Z_M&\cong &\Aut\big(\tau:\Delta\EE\auto\OO_{\EE}\big)_M\\
&\cong& \Aut\big(\tau_M:(\Delta\EE)_M\to(\OO_{\EE})_M\big)\\
\end{array}$$

The stalk of $\OO_{\EE}$ at $M$ is the complete diagram of $M$ (lemma \ref{struct_stalks}) and the stalk of $\tau$ is the canonical inclusion $\top:\EE\to\Diag(M)$. If $\alpha$ is a natural automorphism of this inclusion, then its components are isomorphisms in $\Diag(M)$, and these are all parameter-definable maps. In particular, for every $B\in\EE$ we have a formula $\sigma_B(y,y',a_B)$ for some sort $A_B\in\EE$ and some element $a_b\in A_B^M$.

Remember that $\Diag(M)$ classifies homomorphisms $h:M\to N$. If $H$ classifies a homomorphism $h:M\to N$, then $N$ is reduct of $H\circ\top$:
$$\xymatrix{
\Diag(M) \ar[rrd]^{h:M\to N} \\
\EE \ar[u]^\top \ar[rr]_{N} &&\SS\\
}$$
This tells us that for any homomorphism $h$, the formulas $H(\alpha_B)=\sigma_B(y,y',h(a_B))$ induce and automorphism of $N$. Thus the family $\alpha_B=\sigma_B(y,y',a_B)$ is precisely an $M$-definable automorphism.

On the other hand, if $\{\sigma_B(y,y',a_B)\}$ is an $M$-definable automorphism, then completeness (relative to homomorphisms $M\to N$) ensures that $\Diag(M)$ proves that these parameterized formulas define isomorphisms $\top(B)\cong\top(B)$. Similarly, these automorphisms are natural in $B$ because the naturality conditions are satisfied in every $\Diag(M)$ model: a model automorphism $N\cong N$ is a exactly a natural transformation $\xymatrix{\EE \rtwocell^N_N{\omit\vertiso} & \SS}$.

Thus every stalk automorphism $\alpha\in Z_M$ defines an $M$-definable automorphism and vice versa, and these constructions are mutally inverse. In either case, the group multiplication is interpreted by composition so these are obviously group homomorphisms, proving that $Z_M\cong\Aut(\tau_M)$ is the group of $M$-definable automorphisms.
\end{proof}

\begin{cor}
To every coherent or classical first-order logical theory $\bTT$ we can associate an equivariant sheaf of groups $Z_{\bTT}$ over the spectral groupoid $\MM_{\bTT}$ such that, for every labelled model $\mu\in\MM_{\bTT}$, the stalk of $Z_{\bTT}$ is a normal subgroup of $\Aut(M_\mu)$. Moreover, this subgroup does not depend on the labelling of $\mu$.
\end{cor}

\begin{proof}
This follows almost immediately from the previous corollary by taking $Z_{\bTT}$ to be the isotropy group of the classifying topos of $\bTT$. Because $Z_{\bTT}$ is equivariant, it cannot depend on the labelling of $\mu$. The only thing we must check is that the group of $M_\mu$-definable automorphisms is normal. If $\alpha:M_\mu\cong M_\mu$ is definable then, for every basic sort $A\in\bTT$ there is a formula $\sigma(x,x',y)$ (in context $A\times A\times B$) and an element $b\in B^\mu$ such that
$$\alpha_A(a)=a' \Iff M_\mu\models \sigma(a,a',b).$$
If $\beta$ is any other automorphism of $M_\mu$, let $\gamma=\beta^{-1}\circ\alpha\circ\beta$. Then the component of $\gamma$ at $A$ is given by
$$\gamma_A(a)=a' \Iff \alpha(\beta(a))=\beta(a') \Iff M_\mu\models \sigma(\beta(a),\beta(a'),\beta(b)).$$
Since $\beta$ is an automorphism, $\gamma_A(a)=a'$ if and only if $\gamma_A(\beta^{-1}(a))=\beta^{-1}(a')$ and therefore $\gamma$ is definable by the formula $\sigma(x,x',\beta(b))$. Since any conjugate of an $M_\mu$-definable automorphism is again $M_\mu$-definable, these form a normal subgroup.
\end{proof}

\begin{comment}

\begin{cor}
Suppose that $\alpha:M\cong M$ is a parameter-definable automorphism, defined by some formulas $\sigma_B(y,y',a)$ with some $a\in A^M$. Then there is a subquotient $A \geq S \stackrel{q}{\epi} Q$ such that $M\models S(a)$ and the family $\sigma_B$ is a $Q$-definable automorphism over $q(a)$.
\end{lemma}

\begin{proof}

Notice that $Z_M$ is defined as a colimit
$$Z_M\cong\underset{A\in \EE, a\in A^M}{\colim} Z(A).$$
To each pair $\<\alpha,a\>$ with $a\in A^M$ and $\alpha:\EE\auto\EE\!/A$ we associate the parameter-definable automorphism
$$M\underset{a}{\cdot} \alpha: \EE \auto \EE\!/A \stackrel{\<M,a\>}{\lto} \Sets.$$
The map $\alpha\mapsto M\underset{a}{\cdot}\alpha$ is obviously a group homomorphism and, by

\end{proof}

\end{comment}

\section{Conceptual Completeness}

Makkai \& Reyes' theorem of conceptual completeness \cite{MakkaiReyes} says roughly that a pretopos $\EE$ is determined up to equivalence by its category of models $M:\EE\to\Sets$. More specifically, recall (cf. prop \ref{embed_prop}) that the \emph{reduct} of an $\FF$-model $N:\FF\to\Sets$ along an interpretation $I:\EE\to\FF$ is the composite $I^*N:\EE\to\FF\to\Sets$. This induces a functor $I^*:\Mod(\FF)\to\Mod(\EE)$ and conceptual completeness says the following:
\begin{thm}[Conceptual Completeness, Makkai \& Reyes, \cite{MakkaiReyes}]
A pretopos functor $I:\EE\to\FF$ is an equivalence of categories if and only if the reduct functor $I^*:\Mod(\FF)\to\Mod(\EE)$ is an equivalence of categories.
\end{thm}

The left-to-right direction of conceptual completeness is immediate; if $\EE'\simeq\EE$ is an equivalence, then one easily shows that precomposition $\EE'\simeq\EE\stackrel{M}{\lto}\Sets$ induces an equivalence $\Mod(\EE)\simeq\Mod(\EE')$. The real argument is in the proof of the converse.

Note that the analogous theorem for schemes is immediate: if the induced map $\<I_\flat,I^\sharp\>:\Spec(\FF)\iso\Spec(\EE)$ is an equivalence of affine schemes, then the global sections $\Gamma_{\eq} I^\sharp:\EE\to\FF$ is necessarily an equivalence of pretoposes. However, we can mine the Makkai \& Reyes proof of conceptual completeness in order to establish a sort of Galois theory at the level of spectral groupoids.

Their proof has the following structure (where e.s.o means essentially surjective on objects):

\begin{tabular}{rlcl}
(i) & $I^*$ e.s.o. &$\To$ &$I$ faithful.\\
(ii) & $I^*$ e.s.o. + full & $\To$ & $I$ full on subobjects.\\
(iii) & $I^*$ an equivalence & $\To$ & $I$ subcovering.\\
(iv) & $I$ full on subs + faithful & $\To$ & $I$ full.\\
(v) & $I$ subcovering + full on subs + faithful & $\To$ & $I$ e.s.o.\\
\end{tabular}
\vspace{0cm}

There is clearly an element of duality in this proof, matching ``surjectivity'' conditions with ``injectivity'' conditions, but it is obscured by additional conditions in (ii) \& (iii). In this section we will refactor the existing proof in order to give biconditional statements for (i)-(iii). We also provide another set of equivalent conditions relating syntactic properties of $I$ to topological/groupoidal properties of the spectral dual $I_\flat$.

Our main theorem is below (with definitions to follow). We will prove each of the statements (a), (b) and (c) in turn.
\begin{thm}\label{scheme_cc}
Given a pretopos functor $I:\EE\to\FF$ with a reduct functor $I^*:\Mod(\FF)\to\Mod(\EE)$ and spectral dual $I_\flat:\Spec(\FF)\to\Spec(\EE)$ the items in each row of the following table are equivalent:

\noindent\begin{tabular}{|c|c|c|c|}
\hline
&	\textbf{Syntactic} & \textbf{Semantic} & \textbf{Spectral} \\\hline
(a) & $I$ is conservative & $I^*$ is supercovering & $I_\flat$ is superdense \\\hline
(b) & $I$ is full on subs & $I^*$ stabilizes subobjects &  $I_\flat$ separates subgroupoids \\\hline
(c) & $I$ is subcovering & $I^*$ is faithful & $I_\flat$ is non-folding \\\hline
\end{tabular}
\end{thm}
\vspace{0cm}

A central ingredient in our proof with be the so-called local scheme of $\Spec(\EE)$ at $\mu$. Recall that, for an element of an affine scheme $\mu\in\Spec(\EE)$, the stalk of the structure sheaf $\OO_{\EE}$ at $\mu$ is equivalent to the (syntactic pretopos of) the Henkin diagram of the model $M_\mu$ (denoted $\Diag(\mu)$, cf. lemma \ref{struct_stalks}). Accordingly, for any scheme $\XX$ and any point $x\in\XX$, we let $\Diag(x)$ denote the stalk of $\OO_{\XX}$ at $x$.
\begin{defn}
Given a logical scheme $\XX$ and a point $x\in\XX$, the \emph{local scheme} of $\XX$ at $x$ (which we denote $\XX_x$) is the affine scheme which is dual to the pretopos $\Diag(x)$.
\end{defn}

The local scheme has a canonical projection $\pi_x:\XX_x\to\XX$. This can be defined by choosing an affine neighborhood $\UU\simeq\Spec(\EE)$ containing $x$, so that $x=\mu$ is a labelled $\EE$-model. Then the projection $\pi_x$ is dual to the canonical interpretation $\EE\to\Diag(\mu)$. Here a model of $\Diag(\mu)$ corresponds to $\EE$-model homomorphisms $h:\mu\to\mu'$, and $\pi_x$ sends $h\mapsto\mu'$. One can easily check that this definition does not depend on the choice of $\UU$.

\begin{defn}
Given a scheme morphism $J:\YY\to\XX$, the \emph{blowup} of $\YY$ at $x$ is the pullback
$$\xymatrix{
\rm{Bl}_x(\YY) \pbcorner \ar[rr] \ar[d]_{J_x} && \YY \ar[d]^{J}\\
\XX_x \ar[rr]_{\pi_x} && \XX.\\
}$$
We say that $\YY$ is consistent near $x$ if the blowup of $\YY$ at $x$ is not the empty scheme.
\end{defn}

\begin{lemma}
Fix an interpretation $I:\EE\to\FF$.
\begin{itemize}
\item For any object $A\in\EE$ there is a canonical functor $I_A:\EE\!/A\to\FF\!/IA$ and this is the pushout of $I$ along $A^\times:\EE\to\EE\!/A$.
\item Any localization $t:\EE\to\EE_L$ defines a new localization $It:\FF\to\FF_{IL}$. There is a canonical functor $I_L:\EE_L\to\FF_{IL}$ and this is the pushout of $I$ along $t$.
\end{itemize}
$$\xymatrix{
\EE \ar[r]^-{A^\times} \ar[d]_I & \EE\!/A \ar[d]^{I_A} && \EE \ar[r]^-t \ar[d]_{I} & \EE_L \ar[d]^{I_L}\\
\FF \ar[r]_-{IA^\times} & \FF\!/IA \pushoutcorner && \FF \ar[r]_-{It} & \pushoutcorner \FF_{IL}.\\
}$$
\end{lemma}

\begin{proof}
First consider the pushout $\PP=\FF\overplus{\EE}\EE\!/A$. By the universal property of the pushout, a model of $\PP$ is a pair $\<N,\<M,a\>\>$ where (i) $N$ is an $\FF$-model and (ii) $\<M,a_0\>$ is an $\EE$-model together with an element $a_0\in A^M$ such that (iii) $M=I^*N$. Since $IA^N=A^{I^*N}=A^M$, this is exactly the same as a pair $\<N,a_1\>$ where $N$ is an $\FF$-model and $a_1\in IA^N$. Such pairs $\<N,a_1\>$ are classified by the slice category $\FF\!/IA$ which therefore has the appropriate universal property for the pushout $\FF\!/IA\simeq \FF\overplus{\EE}\EE\!/A$.

Now suppose that $\EE_L\simeq\colim_{l\in L} \EE\!/A_l$ is a localization of $\EE$. Using $I$, we define a new localization $\FF_{IL}\simeq\colim_{l\in L} \FF\!/IA_l$. The universal property of the colimit $\EE_L$ induces a map $I_L:\EE_L\to\FF_{IL}$ as in the diagram below. In order to see that the outer square is a pushout we apply the previous paragraph together with commutation of colimits:

\begin{tabular}{cc}
\mbox{$\begin{array}{rcl}
\FF_{IL} & \simeq & \colim_L \FF\!/IA_l\\
&\simeq& \colim_L \big(\FF\overplus{\EE} \EE\!/A_l\big)\\
&\simeq& \FF \overplus{\EE} \big(\colim_L \EE\!/A_l\big)\\
&\simeq& \FF\overplus{\EE} \EE_L.\\
\end{array}$}
&
\raisebox{1.5cm}{$\xymatrix@=4ex{
\EE \ar[rrrr]^{t} \ar[rrd]_{A_l^\times} \ar[ddd]_I &&&& E_L \ar@{-->}[ddd]^{I_L}\\
&& \EE\!/A_l \ar[urr]_{t_l} \ar[d]^{I_{A_l}} && \\
&& \FF\!/IA_l \ar[rrd]^{It_l} &&\\
\FF \ar[rrrr]_{It} \ar[urr]^{IA_l^\times} &&&& \FF_{IL}\\
}$}
\end{tabular}

\end{proof}

\begin{lemma}\label{cons_loc}
If $I_L:\EE_L\to\FF_{IL}$ is the pushout of a conservative functor $I:\EE\to\FF$ along a localization $t:\EE\to\EE_L$, then $I_L$ is again conservative.
\end{lemma}
\begin{proof}
First notice that for any object $A\in\EE$, the pushout $I_A:\EE\!/A\to\FF\!/IA$ is again conservative. By lemma \ref{eq_cons} it is enough to check that $I_A$ is injective on subobjects. Since a subobject in the slice category $\EE\!/A$ is just an object $E\to A$ together with a subobject $R\leq E$ in $\EE$, we obviously have 
$$\begin{array}{ccc}
\rm{In }\FF && \rm{In }\FF\!/IA\\\hline
IR \lneq IE &\Iff & \raisebox{1.5ex}{$\xymatrix@=1ex{ **[r] IR \ar@{}[rr]|{\textstyle\lneq} \ar[rd] && **[l] IE \ar[ld] \\&IA}$}\\
\end{array}$$
Thus $I$ is conservative if and only if $I_A$ is.

Now suppose that we have a map $s:R\to E$ in $\EE_L$ which has a representative $s_l:R_l\to E_l$ in one of the slice categories $\EE\!/A_l$. As discussed in lemma \ref{localization_limits}, $s$ is monic if and only if $s_l$ is ``eventually monic'': there is a map $f:A_{l'}\to A_l$ in the localization $L$ such that $f^*s_l:f^*R_l\mono f^*E_l$ is monic in $\EE\!/A_{l'}$. In just the same sense, $s$ is an isomorphism just in case $s_l$ is eventually an iso in some further slice category $\EE\!/A_{l''}$.

This means that $R\lneq E$ is a proper subobject in $\EE_L$ just in case, for every map $f:A_{l'}\to A_l$ in $L$, the pullback $f^*R_l\lneq f^*E_l$ is again proper. By the previous observation each of the maps $I_{A_{l'}}$ is conservative, so the images $If^*(R_l)\lneq If^*(E_l)$ in $\FF\!/IA_{l'}$ are also proper. These are representatives of the image of $R$ under $I_L$ and, since these representatives are not eventually iso, $I_L(R)\lneq I_L(E)$ is a proper subobject in $\FF_{IL}$. Thus $I_L$ is injective on subobjects and therefore conservative.
\end{proof}

\begin{cor}
Given a scheme morphism $J:\YY\to\XX$ and a point $x\in\XX$, $\YY$ is consistent near $x$ if and only if there is a point $y\in \YY$ such that $x\in\overline{Jy}$.
\end{cor}
\begin{proof}
Without loss of generality we may replace $\XX$ by an affine neighborhood $x\in\UU\simeq\Spec(\EE)$; in particular, we may think of $x$ as a labelled $\EE$-model $M_x$. The blowup of $\YY$ at $x$ is a pullback in $\LSch$ and, by theorem \ref{scheme_limits}, this is computed (locally) as a pushout in $\Ptop$:
$$\xymatrix{
\Bl_x(\YY) \pbcorner \ar[rr] \ar[d] && **[l] \bigoplus_i \Spec(\FF_i) \simeq\YY \ar[d]^J & \raisebox{-3ex}{$\FF_i\overplus{\EE}\Diag(x)$} \pbcorner & \FF_i \ar[l] \\
\XX_x \ar[rr]_{\pi_x} && \XX & \Diag(x) \ar[u] & \EE \ar[l]_-{t} \ar[u]_{\Gamma J_i} \\
}$$

In particular, $\Bl_x(\YY)$ is non-empty just in case one of these pushouts is a consistent theory. A model for one of these pushouts consists of a model $N\models \FF_i$ together with a homomorphism $h:M_x\to (\Gamma J_i)^*N$. Let $\nu$ denote a labelling on $N$ such that, whenever $k$ is defined at $\mu$, $h(\mu(k))=\nu(k)$. This labelled model $\nu$ corresponds to a point $y\in\YY$, and $J(y)=(\Gamma J_i)^*\nu$. According to proposition \ref{spectral_closure}, $\mu$ belongs to the closure of $(\Gamma J_i)^*\nu$, so $x\in\overline{Jy}$ as asserted.
\end{proof}

\subsection{Proof of (a)}

Na\"ively we expect that the semantic property which is dual to conservativity should bear a family resemblence to essential surjectivity. However, the following example shows that conservative interpretations do not, in general, induce essentially surjective reducts.

Consider the (single-sorted) theories of countably-many and continuum-many distinct constants, respectively:
$$\EE=\Ptop\{a_i=a_j\vdash\bot\ |\ i\not=j\in\omega\}.$$
$$\FF=\Ptop\{a_i=a_j\vdash\bot|i\not=j\in2^\omega\}.$$

Any countable subset of $S\subset 2^\omega$ induces an interpretation $I_S:\EE\to\FF$. $I_S^*$ sends each $\FF$-model $N$ to its reduct $N\reduct S$; this is the forgetful functor which omits constants outside of $S$. It is easy to check that $I_S$ is conservative, but it cannot possibly be essentially surjective: there are no countable models of $\FF$. However, the following weaker property holds:

\begin{defn}
Given an interpretation $I:\EE\to\FF$ we call the reduct functor $I^*:\Mod(\FF)\to\Mod(\EE)$ \emph{supercovering} if for any $\EE$-model $M$, any element $a\in A^M$ and any $R\mono A$ with $a\not\in R^M$, there exists a $\FF$-model $N$ and a homomorphism $h:M\to I^*N$ such that $h(a)\not\in R^{I^*N}$.
\end{defn}

As noted above (cf. proposition \ref{spectral_closure}), one labelled model belongs to the closure of another, $\mu\in\overline{\mu'}\subseteq\Spec(\EE)$, just in case every label defined at $\mu$ is also defined at $\mu'$ and the inclusion $K_\mu\subseteq K_{\mu'}$ induces a model homomorphism $M_\mu\to M_\mu'$. Using this we can tranlate semantic conditions about homomorphisms into spectral conditions on the topological groupoid $\Spec(\EE)$.

\begin{defn}
We say that a map of logical schemes $J:\YY\to\XX$ is \emph{superdense} if for any open subgroupoid $\UU\subseteq\Spec(\EE)$ and any full and open subgroupoid $\VV\subseteq\UU$, whenever $\mu\in\UU-\VV$ there exists $\nu\in I_\flat^{-1}(\UU)-I_\flat^{-1}(\VV)$ such that $\mu\in\overline{I_\flat\mu}$.
\end{defn}

\begin{prop}\label{a_prop}
The following are equivalent:
\begin{enumerate}
\item[(i)] $I:\EE\to\FF$ is conservative.
\item[(ii)] $I^*$ is supercovering.
\item[(iii)] $I_\flat$ is superdense.
\end{enumerate}
\end{prop}

\begin{proof}
We proceed by showing that (i)$\To$(ii)$\To$(iii)$\To$(i).

\noindent\textbf{(i)$\To$(ii).} Consider the following pushout; 
$$\xymatrix{
\EE \ar[d]_I \ar[r] & \EE\!/A \ar[d]_{I_A} \ar[r]^{\tilde{a}} & \Diag(M) \ar[d]^{I_M}\\
\FF \ar[r] & \FF\!/IA \ar[r]_{\widetilde{\cc_a}} & \pushoutcorner \FF_M\\
}$$
We can describe $\FF_M$ syntactically as follows: we extend $\FF$ by a constant $\cc_a:IA$ for each element $a\in A^M$ and an axiom $\vdash I\varphi(\cc_a)$ whenever $M\models \varphi(a)$. This is the theory of an $\FF$-model $N$ together with an $\EE$-model homomorphism $h:M\to I^*N$.

Now fix an element $a\in A^M$ such that $a\not\in R^M$. In order to show that $I^*$ is supercovering we must find an $\FF$-model $N$ and a homomorphism $h:M\to I^*N$ such that $h(a)\not\in R^{I^*N}$. Thinking of the pair $H=\<N,h\>$ as a model for $\FF_M$, $h(a)=\cc_a^H$, so it will be enough to show that $\FF_M$ may be consistently extended by the axiom $IR(\cc_a)\vdash\bot$ or, equivalently, that $IR(\cc_a)$ is not derivable in $\FF_M$.

By assumption $a$ does not belong to $R^M$, so we know that $\Diag(M)\not\vdash R(\cc_a)$. According to lemma \ref{cons_loc}, the pushout of a conservative functor along a localization is again conservative. Since $IR(\cc_a)=I_M(R(\cc_a))$, this tells us that $\FF_M\not\vdash IR(\cc_a)$.

\noindent\textbf{(ii)$\To$(iii).} Now suppose that $I^*$ is supercovering, that $\UU$ is an open subgroupoid of $\Spec(\EE)$ with a full and open subgroupoid $\VV\subseteq\UU$, and that $\mu\in\UU-\VV$. We must find $\nu\in I_\flat^{-1}(\UU)-I_\flat^{-1}(\VV)$ such that $\mu\in\overline{I_\flat\mu}$.

According to lemma \ref{slice_subsch}, we may replace $\UU$ by a smaller affine subgroupoid $\UU_k\simeq\Spec(\EE\!/A)$ for some object $A\in\EE$. By lemma \ref{slice_subsch}, a full and open subgroupoid $\VV\subseteq\UU_k$ must have the form $\bigcup_i \VV_{R_i(k)}$ for some set of subobjects $R_i\leq A$. To say that $\mu\in\UU_k-\VV$ means that the underlying model $M_\mu$ contains an element $a=\mu(k)\in A^\mu$ such that $a\not\in R_i^\mu$ for every $i$.

Let $R^*=R_{i_1}\vee\ldots\vee R_{i_n}$ denote an arbitrary finite join of these subobjects. Since $I^*$ is supercovering, we may find an $\FF$-model $N$ with a homomorphism $h:M\to I^*N$ such that $N\not\models IR^*(h(a))$. This tells us that the following pushout theory $\PP^*$ is consistent; by compactness, the aggregate theory $\PP_{\VV}$ is also consistent:
$$\xymatrix{
\EE\!/A \ar[r] \ar[d] & \FF\!/IA \ar[r] \ar[d] & \FF\!/\neg IR^* \ar[r] \ar[d]& \ldots \ar[r] & \FF\!/\{\neg IR_i\} \ar[d]\\
\Diag(\mu) \ar[r] & \FF_\mu \ar[r] & \PP^* \ar[r] & \ldots \ar[r] & \PP_{\VV}\\
}$$

Therefore, we have an $\FF$-model $N$ and a homomorphism $h:M\to I^*N$ such that $N\not\models IR_i(h(a))$ for every index $i$. By the downward Lowenheim-Skolem theorem, we may assume that $N$ is $\kappa$-small; let $\nu$ denote a labelling of $N$ such that, for each label $l$ defined at $\mu$, $h(\mu(l))=\nu(l)$. In particular, $\nu(k)=h(a)\not\in R_i^\nu$ for any index $i$, so $\nu\not\in\VV$. This gives us a labelled model $\nu\in I_\flat^{-1}(\UU)-I_\flat^{-1}(\VV)$ which, according to proposition \ref{spectral_closure}, belongs to the closure $\overline{I_\flat\nu}$. Therefore, $I_\flat$ is superdense.

\noindent\textbf{(iii)$\To$(i).} We suppose that $I_\flat$ is superdense and that $R\lneq A$ is a proper subobject in $\EE$. We must show that $IR$ is proper in $\FF$, so that $I$ is a conservative functor. Let $k$ denote a parameter of type $A$; set $\UU= \VV_k$ and $\VV=\VV_{R(k)}$ (cf. section \ref{sec_subsch}). Since $R$ is proper, we can find a labelled model $\mu\in\UU$ such that $\mu(k)\not\in R^\nu$ (and hence $\mu\not\in\VV_{R(k)}$). Because $I_\flat$ is superdense we can find a labelled $\FF$-model $\nu$ such that $\nu\not\in I_\flat^{-1}(\VV)$.

But $I_\flat^{-1}(\VV_{R(k)})=\VV_{IR(k)}$, so $\nu(k)\not\in IR^\nu$. The element $\nu(k)$ is a witness demonstrating that $IR\lneq IA$, so $I$ is conservative.
\end{proof}

\subsection{Proof of (b)}

Suppose that $A\in\EE$ and $S\leq IA\in\FF$, so that for any $\FF$-model $N$,
$$S^N\subseteq IA^N=A^{I^*N}.$$
Therefore, if we have $N_0,N_1\in\Mod(\FF)$ and an $\EE$-homomorphism $h:I^*N_0\to I^*N_1$ we can compare $S^{N_1}$ to the image $h_A(S^{N_0})$.

\begin{defn}
Given $I:\EE\to\FF$ we say that $I^*$ \emph{stabilizes} a subobject $S\leq IA$ if
$$h_A\left(S^{N_0}\right)\subseteq S^{N^1}\subseteq A^{I^*N_1}$$
for any $h:I^*N_0\to I^*N_1$. $I^*$ \emph{stabilizes (all) subobjects} if for any $A\in\EE$ and $S\leq IA\in\FF$, $I^*$ stabilizes $S$.
\end{defn}

As above, we can then translate this into a condition on the spectral topology:

\begin{defn}\label{def_separates}
Given a map of logical schemes $J:\YY\to\XX$ we say that $J$ \emph{separates (full \& open) subgroupoids} if for any open $\UU\subseteq\XX$, any full and open $\VV\subseteq J^{-1}(\UU)$ and for any points $y_0\in\VV$ and $y_1\in J^{-1}(\UU)$
$$y_0\in\overline{y_1} \To y_1\in\VV.$$
\end{defn}

\begin{prop}\label{b_prop}
The following are equivalent:
\begin{enumerate}
\item[(i)] $I:\EE\to\FF$ is full on subobjects.
\item[(ii)] $I^*$ stabilizes subobjects.
\item[(iii)] $I_\flat$ separates subgroupoids.
\end{enumerate}
\end{prop}

\begin{proof}
We proceed by showing that (i)$\To$(iii)$\To$(ii)$\To$(i).

\noindent\textbf{(i)$\To$(iii).} Suppose that, as in definition \ref{def_separates} above, we have an open subgroupoid $\UU\subseteq\Spec(\EE)$, a full and open subgroupoid $\VV\subseteq I_\flat^{-1}(\UU)$ and two labelled models $\nu_0\in\VV$ and $\nu_1\in I_\flat^{-1}(\UU)$. We must show that $\nu_1\in\VV$ as well.

As in the previous proposition we can replace $\UU$ by a smaller affine neighborhood $\UU_k\simeq\Spec(\EE\!/A)$ for some object $A\in\EE$. The inverse image $I_\flat^{-1}(\UU_k)$ is equivalent to $\Spec(\FF\!/IA)$ and a full and open subgroupoid of this must have the form $\VV=\bigcup_i \VV_{S_i(k)}$ for some collection of subobjects $S\leq IA$.

This means that for some index $i$, $\nu_0(k)\in S_i^{\nu_0}$. By assumption $I$ is full on subobjects, so $S_i\cong IR_i$ and $I_\flat(\nu_0)\in \VV_{R_i(k)}$. Since $I_\flat(\nu_0)$ belongs to the closure $\overline{I_\flat(\nu_1)}$ it follows that $I_\flat(\nu_1)\in\VV_{R_i(k)}$ as well. But then
$$\nu_1\in I_\flat^{-1}(\VV_{R_i(k)})=\VV_{S_i(k)}\subseteq \VV.$$
This demonstrates that $I_\flat$ separates subgroupoids.

\noindent\textbf{(iii)$\To$(ii).} Fix an arbitrary subobject $S\leq IA$ and and suppose that $N_0$ and $N_1$ are $\FF$-models, $h:I^*N_0\to I_*N_1$ and that $a\in S^{N_0}$. We must show that $h(a)\in S^{N_1}$.

By the downward Lowenheim-Skolem theorem, we may find $\kappa$-small submodels $N_0'\subseteq N_0$ and $N_1'\subseteq N_1$ such that $a\in N_0'$ and $h\upharpoonright I^*N_0'$ factors through $I^*N_1'$:
$$\xymatrix@R=3ex{
I^*N_0 \ar[r]^h & I^*N_1\\
I^*N_0' \ar@{}[u]|{\vertsub} \ar@{-->}[r] & I^*N_1' \ar@{}[u]|{\vertsub}\\
}$$

Now extend these submodels to labelled models $\nu_0$ and $\nu_1$ in such a way that $h(\nu_0(k))=\nu_1(k)$ whenever $k$ is defined at $I_\flat(\nu_0)$. This ensures that $I_\flat(\nu_0)\in\overline{I_\flat(\nu_1)}$. Now $\nu_1$ belongs to the open subgroupoid $I_\flat^{-1}(\UU_k)$ and $\nu_0$ belongs to the full and open subgroupoid $\VV_{S(k)}$ (since $\nu_0(k)=a\in S^{\nu_0}$.

Now we may apply the assumption that $I_\flat$ separates subgroupoids to conclude that $\nu_1\in \VV_{S(k)}$. Therefore $$\nu_1(k)=h(\nu_0(k))=h(a)\in S^{\nu_1}$$
or, in other words, $h(a)\in S^{N_1'}\subseteq S^{N_1}$. Thus any $\EE$-model homomorphism $h:I^*N_0\to I^*N_1$ between reducts stabilizes an arbitrary subobject $S\leq IA$, so $I^*$ stabilizes all subobjects.

\noindent\textbf{(ii)$\To$(i).} Fix an arbitrary subobject $S\leq IA$ and let
$$\Gamma=\{R\leq A\ |\ IR\leq S\}.$$
We will show that for every $\FF$-model $N$ and every element $a\in S^N$ there is a some formula $R\in\Gamma$ such that $a\in IR^N$. This means that we can represent $S$ as a join\footnote{To be more precise, as $\FF$ does not have infinite joins, we should think of this as a statement about representable functors in $\Sh(\FF)$.}
$$S=\bigvee_{R\in\Gamma} IR.$$
Since $S$ is compact, this reduces to a finite join, whence $S\cong I(R_1\vee\ldots\vee R_n)$ belongs to the essential image of $I$.

We are left to show that ever $S$-element satisfies some formula in $\Gamma$, so fix an $\FF$-model $N$ and an element $a\in S^{N}$. Consider the pushout
$$\xymatrix{
\EE \ar[r]^{A^\times} \ar[d]^{\2}&\EE\!/A \ar[r]^-{a} \ar[d] & \Diag(I^*N) \ar[d] \\
\FF \ar[r]_{IA^\times} & \FF\!/IA \ar[r]_-{\cc_a} & \FF_{I^*N}\\
}$$
This is the theory of pairs $H=\<N',h\>$ where $N'$ is another $\FF$-model and $h$ is an $\EE$-model homomorphism $I^*N\to I^*N'$. By assumption $I^*$ stabilizes $S$, so we must have $h(a)=\cc_a^H\in S^{N'}$. It follows by completeness, that is provable: $\FF_{I^*N}\vdash S(\cc_a)$.

We have the following syntactic description of $\FF_{I^*N}$: we extend $\FF$ by a constant $\cc_b:IB$ for each element $b\in IB^N$ and an axiom $\vdash I\varphi(\cc_b)$ whenever $M\models \varphi(b)$. Given this description, compactness implies that the derivation $\FF_{I^*N}\vdash S(\cc_a)$ can only involve finitely many axioms $\{I\varphi_i(\cc_a,\cc_{b_i})\}$ coming from $\Diag(I^*N)$.\footnote{Without loss of generality we have weakened each of these axioms to include the constant $\cc_a$.} We might say that, aside from these axioms, the remainder of the derivation takes place in $\FF$.

Formally, this means that after extending $\FF$ by the constants $\cc_a$ and $\cc_{b_i}$ the following sequent is provable 
$$\bigwedge_i I\varphi_i(\cc_a,\cc_{b_i}) \vdash S(\cc_a).$$
This is a (closed) sequent in the slice category $\FF\!/(IA\times IB)$. Moreover, the constants $\cc_{b_i}$ only appear on the left of the turnstile, so we are free to replace them by an existential quantifier.

If we let $y=\<y_i\>$ and $\epsilon(x)=\exists y. \bigwedge_i \varphi_i(x,y_i)$ this gives us the following deduction:
$$\begin{array}{cc}
\vdash S(\cc_a) & \rm{in }\FF_{I^*N}\\\cline{1-1}
\exists b_i\in B_i^{I^*N}\rm{ and } \varphi_i\leq A\times B_i & \\
\bigwedge_i I\varphi_i(\cc_a,\cc_{b_i})\vdash S(\cc_a) & \rm{in }\FF\!/(IA\times IB)\\\cline{1-1}
\exists \epsilon\leq IA\\
I\epsilon(\cc_a)\vdash S(\cc_a) & \rm{in }\FF\!/A\\\cline{1-1}
\exists \epsilon\leq IA\\
I\epsilon(x)\underset{x:A}{\vdash} S(x)& \rm{in }\FF.
\end{array}$$

On one hand, the last sequent tells us that $\epsilon\in\Gamma$. On the other hand, $b=\<b_i\>$ gives us an existential witness to the fact that $N\models I\epsilon(a)$. This shows that every $S$-element satisfies some formula in $\Gamma$, completing the proof.
\end{proof}

\subsection{Proof of (c)}

Recall from definition \ref{ptop_morph_props} that a pretopos functor $I:\EE\to\FF$ is \emph{subcovering} if for every object $B\in\FF$ there is an object $A\in\EE$, a subobject $S\leq IA$ and an epimorphism (necessarily regular) $\sigma:S\epi B$. To say that $I$ is subcovering means, roughly, that $\EE$ and $\FF$ have the same basic sorts; more precisely, the basic sorts of $\FF$ are definable (in $\FF$) from the images of the basic sorts of $\EE$.

Classically, an $\FF$-model homomorphism $f:N\to N'$ consists of a family of functions $f_B:B^N\to B^{N'}$ ranging over the \emph{basic} sorts of $\FF$ which preserve the basic functions and relations. In particular, two homomorphisms are equal just in case they agree on the basic sorts. From this it is obvious that when $\EE$ and $\FF$ share the same basic sorts, the reduct functor $I^*:\Mod(\FF)\to\Mod(\EE)$ must be faithful. The same holds when $I$ is subcovering.

Below we will show that the converse is also true: if $I^*$ is faithful then $I$ must be subcovering. We must also translate this into a spectral condition. We have already seen that statements about model homomorphisms can be recast into statements about the closure of points: an inclusion $\mu_0\in\overline{\mu_1}$ induces a canonical model homomorphism between their underlying models $M_0\to M_1$. However, it is not immediately obvious how to talk about parallel homomorphisms $M_0 \rightrightarrows M_1$.

The trick is to replace $M_0$ by an isomorphism $M_0\cong M_0'$ and the parallel arrows by a \emph{non}commuting diagram
$$\xymatrix{
M_0 \ar[d] \ar@{=}[r]^{\sim} \ar@{}[rd]|{\displaystyle\times} & M_0' \ar[d]\\
M_1 \ar@{=}[r] & M_1\\
}$$
According to proposition \ref{iso_closure} this can be rephrased in terms of isomorphisms in the spectral groupoid.

% \begin{defn}\label{def_nonfold}
% Given a scheme $\YY$, we say that two isomorphic points $y_0\stackrel{\alpha}{\cong} y_0'$ are \emph{essentially different} relative to another point $y_1$ if $y_0,y_0'\in\overline{y_1}$ and $\alpha\not\in\overline{1_{y_1}}$.\footnote{Technically this is an abuse of terminology. The definition of essential difference depends on $\alpha$ as well as $y_0$ and $y_0'$, so we should say something like ``$\alpha$ expresses an essential difference relative to $y_1$,'' but in practice we will stick with the simpler phrase.}
% 
% A map of schemes $J:\YY\to\XX$ is \emph{non-folding} if it does not identify essentially different points: if there exists a point $y_1$ such that $\alpha:y_0\cong y_1$ are essentially different relative to $y_1$, then $J(y_0)\not=J(y_0')$.
% \end{defn}

\begin{defn}\label{def_nonfold}
Suppose that $J:\YY\to\XX$ is a morphism of schemes. We say that $J$ is \emph{non-folding} if for any two isomorphisms $\alpha:\nu_0\cong\nu_0'$ and $\beta:\nu_1\cong\nu_1'$ such that $\nu_0\in\overline{\nu_1}$ and $\nu_0'\in\overline{\nu_1'}$
$$J(\alpha)\in\overline{J(\beta)}\Iff \alpha\in\overline{\beta} .$$
\end{defn}
Of course continuous functions preserve closure, so $\alpha\in\overline{\beta}$ automatically implies that $J(\alpha)\in\overline{J(\beta)}$. To see that a map is non-folding it suffices to check the right-to-left direction.

\begin{prop}\label{c_prop}
The following are equivalent:
\begin{enumerate}
\item[(i)] $I:\EE\to\FF$ is subcovering.
\item[(ii)] $I^*$ faithful.
\item[(iii)] $I_\flat$ is non-folding.
\end{enumerate}
\end{prop}

\begin{proof}
We proceed by showing that (i)$\To$(iii)$\To$(ii)$\To$(i).

\noindent\textbf{(i)$\To$(iii).} Suppose that $\alpha:\nu_0\cong\nu_0'$ and $\beta:\nu_1\cong\nu_1'$ are labelled $\FF$-models as in definition \ref{def_nonfold}. Let $h,h'$ denote the $\FF$-model homomorphisms induced by the inclusions $\nu_0\in\overline{\nu_1}$ and $\nu_0'\in\overline{\nu_1'}$. According to proposition \ref{iso_closure}, $\alpha\in\overline{\beta}$ if and only if the following diagram commutes, where $N_i$ (resp. $N_i'$) is the underlying model of $\nu_i$ (resp. $\nu_i'$)
$$\xymatrix{
N_1 \ar[rr]^\beta && N_1'\\
N_0 \ar[u]^{h} \ar[rr]_\alpha && N_0' \ar[u]_{h'}.\\
}$$

By assumption $I$ is subcovering so, for each object $B\in\FF$ there is an object $A\in\EE$ and a subquotient $IA\geq S \stackrel{q}{\epi}$. Since $q$ is an epimorphism we have
$$\big(\beta_B\circ h_B=h'_B\circ\alpha_B\big)\Iff\big(\beta_B\circ h_B\circ q^{N_0}=h'_B\circ\alpha_B\circ q^{N_0}\big).$$
As displayed in the following diagram, the naturality of $h$, $h'$, $\alpha$ and $\beta$ ensures that this is equivalent to the equation $q^{N_1'}\circ \beta_S\circ h_S=q^{N_1'}\circ h'_S\circ \alpha_S$:
$$\xymatrix{
S^{N_1} \ar@{->>}[rd]_{q^{N_1}} \ar[rrr]^{\beta_S} &&&S^{N_1} \ar@{->>}[ld]^{q^{N_1'}}\\
& B^{N_1} \ar[r]^{\beta_B} &B^{N_1'}& \\
& B^{N_0} \ar[r]_{\alpha_B} \ar[u]^{h_B} & B^{N_0'} \ar[u]_{h'_B} &\\
S^{N_0} \ar[uuu]^{h_S} \ar@{->>}[ur]^{q^{N_0}} \ar[rrr]_{\alpha_S} &&& S^{N_0'} \ar[uuu]_{h'_S} \ar@{->>}[ul]_{q^{N_0'}}\\
}$$
Clearly this can hold just in case $\beta_S\circ h_S=\circ h'_S\circ \alpha_S$.

Now suppose that $I_\flat(\alpha)\in\overline{I_\flat(\beta)}$. This tells us that the following diagram of $\EE$-models commutes:
$$\xymatrix{
I^*N_1 \ar[rr]^{I^*\beta} && I^*N_1'\\
I^*N_0 \ar[rr]_{I^*\alpha} \ar[u]^{I^*h} && I^*N_0' \ar[u]_{I^*h'}\\
}$$
In particular, $I^*\beta_A\circ I^*h_A=I^*h'_A\circ I^*\alpha_A$.

Since $S$ is a subobject of $IA$ and $IA^{N_0}=A^{I^*N_0}$, it follows that $\beta_S\circ h_S=\circ h'_S\circ \alpha_S$. Hence $\alpha$ belongs to the closure of $\beta$ and $I_\flat$ is non-folding.

\noindent\textbf{(iii)$\To$(ii).} We must show that when $I_\flat$ is non-folding the reduct functor $I^*:\Mod(\FF)\to\Mod(\EE)$ is faithful, so suppose that we have distinct $\FF$-model homomorphisms $g',h':N_0'\rightrightarrows N_1'$. By the downward Lowenheim-Skolem theorem we can find a $\kappa$-small submodel $N_0\subseteq N_0'$ such that $g'\upharpoonright N_0\not= h'\upharpoonright N_0$. Similarly, we can find a $\kappa$-small submodel $N_1\subseteq N_1'$ such that both of these restrictions factor through $N_1$; call the resulting factorizations $g$ and $h$:
$$\xymatrix@R=3ex{
N_0' \ar@<1ex>[rr]^(.6){g'} \ar@<-1ex>[rr]_(.6){h'} && N_1'\\
N_0 \ar@{}[u]|{\vertsub} \ar@<1ex>[rr]^(.4){g} \ar@<-1ex>[rr]_(.4){h} && N_1' \ar@{}[u]|{\vertsub}\\
}$$

Fix a labelled model $\nu_1$ whose underlying model is $N_1$. In addition, choose two labellings $\nu_0$ and $\nu_0'$ on $N_0$ such that whenever $k$ or $l$ is defined we have $g(\nu_0(k))=\nu_1(k)$ or $h(\nu_0'(l))=\nu_1(l)$. This ensures that $\nu_0,\nu_0'\in\overline{\nu_1}$. Moreover, the identity on $N_0$ induces an isomorphism of labelled models $\alpha:\nu_0\cong\nu_0'$; the fact that $g\not=h$ tells us that $\alpha$ does not belong to the closure of the identity on $\nu_1$.

Now we can apply the assumption that $I_\flat$ is non-folding; this ensures that $I_\flat(\alpha)\not\in\overline{I_\flat(1_{\nu_1})}$. The inclusion $I_\flat(\nu_0)\in\overline{I_\flat\nu_1}$ is induced by the reduct morphism $I^*g:I^*N_0\to I^*N_1$ and similarly for $\nu_0'$, while the underlying map of $I_\flat(\alpha)$ is the identity on $I^*N_0$, so this tells us that the following diagram does not commute in $\Mod(\EE)$
$$\xymatrix{
N_1 \ar@{=}[rr]^{1_{\nu_1}} & \ar@{}[d]|{\textstyle\times} & N_1\\
N_0 \ar[u]^{I^*g} \ar@{=}[rr]_{\alpha} && N_0 \ar[u]_{I^*h}\\
}$$
Since $g$ and $h$ are restrictions of $g'$ and $h'$ it follows that $I^*g'\not=I^*h'$, so $I^*$ must be faithful.

\noindent\textbf{(ii)$\To$(i).} The proof is formally similar to the last argument in part (b). First fix an object $B\in\FF$. Given an object $A\in\EE$, a partial function $IA\part B$ is a two-place relation $\sigma\leq IA\times B$ which is provably many-one in $\FF$:
$$\sigma(x,y)\wedge\sigma(x,y')\underset{\underset{x:IA,}{y,y':B}}{\vdash} y=y'.$$
Categorically speaking, $\varphi$ is a partial function just in case the composite $\varphi\mono IA\times B\to IA$ is monic.

Now let
$$\Sigma=\{\sigma\leq IA\times B\ |\ A\in\EE, \sigma:IA\part B\}.$$
We will show that for every $\FF$-model $N$ and every element $b\in B^N$ there is some partial function $\sigma\in\Sigma$ and an element $a\in IA^N$ such that $N\models \sigma(a,b)$. It will follow that we can express $B$ as a join
$$B(y)\cong \bigvee_{\sigma\in\Sigma} \exists x.\sigma(x,y).$$
By compactness we can reduce this to a finite subcover. If we set $S(x)=\exists y.\sigma(x,y)$, this allows us to represent $B$ as a subquotient of $\EE$:
$$\xymatrix@R=3ex{
I(A_1+\ldots A_n)\\
S_1+\ldots+S_n \ar@{->>}[r] \ar@{}[u]|{\vertleq} &B.\\
}$$

It remains to show that every $B$-element is the image of an $IA$ element under some partial function, so fix an $\FF$-model $N$. The first step is to axiomatize the following data: a second $\FF$-model $N'$ together with a pair of homomorphisms $g_1,g_2:N\rightrightarrows N'$ such $I^*g_1=I^*g_2$. This theory can be presented as the following iterated pushout:
$$\xymatrix{
\EE \ar[r] \ar[d]_I & \Diag(I^*N) \ar[d] & \\
\FF \ar[r] & \pushoutcorner \FF_{I^*N} \ar[r] \ar[d] & \Diag^1(N) \ar[d]\\
& \Diag^2(N) \ar[r] & \pushoutcorner \TT\\
}$$

$\TT$ contains two copies of $\Diag(N)$ which we distinguish with superscripts; each element $b\in B^N$ defines two constants $\cc^1_b,\cc^2_b:B$ and every parameterized formula $\varphi(x,b)\in\Diag(N)$ defines two objects $\varphi^1(x,\cc_b^1)$ and $\varphi^2(x,\cc_b^2)$. Whenever $N\models\varphi(b)$ we attach two axioms $\{\vdash \varphi^1(\cc_b^1), \vdash\varphi^2(\cc_b^2)\}$. Given a model $G\models \TT$ we can recover the functions $g_1$ and $g_2$ by setting $g_i(b)=(\cc_b^i)^H$. Finally, we ensure that $I^*g_1=I^*g_2$ by adding an axiom $\vdash \cc_a^1=\cc_a^2$ for every element $a\in IA^N$; to simplify notation we will treat these two as a single constant $\cc_a$.

By assumption $I^*$ is faithful, so $I^*g_1=I^*g_2$ implies that $g_1=g_2$. By completeness this must be provable in $\TT$: for a fixed element $b\in B$, $\TT\vdash \cc_{b}^1=\cc_{b}^2$. This derivation must involve only a finite number of the axioms mentioned in the last paragraph, say
$$\{\vdash \varphi^1_i(\cc_{a_i},\cc^1_{d_i}), \vdash \psi^2_j(\cc_{a_j},\cc^2_{d_j})\}.$$

Now we unify these formulas. First, we may throw in some irrelevant assumptions in order to assume that the two sets of formulas $\{\varphi_i(z_i)\}$ and $\{\psi_j(z_j)\}$ are identical. Next, let $A=\prod_i A_i\times\prod_j A_j$ and $D=\prod_i D_i\times \prod_j D_j$ and weaken all of these formulas to a common context $A\times B\times D$.

The remainder of the derivation can proceed in the slice category $\FF\!/(IA\times B^2\times D^2)$. As in part (b), we first conjoin the formulas and eliminate constants to the left of the turnstile. We then replace the remaining closed sequent in $\FF\!/(IA\times B^2)$ by an open sequent in $\FF$ (where the formulas $\varphi^1=\varphi^2$ are identified, though $\cc_{b}^1$ and $\cc_b^2$ are not). Letting $\epsilon(x,y)=\exists z.\bigwedge_i \varphi_i(x,y,z)$, this leaves us with the following series of equivalent statements

$$\begin{array}{cc}
\vdash \cc_b^1=\cc_b^2 & \rm{in }\TT\\\cline{1-1}
\exists d_i\in D_i^{I^*N}\rm{ and } \varphi_i\leq IA\times B\times D_i&\\
\Big(\bigwedge_i \varphi_i(\cc_a,\cc^1_b,\cc^1_{d_i})\Big)\wedge\Big(\bigwedge_i \varphi_i(\cc_a,\cc^2_b,\cc^2_{d_i})\Big)\vdash \cc_b^1=\cc_b^2
& \rm{in }\FF\!/(IA\times B^2\times D^2)\\\cline{1-1}
\exists \epsilon\leq IA\times B\\
\epsilon(\cc_a,\cc_b^1)\wedge\epsilon(\cc_a,\cc_{b}^2)\vdash \cc_b^1=\cc_b^2& \rm{in }\FF\!/(IA\times B^2)\\\cline{1-1}
\exists \epsilon\leq IA\times B\\
\epsilon(x,y_1)\wedge\epsilon(x,y_2)\underset{\underset{x:A}{y_1,y_2:B}}{\vdash} y_1=y_2& \rm{in }\FF.
\end{array}$$

On one hand, the last sequent tells us that $\epsilon$ defines a partial function and so belongs to $\Sigma$. On the other hand, the element $d=\<d_i\>$ gives us an existential witness to the fact that $N\models I\epsilon(a,b)$. This shows that every $B$-element is the image of an $IA$-element under a definable partial function, completing the proof.
\end{proof}

\subsection{Conceptual Completeness}

With the model-theoretic characterizations of these syntactic properties, we are in a good position to complete the proof of conceptual completeness.

\begin{prop}
Suppose that $I:\EE\to\FF$ is a pretopos functor.
\begin{itemize}
\item If $I$ is conservative and full on subobjects, then $I$ is full.
\item If $I$ is conservative, full on subobjects and subcovering, then $I$ is essentially surjective.
\end{itemize}
\end{prop}

\begin{proof}
For the first claim, suppose that $f:IA\to IB$ is a map in $\FF$. We can represent this as a graph $\Gamma_f\leq I(A\times B)$. Since $I$ is full on subobjects, there is a preimage $R\leq A\times B$ such that $IR\cong \Gamma_f$. To say that $f$ is functional means that the composite $\Gamma_f\mono I(A\times B)\to IA$ is an isomorphism. $I$ is conservative, so it reflects this isomorphism, giving us an isomorphism $R\mono A\times B \to A$. That means that $R=\Gamma_{\overline{f}}$ is a graph in $\EE$, and $I(\overline{f})=f$.

For the second, take any object $B\in \FF$. Since $I$ is subcovering, this can be represented as a subquotient
$$\xymatrix{
IA & \\
S \ar@{}[u]|{\vertleq} \ar@{->>}[r] & B\\
}$$
Since $I$ is full on subobjects, $S\cong IR$ lies in the image of $I$. Now form the kernel pair $K\rightrightarrows IR \epi B$. $K$ is a subobject of $I(R\times R)$, so $K\cong IL$ also belongs to the image of $I$. As discussed in the proof of proposition \ref{qc_ortho}, conservative functors reflect equivalence relations, so the preimage $L\rightrightarrows R$ is an equivalence relation in $\EE$. Since pretopos functors preserve quotients, $B\cong I(R/L)$.
\end{proof}

\begin{thm}
A pretopos functor $I:\EE\to\FF$ is an equivalence of categories if and only if the reduct functor $I^*:\Mod(\FF)\to\Mod(\EE)$ is an equivalence of categories.
\end{thm}
\begin{proof}
This follows immediately from the previous proposition, together with theorem \ref{scheme_cc}. The right-to-left direction is trivial.

As for the converse, suppose $I^*$ is essentially surjective. Then it is certainly supercovering: for any $\EE$-model $M$ there is an isomorphism $h:M\iso I^*N$ and if $a\not\in R^M$ then $h(a)\not\in R^{I^*N}$. Therefore $I$ must be conservative (and hence faithful by lemma \ref{eq_cons}.

If $I^*$ is full then any morphism $h:I^*N_0\to I^*N_1$ has a lift $\overline{h}:N_0\to N_1$ with $I^*\overline{h}=h$. For any $S\leq IA\in\FF$, this means
$$h_A(S^{N_0})=(I^*\overline{h})_A(S^{N_0})=\overline{h}_{IA}(S^{N_0})=\overline{h}_S(S^{N_0})\subseteq S^{N_1}.$$
Thus $I^*$ stabilizes subobjects and $I$ must be full on subobjects.

When $I^*$ is faithful $I$ is subcovering, and we have just seen that a functor which is conservative, full on subobjects and subcovering is an equivalence of categories.
\end{proof}